\pdfoutput=1
\documentclass[11pt,oneside,reqno]{amsart}
\usepackage[top=25mm, bottom=25mm, left=25mm, right=25mm]{geometry}
\usepackage{amsfonts,amssymb,amsmath,amsthm,amsxtra}
\usepackage[T1]{fontenc}
\usepackage[utf8]{inputenc}
\usepackage{bm}
\usepackage{xparse}
\usepackage{mathtools}
\usepackage{dsfont}
\usepackage{bbm}
\usepackage{mathrsfs}
\usepackage{enumitem}
\usepackage{tikz}

\usepackage{graphics}
\usepackage{color}
\usepackage[unicode,hidelinks,bookmarks]{hyperref}
\hypersetup{
colorlinks=true,
citecolor=[rgb]{0,0,0.8},
linkcolor=[rgb]{0.89,0.02,0.17},
urlcolor=[rgb]{0,0,0.75},
pdfpagemode=UseNone,
pdfstartview=FitH,
pdfdisplaydoctitle=true,
pdftitle={On a multi-parameter variant of the Bellow--Furstenberg problem},
pdfauthor={Bourgain, Mirek, Stein, Wright},
pdflang=en-US
}

\numberwithin{equation}{section}
\newtheorem{theorem}[equation]{Theorem}

\newtheorem{lemma}[equation]{Lemma}
\newtheorem{claim}[equation]{Claim}
\newtheorem{proposition}[equation]{Proposition}
\newtheorem{conjecture}[equation]{Conjecture}

\theoremstyle{definition}
\newtheorem{remark}[equation]{Remark}

\newtheorem{example}[equation]{Example}

\newcommand{\BB}{\mathbb{B}}
\newcommand{\CC}{\mathbb{C}}
\newcommand{\DD}{\mathbb{D}}
\newcommand{\EE}{\mathbb{E}}

\newcommand{\GG}{\mathbb{G}}

\newcommand{\II}{\mathbb{I}}
\newcommand{\JJ}{\mathbb{J}}

\newcommand{\NN}{\mathbb{N}}

\newcommand{\PP}{\mathbb{P}}
\newcommand{\QQ}{\mathbb{Q}}
\newcommand{\RR}{\mathbb{R}}
\renewcommand{\SS}{\mathbb{S}}
\newcommand{\TT}{\mathbb{T}}

\newcommand{\ZZ}{\mathbb{Z}}

\newcommand{\calF}{\mathcal{F}}
\newcommand{\calB}{\mathcal{B}}

\newcommand{\ex}{\bm{e}}

\newcommand{\proj}{\mathrm{p}}

\newcommand{\ind}[1]{\mathds{1}_{{#1}}}

\usepackage{scalerel,stackengine}
\stackMath
\newcommand\reallywidehat[1]{%
\savestack{\tmpbox}{\stretchto{%
  \scaleto{%
    \scalerel*[\widthof{\ensuremath{#1}}]{\kern-.7pt\bigwedge\kern-.7pt}%
    {\rule[-\textheight/2]{1.5ex}{\textheight}}
  }{\textheight}%
}{0.9ex}}%
\stackon[2.25pt]{#1}{\tmpbox}%
}

\newcommand*{\DMO}[1]{\expandafter\DeclareMathOperator\csname #1\endcsname {#1}}
\DMO{ord}
\DMO{supp}
\DMO{Sym}
\DMO{cl}
\DMO{lcm}
\DMO{dist}
\DMO{tr}
\DMO{diam}

\DeclarePairedDelimiter\abs{\lvert}{\rvert}

\DeclarePairedDelimiter\norm{\lVert}{\rVert}

\DeclarePairedDelimiterX\spr[2]{\langle}{\rangle}{#1,#2}
\newcommand{\ipr}[2]{#1\cdot#2}
\DeclarePairedDelimiterX\Set[2]{\{}{\}}{#1\colon #2}
\DeclarePairedDelimiterX\Seq[1]{(}{)}{#1}

\begin{document}
\title[On a multi-parameter variant of the Bellow--Furstenberg problem]{On a multi-parameter variant of\\ the Bellow--Furstenberg problem}

\author{Jean Bourgain}
\address{Jean Bourgain \\
  School of Mathematics\\
  Institute for Advanced Study\\
  Princeton, NJ 08540\\
  USA}
\email{bourgain@math.ias.edu}

\author{Mariusz Mirek }
\address[Mariusz Mirek]{
Department of Mathematics,
Rutgers University,
Piscataway, NJ 08854-8019, USA \\
\&  School of Mathematics,
  Institute for Advanced Study,
  Princeton, NJ 08540,
  USA
\&
Instytut Matematyczny,
Uniwersytet Wroc{\l}awski,
Plac Grunwaldzki 2/4,
50-384 Wroc{\l}aw
Poland}
\email{mariusz.mirek@rutgers.edu}

\author{Elias M.\ Stein}
\address[Elias M.\ Stein]{
Department of Mathematics,
Princeton University,
Princeton,
NJ 08544-100 USA}
\email{stein@math.princeton.edu}

\author{James Wright}
\address[James Wright]{
James Clerk Maxwell Building,
The King's Buildings,
Peter Guthrie Tait Road,
City Edinburgh,
EH9 3FD}
\email{J.R.Wright@ed.ac.uk}

\begin{abstract}
We prove convergence in norm and pointwise almost everywhere on $L^p$,
 $p\in (1,\infty)$, for certain multi-parameter polynomial ergodic
averages by establishing the
corresponding multi-parameter maximal and oscillation
inequalities. Our result, in particular, gives an affirmative answer
to a multi-parameter variant of the Bellow--Furstenberg problem.  This
paper is also the first systematic treatment of multi-parameter
oscillation semi-norms which allows an efficient handling of
multi-parameter pointwise convergence problems with arithmetic
features. The methods of proof of our main result develop estimates
for multi-parameter exponential sums, as well as introduce new ideas
from the so-called multi-parameter circle method in the context of the
geometry of backwards Newton diagrams that are dictated by the shape
of the polynomials defining our ergodic averages.
\end{abstract}


\thanks{ Jean Bourgain was supported by NSF grant DMS-1800640.
Mariusz Mirek was partially supported by NSF grant DMS-2154712, and by the National Science Centre
in Poland, grant Opus 2018/31/B/ST1/00204. Elias M. Stein was partially
supported by NSF grant DMS-1265524.}
\maketitle

\setcounter{tocdepth}{1}

\tableofcontents

\section{Introduction}\label{section:1}

\subsection{A brief history} In 1933 Khintchin \cite{Khin} had the
great insight to see how to generalize the classical equidistribution
result of Bohl \cite{Bohl}, Sierpi\'nski \cite{S} and Weyl
\cite{Weyl-10} from 1910 to a pointwise ergodic theorem, observing
that as a consequence of Birkhoff's famous ergodic theorem \cite{BI},
the following equidistribution result holds: namely, for any
irrational $\theta \in {\mathbb R}$, for any Lebesgue measurable set
$E\subseteq [0,1)$, and for almost every $x\in {\mathbb R}$,
\begin{align*}
\lim_{M\to\infty} \frac{\# \{ m\in[M]: \{x + m \theta\} \in E \}}{M}  =  |E|,
\end{align*}
where $\{x\}$ denotes the fractional part of $x\in\RR$, and
$[N]:=(0, N]\cap\ZZ$ for any real number $N\ge1$. In 1916 Weyl
\cite{Weyl} extended the classical equidistribution theorem to general
polynomial sequences $(\{P(n)\})_{n\in\NN}$ having at least one
irrational coefficient, and so it was natural to ask whether a
pointwise ergodic extension of Weyl's equidistribution theorem
holds. This question was posed by Bellow \cite{Bel} and Furstenberg
\cite{F} in the early 1980's; precisely, they asked if for any
polynomial $P \in {\mathbb Z}[{\rm m}]$ with integer coefficients and
$P(0)=0$ and for any invertible measure-preserving transformation
$T: X \to X$ on a probability space $(X,\mathcal B(X), \mu)$, does the
limit
\begin{align*}
\lim_{M\to \infty} \EE_{m\in[M]} f(T^{P(m)} x) 
\end{align*}
exist for almost every $x\in X$ and for every $f \in L^\infty(X)$?
Here and throughout the paper we use the notation
$\EE_{y\in Y}f(y):=\frac{1}{\#Y}\sum_{y\in Y}f(y)$ for any finite set
$Y\neq\emptyset$ and any function $f:Y\to\CC$.  In the mid 1980's, the
first author \cite{B1,B2,B3} established that this is indeed the case
whenever $f \in L^p(X)$ and $p\in (1,\infty)$, leaving open the
question of what happens on $L^1(X)$.  Interestingly it was shown much
later by Buczolich and Mauldin \cite{BM} that the above pointwise
convergence result fails for general $L^1$ functions when
$P(m) = m^2$, see also \cite{LaV1} for further refinements. In any
case, the papers \cite{B1,B2,B3} represent a far-reaching common
generalization of Birkhoff's pointwise ergodic theorem and Weyl's
equidistribution theorem.

Both Birkhoff and Weyl's results have natural multi-parameter extensions.
In 1951, Dunford \cite{D} and Zygmund \cite{Z} independently extended Birkhoff's theorem to multiple measure-preserving transformations
$T_1, \ldots, T_k : X \to X$. They showed that the limit
\begin{align}
\label{eq:11}
\lim_{M_1,\ldots,M_k \to \infty} \EE_{(m_1,\ldots, m_k)\in\prod_{j=1}^k[M_j]}
f(T_1^{m_1}  \cdots  T_k^{m_k} x) 
\end{align}
exists for almost every $x \in X$ and for any $f\in L^p(X)$ with
$p\in(1, \infty)$, where $\prod_{j=1}^k[M_j]:=[M_1]\times\ldots\times[M_k]$. The limit is taken in the unrestricted sense; that
is, when $\min\{M_1,\ldots,M_k\} \to \infty$. Here, when $k\ge 2$, the
pointwise convergence result is manifestly false for general
$f\in L^1(X)$.

In 1979, Arkhipov, Chubarikov and Karatsuba \cite{ACK-fractional}
extended Weyl's equidistribution result to polynomials (even multiple polynomials) of several
variables. In its simplest form, their result asserts that for any $k$-variate
polynomial $P\in\ZZ[\rm m_1,\ldots, \rm m_k]$, any irrational $\theta\in\RR$, and any interval $[a, b)\subseteq [0, 1)$ one
has
\begin{align}
\label{eq:22}
\lim_{\min\{M_1,\ldots,M_k\} \to \infty} \frac{\# \{(m_1, \ldots, m_k)\in\prod_{j=1}^k[M_j]: \{\theta P(m_1,\ldots, m_k)\} \in [a, b) \}}{M_1 \cdots M_k}  =  b-a.
\end{align}

In the late 1980's, after \cite{B1, B2, B3} and in light of
these results, it was natural to seek a common generalization of the
results of Dunford and Zygmund on the one hand (which generalize
Birkhoff's original theorem) and Arkhipov, Chubarikov and Karatsuba on
the other hand (which generalize Weyl's theorem), which can be
subsumed under the following conjecture, a multi-parameter
variant of the Bellow--Furstenberg problem:

\begin{conjecture}
\label{con:0}
Let $k\in\ZZ_+$ with $k\ge2$ be given and let $(X, \mathcal B(X), \mu)$ be a probability measure space with an invertible  measure-preserving transformation $T:X\to X$. Assume that $P\in\ZZ[{\rm m}_1, \ldots, {\rm m}_k]$ with $P(0)=0$. Then for any $f\in L^{\infty}(X)$ the limit
\begin{align}
\label{eq:21}
\lim_{\min\{M_1,\ldots,M_k\} \to \infty}\EE_{(m_1,\ldots, m_k)\in\prod_{j=1}^k[M_j]} 
f(T^{P(m_1, \ldots, m_k)}x) \quad \text{exists for $\mu$-almost every $x\in X$}.
\end{align}
\end{conjecture}

Our main theorem resolves this conjecture.
\begin{theorem}
\label{thm:3}
Conjecture \ref{con:0} is true for all $k\in\ZZ_+$.
\end{theorem}

The case $k=1$ corresponds to the classical one-parameter question
of Bellow \cite{Bel} and Furstenberg \cite{F} and was 
resolved in \cite{B1, B2, B3}. In this paper we will establish the
cases $k\ge2$.  In fact, we will prove stronger quantitative results
including corresponding multi-parameter maximal and oscillation
estimates, see Theorem \ref{thm:main} below, which will imply
Conjecture \ref{con:0}. This paper also represents a first systematic
treatment of multi-parameter oscillation semi-norms which allows an
efficient handling of multi-parameter pointwise convergence problems
for ergodic averaging operators with polynomial orbits. Before we
formulate our main quantitative results, we briefly describe the interesting history of
Conjecture \ref{con:0}.

The theorems of Dunford \cite{D} and Zygmund \cite{Z} have 
simple proofs, which can be deduced by iterative applications of the
classical Birkhoff ergodic theorem. For this purpose, it suffices to
note that the Dunford--Zygmund averages from \eqref{eq:11} can be written as a
composition of $k$ classical Birkhoff averages as follows
\begin{align}
\label{eq:4}
\EE_{(m_1,\ldots, m_k)\in\prod_{j=1}^k[M_j]}
f(T_1^{m_1}  \cdots  T_k^{m_k} x) =
\EE_{m_k\in[M_k]}\big[\cdots \EE_{m_1\in[M_1]}f(T_1^{m_1}(   \cdots  T_k^{m_k}) x)\big].
\end{align}
The order in this composition is important since the transformations $T_1,\ldots, T_k$ do not need to commute. The first author, in view of \cite{B1, B2, B3}, extended the observation from \eqref{eq:4} to polynomial orbits and showed that for every $f\in L^p(X)$ with $p\in(1, \infty)$ the limit 
\begin{align}
\label{eq:26}
\lim_{\min\{M_1,\ldots,M_k\} \to \infty}\EE_{(m_1,\ldots, m_k)\in\prod_{j=1}^k[M_j]} f(T_1^{P_1(m_1)}  \cdots  T_k^{P_k(m_k)} x) 
\end{align}
exists for $\mu$-almost every $x\in X$, whenever
$P_1,\ldots, P_k\in\ZZ[\rm m]$ with $P_1(0)=\ldots=P_k(0)=0$ and
$T_1,\ldots, T_d:X\to X$ is a family of commuting and invertible
measure-preserving transformations. The result from \eqref{eq:26} was
never published, nonetheless it can be thought of as a
polynomial extension of the theorem of Dunford \cite{D} and Zygmund
\cite{Z} (the arguments in Section \ref{deg} can be used to derive a quantitative
version of \eqref{eq:26}). Interestingly, as observed by Benjamin Weiss (privately communicated to the first author), 
any ergodic theorem
for these averages fail in general for $k\ge 2$ when the $T_1,\ldots, T_k$ are general
non-commuting transformations. It may even fail in the one-parameter situation for the averages of the form 
$\EE_{m\in[M]} f(T_1^{P_1(m)}  \cdots  T_k^{P_k(m)} x)$, see also \cite{BL} for interesting counterexamples.

This was a turning point, illustrating that the multi-parameter
theory for averages with orbits along polynomials with separated
variables as in \eqref{eq:26} is well understood and can be readily
deduced  from the one-parameter theory \cite{B1, B2, B3} by simple iteration as in \eqref{eq:4}.
On the other hand, the equidistribution result \eqref{eq:22} of Arkhipov,
Chubarikov and Karatsuba \cite{ACK-fractional}, based on 
the so-called multi-parameter circle method (deep and intricate tools in analytic number theory which go beyond the classical circle method)  showed that the situation
may be dramatically different when orbits are defined along genuinely
$k$-variate polynomials $P\in\ZZ[{\rm m}_1, \ldots, {\rm m}_k]$ and
led to Conjecture \ref{con:0}.  Even for $k=2$ with
$P(m_1, m_2)=m_1^2m_2^3$ in \eqref{eq:21}, the problem becomes very
challenging. Surprisingly it seems that there is no simple way (like
changing variables or interpreting the average from \eqref{eq:21} as a
composition of simpler one-parameter averages as in \eqref{eq:4}) that would help us to reduce the matter to the setup
where pointwise convergence is known.

The multi-parameter case $k\ge2$ in Conjecture \ref{con:0} lies in sharp
contrast to the one-parameter situation $k=1$, causing serious
difficulties that were not apparent in \cite{B1, B2, B3}.  The most
notable differences are multi-parameter estimates of corresponding
exponential sums and a delicate control of error terms that arise in
implementing the circle method. These difficulties arise from the
lack of nestedness when the parameters $M_1,\ldots,M_k$ are independent,
see Figure \ref{fig:1} and Figure \ref{fig:2} below.  We now turn to a more
detailed discussion and precise formulation of the results in this
paper.

\subsection{Statement of the main results}\label{main results} 

Throughout this paper the triple $(X, \mathcal B(X), \mu)$ denotes a
$\sigma$-finite measure space, and $\ZZ[{\rm m}_1, \ldots, {\rm m}_k]$
denotes the space of all formal $k$-variate polynomials
$P({\rm m}_1, \ldots, {\rm m}_k)$ with $k\in\ZZ_+$ indeterminates
${\rm m}_1, \ldots, {\rm m}_k$ and integer coefficients. Each
polynomial $P\in\ZZ[{\rm m}_1, \ldots, {\rm m}_k]$ will always be
identified with a map
$\ZZ^k\ni(m_1,\ldots, m_k)\mapsto P(m_1,\ldots, m_k)\in\ZZ$.

Let $d, k \in\ZZ_+$, and given a family
${\mathcal T} = \{T_1,\ldots, T_d\}$ of invertible commuting
measure-preserving transformations on $X$, a measurable function $f$
on $X$, polynomials
${\mathcal P} = \{P_1,\ldots, P_d\} \subset \ZZ[\mathrm m_1, \ldots, \rm m_k]$, and a
vector of real numbers $M = (M_1,\ldots, M_k)$ whose entries are
greater than $1$, we define the multi-parameter polynomial ergodic
average by
\begin{align}
\label{eq:229}
A_{{M}; X, {\mathcal T}}^{\mathcal P}f(x):= \EE_{m\in Q_{M}}f(T_1^{P_1(m)}\cdots T_d^{P_d(m)}x), \qquad x\in X,
\end{align}
where $Q_{M}:=[M_1]\times\ldots\times[M_k]$ is a rectangle in $\ZZ^k$.
We will often abbreviate $A_{M; X, {\mathcal T}}^{{\mathcal P}}$ to  $A_{M; X}^{{\mathcal P}}$ when the tranformations are understood. In some instances we will write out the averages
$$
A_{M;X}^{\mathcal P}f(x) \ = \  A_{M_1,.\ldots, M_k;X}^{P_1,\ldots, P_d} f(x) \ \ \ {\rm or} \ \ \ 
A_{{M}; X, {\mathcal T}}^{\mathcal P}f(x) \ = \
A_{M_1,\ldots, M_k;X,T_1,\ldots, T_d}^{P_1,\ldots, P_d} f(x),
$$
depending on how explicit we want to be.

\begin{example}\label{ex:1}
From the point of view of pointwise convergence problems, due to the
Calder{\'o}n transference principle \cite{Cald}, the most important
dynamical system is the integer shift system. Consider the
$d$-dimensional lattice $(\ZZ^d, \mathcal B(\ZZ^d), \mu_{\ZZ^d})$
equipped with a family of shifts $S_1,\ldots, S_d:\ZZ^d\to\ZZ^d$,
where $\mathcal B(\ZZ^d)$ denotes the $\sigma$-algebra of all subsets
of $\ZZ^d$, $\mu_{\ZZ^d}$ denotes counting measure on $\ZZ^d$, and
$S_j(x)=x-e_j$ for every $x\in\ZZ^d$ (here $e_j$ is $j$-th basis
vector from the standard basis in $\ZZ^d$ for each $j\in[d]$). The
average $A_{M; X, {\mathcal T}}^{{\mathcal P}}$ with
${\mathcal T} = (T_1,\ldots, T_d)=(S_1,\ldots, S_d)$ can be rewritten
for any $x=(x_1,\ldots, x_d)\in\ZZ^d$ and any finitely supported
function $f:\ZZ^d\to\CC$ as
\begin{align}
\label{eq:69}
A_{M; \ZZ^d}^{{\mathcal P}}f(x)=\EE_{m\in Q_{M}}f(x_1-P_1(m),\ldots, x_d-P_d(m)).
\end{align}
\end{example}

The main result of this paper, which implies Conjecture \ref{con:0},  is the following ergodic theorem.

\begin{theorem}
\label{thm:main}
Let $(X, \mathcal B(X), \mu)$ be a $\sigma$-finite measure space with an  invertible  measure-preserving transformation $T:X\to X$. Let $k\in\ZZ_+$ with $k\ge2$ be given, and  $P\in\ZZ[{\rm m}_1,\ldots,  {\rm m}_k]$ be a polynomial such that
$P(0)=0$. Let $f\in L^p(X)$ for some $1\le  p\le \infty$, and let $A_{M_1,\ldots,  M_k; X, T}^P f$ be the  average defined  in \eqref{eq:229} with $d=1$ and arbitrary $k\in\ZZ_+$.
\begin{itemize}
\item[(i)] \textit{(Mean ergodic theorem)} If $1<p<\infty$, then the averages
$A_{M_1,\ldots,  M_k; X, T}^{P}f$ converge in $L^p(X)$ norm.

\item[(ii)] \textit{(Pointwise ergodic theorem)} If $1<p<\infty$, then the averages
$A_{M_1,\ldots,  M_k; X, T}^{P}f$ converge pointwise almost everywhere.

\item[(iii)] \textit{(Maximal ergodic theorem)}
If $1<p\le\infty$, then one has
\begin{align}
\label{eq:230}
\big\|\sup_{M_1,\ldots, M_k\in\ZZ_+}|A_{M_1,\ldots,  M_k; X, T}^{P}f|\big\|_{L^p(X)}\lesssim_{p, P}\|f\|_{L^p(X)}.
\end{align}

\item[(iv)] \textit{(Oscillation ergodic theorem)}
If $1<p<\infty$ and $\tau>1$, then one has
\begin{align}
\label{eq:284}
\qquad \qquad\sup_{J\in\ZZ_+}\sup_{I\in\mathfrak S_J(\DD_{\tau}^k) }\big\|O_{I, J}(A_{M_1,\ldots,  M_k; X, T}^{P}f: M_1, \ldots, M_k\in\DD_{\tau})\|_{L^p(X)}\lesssim_{p, \tau, P}\|f\|_{L^p(X)},
\end{align}
where $\DD_{\tau}:=\{\tau^n:n\in\NN\}$, see Section \ref{section:2} for a definition of the oscillation semi-norm $O_{I, J}$. The implicit constant in \eqref{eq:230} and \eqref{eq:284} may depend on $p, \tau, P$.
\end{itemize}
\end{theorem}

For ease of exposition, we only prove Theorem \ref{thm:main} in the
two-parameter setting $k=2$, though there are some places in the paper
where some arguments are formulated and proved in the multi-parameter setting to
convince the reader that our arguments are adaptable to the general
multi-parameter setup.  However, the patient reader will readily see
that all two-parameter arguments are adaptable (at the expense of
introducing cumbersome notation, which would make the exposition 
unreadable) to the general multi-parameter setting for arbitrary $k\ge2$,
by multiple iterations of the arguments presented in the paper.

We now give some remarks about Theorem \ref{thm:main}.
\begin{enumerate}[label*={\arabic*}.]

\item Theorem \ref{thm:main} establishes Conjecture \ref{con:0} for
the averages
$A_{M; X, T}^{P}f$. This is
the first nontrivial result in the literature establishing pointwise almost
everywhere convergence for polynomial ergodic averages in the
multi-parameter setting. See \cite{MSW-survey} for other pointwise convergence results
in the multi-parameter setting. 

\item The proof of Theorem \ref{thm:main} is relatively simple if
$P\in\ZZ[\rm m_1,\ldots, \rm m_k]$ is degenerate, see inequality
\eqref{eq:64} in Section \ref{section:3}. We will say that
$P\in\ZZ[\rm m_1,\ldots, \rm m_k]$ is degenerate if it can be written
as
\begin{align}
\label{eq:66}
P({\rm m}_1,\ldots, {\rm m}_k)=P_1({\rm m}_1)+\ldots+ P_k({\rm m}_k),
\end{align}
 where $P_1\in\ZZ[{\rm m}_1],\ldots, P_k\in\ZZ[{\rm m}_k]$ with $P_1(0)=\ldots= P_k(0)=0$. 
Otherwise we say that  $P\in\ZZ[\rm m_1, \ldots, \rm m_k]$ is non-degenerate. The method of proof of Theorem \ref{thm:main} in the degenerate case can be also used to derive quantitative oscillation bounds for the polynomial  Dunford and Zygmund theorem establishing \eqref{eq:26}.

\item At the
expense of great complexity, one can also prove that  inequality \eqref{eq:284} holds with $\ZZ_+$  in place of $\DD_{\tau}$. However, we do not address this question here, since \eqref{eq:284} is 
sufficient for our purposes, and will allow us to establish
Theorem \ref{thm:main}(ii).
\item If $(X, \mathcal B(X), \mu)$ is a probability space and the
measure preserving transformation $T$ in Theorem \ref{thm:main} is
totally ergodic, then Theorem \ref{thm:main}(ii) implies
\begin{align}
\label{eq:45}
\lim_{\min\{M_1,\ldots, M_k\}\to\infty}A_{M_1,\ldots, M_k; X, T}^{P}f(x)=\int_Xf(y)d\mu(y)
\end{align}
$\mu$-almost everywhere on $X$. We recall  that a measure preserving transformation $T$ is called \emph{ergodic} on $X$ if $T^{-1}[B]=B$ implies $\mu(B)=0$ or $\mu (B)=1$, and \emph{totally ergodic} if $T^n$ is ergodic for every $n\in\ZZ_+$.

\item This paper is the first systematic treatment of multi-parameter
oscillation semi-norms; see \eqref{eq:102}, Proposition \ref{prop:5}
and Proposition \ref{prop:4}.  Moreover, it seems that the oscillation
semi-norm is the only available tool that allows us to handle
efficiently multi-parameter pointwise convergence problems with
arithmetic features.  This contrasts sharply with the one-parameter
setting, where we have a variety of tools including oscillations,
variations or jumps to handle pointwise convergence problems; see
\cite{jsw, MST2} and the references therein. Multi-parameter
oscillations \eqref{eq:102} were considered for the first time in
\cite{JRW} in the context of the Dunford--Zygmund averages
\eqref{eq:11} for commuting measure-preserving transformations.

\end{enumerate}

We close this subsection by emphasizing that the methods developed in this paper allow us to handle averages \eqref{eq:229} with multiple polynomials. At the expense of some additional work one can prove the following ergodic theorem.  

\begin{theorem}
\label{thm:main0}
Let $(X, \mathcal B(X), \mu)$ be a $\sigma$-finite measure space equipped with a family of commuting  invertible and  measure-preserving transformations $T_1,T_2, T_{3}:X\to X$. Let $P\in\ZZ[{\rm m}_1, {\rm m}_2]$ be a polynomial such that
$P(0, 0)=\partial_1P(0, 0)=\partial_2P(0, 0)=0$, which  additionally
has partial degrees  (as a polynomial of the variable ${\rm m}_1$ and a polynomial of the variable ${\rm m}_2$) at least two. Let $f\in L^p(X)$ for some $1\le  p\le \infty$, and let $ A_{M_1, M_2; X}^{{\rm m}_1,{\rm m}_2 , P({\rm m}_1, {\rm m}_2)}f$ be the  average defined  in \eqref{eq:229} with $d=3$, $k=2$,  and $P_1({\rm m}_1, {\rm m}_2)={\rm m}_1$, $P_2({\rm m}_1, {\rm m}_2)={\rm m}_2$ and $P_3({\rm m}_1, {\rm m}_2)=P({\rm m}_1, {\rm m}_2)$.
\begin{itemize}
\item[(i)] \textit{(Mean ergodic theorem)} If $1<p<\infty$, then the averages
$A_{M_1, M_2; X}^{{\rm m}_1,{\rm m}_2 , P({\rm m}_1, \rm m_2)}f$ converge in $L^p(X)$ norm.

\item[(ii)] \textit{(Pointwise ergodic theorem)} If $1<p<\infty$, then the averages
$A_{M_1, M_2; X}^{{\rm m}_1,{\rm m}_2 , P({\rm m}_1, \rm m_2)}f$ converge pointwise almost everywhere.

\item[(iii)] \textit{(Maximal ergodic theorem)}
If $1<p\le\infty$, then one has
\begin{align}
\label{eq:35}
\big\|\sup_{M_1, M_2\in\ZZ_+}|A_{M_1, M_2; X}^{{\rm m}_1,{\rm m}_2 , P({\rm m}_1, \rm m_2)}f|\big\|_{L^p(X)}\lesssim_{p, P}\|f\|_{L^p(X)}.
\end{align}

\item[(iv)] \textit{(Oscillation ergodic theorem)}
If $1<p<\infty$ and $\tau>1$, then one has
\begin{align}
\label{eq:36}
\qquad \qquad\sup_{J\in\ZZ_+}\sup_{I\in\mathfrak S_J(\DD_{\tau}^2) }\big\|O_{I, J}(A_{M_1, M_2; X}^{{\rm m}_1,{\rm m}_2 , P({\rm m}_1, \rm m_2)}f: M_1, M_2\in\DD_{\tau})\|_{L^p(X)}\lesssim_{p, \tau, P}\|f\|_{L^p(X)},
\end{align}
where $\DD_{\tau}:=\{\tau^n:n\in\NN\}$. 
The implicit constant in \eqref{eq:35} and \eqref{eq:36} may depend on $p, \tau, P$.
\end{itemize}
\end{theorem}

For simplicity of notation, we have only formulated Theorem \ref{thm:main0} in the two-parameter setting but it can be extended to a multi-parameter setting as well. Namely, 
let $d\ge2$ and let $(X, \mathcal B(X), \mu)$ be a $\sigma$-finite measure space equipped with a family of commuting  invertible and  measure-preserving transformations $T_1,\ldots, T_{d}:X\to X$. Suppose that $P\in\ZZ[{\rm m}_1,\ldots,  {\rm m}_{d-1}]$ is a polynomial such that
\begin{align*}
P(0,\ldots, 0)=\partial_1P(0,\ldots, 0)=\ldots=\partial_{d-1}P(0,\ldots, 0)=0,
\end{align*}
which  
has partial degrees  (as a polynomial of the variable ${\rm m}_i$ for any $i\in[d-1]$) at least two. Then the conclusions of Theorem \ref{thm:main0} remain true for the averages 
\begin{equation}
\label{eq:44}
A_{M_1,\ldots, M_{d-1}; X, T_1,\ldots, T_d}^{{\mathrm m}_1,\ldots,{\mathrm m}_{d-1}, P({\mathrm m}_1,\ldots, {\rm m}_{d-1})}f
\ \ \ {\rm in \ place \ of} \ \ \  A_{M_1, M_2; X}^{{\rm m}_1,{\rm m}_2 , P({\rm m}_1, \rm m_2)}f.
\end{equation}
All remarks from items 1--4 after Theorem \ref{thm:main} remain true
for ergodic averages from \eqref{eq:44}. Finally, we emphasize that
Theorem \ref{thm:main} and Theorem \ref{thm:main0} make a contribution
to the famous Furstenberg--Bergelson--Leibman conjecture, which we now discuss.

\subsection{Contributions to the Furstenberg--Bergelson--Leibman conjecture} 
Furstenberg's ergodic proof \cite{Fur0} of Szemer{\'e}di's theorem
\cite{Sem1} (on the existence arbitrarily long arithmetic progressions
in subsets of integers with positive density) was a
departure point for modern ergodic Ramsey theory.
We refer to the survey articles \cite{Ber1}, \cite{Ber2}, and
\cite{Fra}, where details (including comprehensive historical
background) and an extensive literature are given about this
fascinating subject. Ergodic Ramsey theory is a very
rich body of research, consisting of many natural generalizations of
Szemer{\'e}di's theorem, including the celebrated polynomial Szemer{\'e}di theorem
of Bergelson and Leibman \cite{BL1} that motivates the following far-reaching conjecture:

\begin{conjecture}[Furstenberg--Bergelson--Leibman conjecture {\cite[Section 5.5, p. 468]{BL}}]
\label{con:3}
For given parameters $d, k, n\in\NN$, let $T_1,\ldots, T_d:X\to X$ be a
family of invertible measure-preserving transformations of a
probability measure space $(X, \mathcal B(X), \mu)$ that generates a
nilpotent group of step $l\in\ZZ_{+}$, and assume that
$P_{1, 1},\ldots,P_{i, j},\ldots, P_{d, n}\in \ZZ[\mathrm m_1,\ldots, \mathrm m_k]$. Then for any
$f_1, \ldots, f_n\in L^{\infty}(X)$, the non-conventional multiple
polynomial averages
\begin{align}
\label{eq:43}
A_{M; X, T_1,\ldots, T_d}^{P_{1, 1}, \ldots, P_{d, n}}(f_1,\ldots, f_n)(x)
=\EE_{m\in\prod_{j=1}^k[M_j]}\prod_{j=1}^nf_j(T_1^{P_{1, j}(m)}\cdots T_d^{P_{d, j}(m)} x)
\end{align}
converge for $\mu$-almost every $x\in X$ as $\min\{M_1,\ldots, M_k\}\to\infty$.
\end{conjecture}

Variants of this conjecture were promoted in person by Furstenberg,
(we refer to Austin's article \cite[pp. 6662]{A1}), before it was
published by Bergelson and Leibman \cite[Section 5.5, pp. 468]{BL}
for $k=1$.  The nilpotent and multi-parameter setting is the
appropriate setting for Conjecture \ref{con:3} as convergence may
fail if the transformations $T_1,\ldots, T_d$ generate a solvable
group, as shown by Bergelson and Leibman \cite{BL}.  The $L^2(X)$
norm convergence of \eqref{eq:43} has been studied since Furstenberg's
ergodic proof \cite{Fur0} of Szemer{\'e}di's theorem \cite{Sem1}, and
is fairly well understood (even in the setting of nilpotent groups)
due to the groundbreaking work of Walsh \cite{W} with
$M_1=\ldots=M_k$. Prior to Walsh's paper, extensive efforts had been
made towards understanding $L^2(X)$ norm convergence, including
breakthrough works of Host--Kra \cite{HK}, Ziegler \cite{Z1},
Bergelson \cite{Ber0}, and Leibman \cite{Leibman}. For more details
and references we also refer to \cite{A2, CFH, FraKra, HK1, Tao} and
the survey articles \cite{Ber1,Ber2,Fra}.

The situation is dramatically different for the pointwise convergence
problem \eqref{eq:43},  but recently, significant progress has been made towards
establishing the Furstenberg--Bergelson--Leibman conjecture. Now let us make a few
remarks about this conjecture, its history, and the current state of
the art.

\begin{enumerate}[label*={\arabic*}.]
\item The case $d=k=n=1$ of Conjecture \ref{con:3} with $P_{1, 1}(m)=m$ follows from Birkhoff's ergodic theorem \cite{BI}. In fact, the almost everywhere limit (as well as the norm limit, see also \cite{vN}) of \eqref{eq:43} exists also for all functions $f\in L^p(X)$, with $1\le p<\infty$, defined on any $\sigma$-finite measure space $(X, \mathcal B(X), \mu)$.

\item The case $d=k=n=1$  of Conjecture \ref{con:3} with arbitrary polynomials $P_{1, 1}\in\ZZ[\rm m]$ (as we have seen above)
was the famous open problem of Bellow \cite{Bel} and Furstenberg \cite{F}, which was solved by the first author \cite{B1, B2, B3} in the mid 1980's. In fact, in \cite{B1, B2, B3} it was shown that the almost everywhere limit (as well as the norm limit, see also \cite{Furbook}) of \eqref{eq:43} exists also for all functions $f\in L^p(X)$, with $1< p<\infty$, defined on any $\sigma$-finite measure space $(X, \mathcal B(X), \mu)$. In contrast to  the Birkhoff theorem, if $P_{1,1}\in\ZZ[\rm n]$ is a polynomial of degree at least two, the pointwise convergence at the endpoint for $p=1$ may fail as was shown by Buczolich and  Mauldin \cite{BM} for $P_{1,1}(m)=m^2$ and by LaVictoire \cite{LaV1} for $P_{1,1}(m)=m^k$ for any $k\ge2$.

\item In the commutative case (step $\ell = 1$) where $d,k\in\ZZ_+$ and $n=1$ of Conjecture \ref{con:3} with arbitrary polynomials $P_{1,1},\ldots, P_{d,1}\in\ZZ[{\rm m}_1, \ldots, {\rm m}_k]$ in the diagonal setting $M_1=\ldots=M_k$, that is, the multi-dimensional one-parameter setting, was solved by the second author with Trojan in \cite{MT1}. As before it was shown that the almost everywhere limit (as well as the norm limit) of \eqref{eq:43} exists also for all functions $f\in L^p(X)$, with $1< p<\infty$, defined on any $\sigma$-finite measure space $(X, \mathcal B(X), \mu)$.

\item The question to what extent one can relax the commutation relations
between $T_1,\ldots, T_d$ in \eqref{eq:43}, even in the one-parameter
case $M_1=\ldots=M_k$, is very intriguing. Some particular examples of
averages \eqref{eq:43} with $d,k\in\ZZ_+$ and $n=1$ and polynomial
mappings with degree at most two in the step two nilpotent setting
were studied in \cite{IMSW, MSW1}.  Recently, the second author with
Ionescu, Magyar and Szarek \cite{IMMS} established Conjecture
\ref{con:3} with $d\in\ZZ_+$ and $k=n=1$ and arbitrary polynomials
$P_{1,1},\ldots, P_{d,1}\in\ZZ[{\rm m}]$ in the nilpotent setting, i.e. when
$T_1,\ldots, T_{d}:X\to X$ is a family of invertible
measure-preserving transformations of a $\sigma$-finite measure space
$(X, \mathcal B(X), \mu)$ that generates a nilpotent group of step
two.

\item In contrast to the commutative linear theory, the multilinear
theory is wide open.  Only a few results are known in the bilinear $n=2$ and
commutative $d=k=1$ setting. The first author \cite{B0}  established
pointwise convergence when $P_{1,1}(m)=am$ and $P_{1,2}(m)=bm$, with
$a, b\in\ZZ$. Recently, the third author with Krause and Tao
\cite{KMT} proved pointwise convergence for the polynomial
Furstenberg--Weiss averages \cite{Fur1,FurWei} corresponding to
$P_{1,1}(m)=m$ and $P_{1, 2}(m)=P(m)$ with $P\in\ZZ[\rm m]$ and  ${\rm deg }\,P\ge2$.

\item A genuinely multi-parameter case $d=k\ge2$ with $n=1$ of Conjecture
\ref{con:3} for averages \eqref{eq:43} with linear orbits,
i.e. $P_{j,1}(m_1, \ldots, m_d)=m_j$ for $j\in[d]$ was established
independently by Dunford \cite{D} and Zygmund \cite{Z} in the early
1950's. Moreover, it follows from \cite{D, Z} that the
almost everywhere convergence (as well as the norm convergence) of
\eqref{eq:43} holds for all functions $f\in L^p(X)$, with
$1< p<\infty$, defined on any $\sigma$-finite measure space
$(X, \mathcal B(X), \mu)$ equipped with a family of measure-preserving
transformations $T_1,\ldots, T_d:X\to X$, which does not need to be
commutative. One also knows that pointwise convergence fails if $p=1$.
A polynomial variant of the Dunford and Zygmund theorem was discussed above, see  \eqref{eq:26}.

\end{enumerate}

We close this discussion by emphasizing that Theorem \ref{thm:main}
and Theorem \ref{thm:main0} also contribute to the 
Furstenberg--Bergelson--Leibman conjecture and together with all the
results listed above support the evidence that Conjecture \ref{con:3}
may be true in full generality though a complete solution seems very difficult.

\subsection{Overview of the paper}
The paper is organized as follows. In Section \ref{section:2} we fix
necessary notation and terminology. We also introduce the definition
of multi-parameter oscillations \eqref{eq:102} and collect their
useful properties, see Proposition \ref{prop:5} and Proposition
\ref{prop:4}.  In Section \ref{section:3} we give a detailed proof of
Theorem \ref{thm:main} by reducing the matter to oscillation estimates
for truncated variants of averages
$A_{M_1, M_2; X}^{ P}f$, see definition
\eqref{eq:229'}, and Theorem \ref{thm:main'}, which in turn is reduced
to the integer shift system, see Theorem \ref{thm:main''}. A result
that may be of independent interest is Proposition \ref{prop:7}, which
shows that oscillations for $A_{M_1, M_2; X}^{P}f$
and their truncated variants are in fact comparable. In Section
\ref{section:3}, see inequality \eqref{eq:64}, we also illustrate how
to prove Theorem \ref{thm:main} in the degenerate case in the sense of
definition \eqref{eq:66} stated after Theorem \ref{thm:main}. These
arguments can be also used to prove oscillation bounds for the
polynomial Dunford and Zygmund theorem, which in turn imply
\eqref{eq:26}.

We start with a brief overview of the proof of Theorem
\ref{thm:main''} which implies Theorem \ref{thm:3} when $k=2$ and takes up the bulk of this paper.  The proof requires substantial new ideas to overcome a
series of new difficulties arising in the multi-parameter
setting. These complications do not arise in the one-parameter setup
\cite{B1, B2, B3}.  The most notable obstacle is the lack of
nestedness in the definition of averaging operators \eqref{eq:229}
when the parameters $M_1,\ldots, M_k$ are allowed to run
independently. The lack of nestedness complicates every argument in
the circle method, which is the main tool in these kinds of
problems. In order to understand how the lack of nestedness may affect the
underlying arguments it will be convenient to illustrate this
phenomenon by comparing Figure \ref{fig:1} and Figure \ref{fig:2}
below. The first picture (Figure \ref{fig:1}) represents the family of
nested cubes, which is increasing when the time parameter increases.
The diagonal relation between parameters $M_1=\ldots=M_k$ is critical.

\begin{figure}[h]
\begin{tikzpicture}
\draw [<->] (0,3.5) -- (0,0) -- (3.5,0);
\draw (0,0.125) -- (0.125,0.125) -- (0.125, 0);
\draw (0,0.25) -- (0.25,0.25) -- (0.25, 0);
\draw (0,0.375) -- (0.375,0.375) -- (0.375, 0);
\draw (0,0.5) -- (0.5,0.5) -- (0.5, 0);
\draw (0,0.75) -- (0.75,0.75) -- (0.75, 0);
\draw (0,1) -- (1,1) -- (1, 0);
\node[below] at (1.0,1.5) {\tiny{$Q_{M, M}$}};
\draw (0,1.5) -- (1.5,1.5) -- (1.5, 0);
\node[below] at (1.5,2) {\tiny{$Q_{N, N}$}};
\draw (0,2) -- (2,2) -- (2, 0);
\draw (0,2.75) -- (2.75,2.75) -- (2.75, 0);
\node[below] at (3.23,3.65) {\reflectbox{$\ddots$}};
\node[below] at (1.5,-0.05) {\tiny{$M$}};
\node[below] at (2,-0.03) {\tiny{$N$}};
\node[below] at (-0.175,1.675) {\tiny{$M$}};
\node[below] at (-0.175,2.175) {\tiny{$N$}};
\end{tikzpicture}
\caption{Family of nested rectangles (cubes)
$Q_{M,M}\subset Q_{N,N}$ with $M<N$, for $k=2$.}\label{fig:1}
\end{figure}
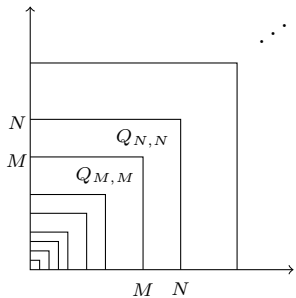

The second picture (Figure \ref{fig:2}) represents the family
which is genuinely multi-parameter and there is no nestedness as the parameters $M_1,\ldots, M_k$ vary independently.

\begin{figure}[h]
\begin{tikzpicture}
\draw [<->] (0,3.5) -- (0,0) -- (3.5,0);
\draw (0,0.125) -- (3.325,0.125) -- (3.325, 0);
\draw (0.125,0) -- (0.125,0.3) -- (0, 0.3);
\draw (0,0.20) -- (0.25,0.20) -- (0.25, 0);
\draw (0,0.375) -- (0.375,0.375) -- (0.375, 0);
\draw (0,3.25) -- (0.45,3.25) -- (0.45, 0);
\draw (0,0.75) -- (1,0.75) -- (1, 0);
\draw (0,1.75) -- (0.75,1.75) -- (0.75, 0);
\node[below] at (0.975,2.9) {\tiny{$Q_{M_1, M_2}$}};
\draw (0,2.9) -- (1.5,2.9) -- (1.5, 0);
\node[below] at (2.05,1.45) {\tiny{$Q_{N_1, N_2}$}};
\draw (0,1.5) -- (2.6,1.5) -- (2.6, 0);
\draw (0,2.35) -- (2.95,2.35) -- (2.95, 0);
\node[below] at (3.23,3.65) {\reflectbox{$\ddots$}};
\node[below] at (1.55,-0.05) {\tiny{$M_1$}};
\node[below] at (2.65,-0.05) {\tiny{$N_1$}};
\node[below] at (-0.175,1.675) {\tiny{$N_2$}};
\node[below] at (-0.175,3.075) {\tiny{$M_2$}};
\end{tikzpicture}
\caption{Family of un-nested rectangles
$Q_{M_1, M_2}\not\subseteq Q_{N_1, N_2}$ with $M_1<N_1$ and $M_2>N_2$,
for $k=2$.}\label{fig:2}
\end{figure}
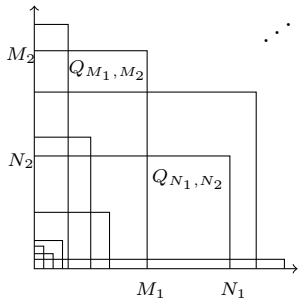

Our remedy to overcome the lack of nestedness will be to develop the
so-called multi-parameter circle method, which will be based on an
iterative implementation of the classical circle method.  Although
this idea sounds very simple it is fairly challenging to formalize it
in the context of Conjecture \ref{con:0}.
We remark that the  multi-parameter circle
method has been developed for many years in the context of various
problems arising in number theory, (see \cite{ACK} for more details and
references, including a comprehensive historical background), though it is not applicable directly in the ergodic context.
 We now highlight the key ingredients that we develop in this paper
 and that will lead us to develop the multi-parameter circle method in the context of Theorem \ref{thm:main''}:
\begin{itemize}
\item[(i)] ``Backwards'' Newton diagram is the key tool
allowing us to  overcome the problem with the lack of nestedness. In particular, it permits us to  understand geometric properties of the underlying polynomials in Theorem \ref{thm:main''} by extracting dominating monomials. The latter  are critical in making a distinction between minor and major arcs in the multi-parameter circle method.
As far as we know this is the first time when the concept of Newton diagrams is exploited in  problems concerning pointwise ergodic theory. We refer to Section \ref{section:4} for details. 
\item[(ii)] We derive new estimates for multi-parameter exponential sums arising in the analysis of Fourier multipliers corresponding to averages \eqref{eq:69}. In Section \ref{section:5} we build a theory of double exponential sums, which is dictated by the geometry of the corresponding ``backwards'' Newton diagrams. Although the theory of multi-parameter exponential sums is rich (see for example, \cite{ACK}) our results seem to be new and the idea of exploiting ``backwards'' Newton diagrams and iterative applications of the Vinogradov mean value theorem \cite{BDG} in estimates of exponential sums is quite efficient. 
\item[(iii)] A multi-parameter Ionescu--Wainger multiplier theory is developed in Section \ref{section:6}. The Ionescu--Wainger multiplier theorem \cite{IW} was originally proved for linear operators, see also \cite{M10, MSZ3, Pierce, TaoIW}. In this paper we prove a  semi-norm variant of the Ionescu--Wainger theory in the one-parameter setting, which is consequently upgraded to the multi-parameter setup.  ``Backwards'' Newton diagrams play an essential role in our considerations here as well.
\item[(iv)] Finally, we arrive at the stage where the multi-parameter circle method is feasible by a delicate iterative application of the classical circle method.
In this part of the argument the lack of nestedness is particularly unpleasant causing serious difficulties in controlling error terms that arise in estimating contributions of the corresponding Fourier multiplirs on minor and major arcs, which are genuinely multi-parameter. In Section  \ref{section:7} we illustrate how one can use all the tools developed in the previous sections to give a rigorous proof of Theorem \ref{thm:main''}.
\end{itemize}

We now take a closer look at the tools highlighted above. In Section \ref{section:4} we introduce the concept of ``backwards'' Newton diagram, which is the key to circumvent the difficulties caused by the lack of nestedness. The ``backwards'' Newton diagram splits the parameter space  into a finite number of sectors, where 
certain relations between parameters are given. In each of these sectors there is a dominating monomial which in turn gives rise to an implementation of the circle method to each of the sectors separately. The distinctions between minor and major arcs are then dictated by the degree of the associated dominating monomial. At this stage we eliminate minor arcs by invoking estimates of double exponential sums from Proposition \ref{prop:23}. This proposition is essential in our argument, its proof is given in Section \ref{section:5}. The key ingredients are Proposition \ref{prop:22}, which may be thought of as a two parameter counterpart of the classical Weyl's inequality, and the properties of the ``backwards'' Newton diagram. Although the theory of multi-parameter exponential sums has been developed over the years (see
\cite{ACK} for a comprehensive treatment of the subject), we require more delicate estimates than those available in the existing literature. In this paper we give an ad-hoc proof of Proposition \ref{prop:22}, which follows from an iterative application of Vinogradov's mean value theorem, and may be interesting in its own right. 
In Section \ref{section:5} we also develop estimates for complete exponential sums.
In Section \ref{section:6} we develop the Ionescu--Wainger multiplier theory for various semi-norms in one-parameter as well as in multi-parameter settings. Our result in the one-parameter setting, Theorem \ref{thm:IW1}, is formulated   for oscillations and maximal functions, but the proofs also work for $\rho$-variations or jumps. In fact, Theorem \ref{thm:IW1} is the starting point for establishing the corresponding multi-parameter Ionescu--Wainger theory for oscillations. The latter theorem  will be directly applicable in the analysis of multipliers associated with the averages $A_{M_1, M_2; X}^{P}f$. The results of Section \ref{section:6} are critical in our multi-parameter circle method that is presented in Section \ref{section:7}, as it allows us to efficiently control the error terms that arise on major arcs as well as the contribution coming from the main part. In contrast to the one-parameter theory \cite{B1, B2, B3}, the challenge here is to control, for instance, maximal functions corresponding to error terms. For this purpose all error terms have to be provided with asymptotic precision, which usually requires careful arguments. The details of the multi-parameter circle method are presented in Section \ref{section:7} in the context of the proof of Theorem \ref{thm:main''}.
\subsection{More about Conjecture \ref{con:3}}
Conjecture \ref{con:3} is one of the major open problems in
 pointwise ergodic theory, which seems to be very difficult due to
its multilinear nature.  Here, in light of the Arkhipov, Chubarikov and
Karatsuba \cite{ACK-fractional} equidistribution theory which works
also for multiple polynomials, it seems reasonable to propose a
slightly more modest problem (implied by Conjecture \ref{con:3}) though still very interesting and
challenging that can be subsumed under the following conjecture:

\begin{conjecture}
\label{con:1}
Let $d, k\in\ZZ_+$ be given and let $(X, \mathcal B(X), \mu)$ be a probability measure space endowed with a family  of invertible commuting  measure-preserving transformations $T_1,\ldots, T_d:X\to X$. Assume that $P_1,\ldots, P_d \in\ZZ[{\rm m}_1, \ldots, {\rm m}_k]$. Then for any $f\in L^{\infty}(X)$ the multi-parameter linear polynomial averages
\begin{align*}
A_{M_1,\ldots, M_k; X, T_1,\ldots, T_d}^{P_1, \ldots, P_d}f(x)=\EE_{m\in \prod_{j=1}^k[M_j]}f(T_1^{P_1(m)}\cdots T_d^{P_d(m)}x)
\end{align*}
converge for $\mu$-almost every $x\in X$, as  $\min\{M_1,\ldots, M_k\}\to\infty$.
\end{conjecture}

Even though we prove Conjecture \ref{con:0} here, it is not clear whether
Conjecture \ref{con:1} is true for all polynomials. If it is not true
for all polynomials, it would be interesting, in view of Theorem
\ref{thm:main0}, to characterize the class of those polynomials for
which Conjecture \ref{con:1} holds. Although the averages from Theorem
\ref{thm:main} and Theorem \ref{thm:main0} share a lot of difficulties
that arise in the general case there are some cases that are not
covered by the methods of this paper. An interesting difficulty arises for
the so-called partially complete exponential sums when we are seeking estimates  of the form
\begin{align}
\label{eq:75}
\frac{1}{M_1q}\sum_{m_1=1}^{M_1}\Big|\sum_{m_2=1}^q\ex(a_2m_2/q+a_3P(m_1, m_2)/q)\Big|\lesssim q^{-\delta},
\end{align}
for all $M_1, q\in\ZZ_+$ and some $\delta\in(0,1)$, whenever $(a_2, a_3, q)=1$. These kinds of estimates arise from applications of the circle method with respect to the second variable $m_2$ for the averages 
$A_{M_1, M_2; X}^{{\rm m}_1,{\rm m}_2 , P({\rm m}_1, \rm m_2)}f$ when we are at the stage of applying the circle method with respect to the first variable $m_1$. Here the assumption that $P$ has partial degrees  (as a polynomial of the variable $m_1$ and a polynomial of the variable $m_2$) at least two is essential. Otherwise, if $M_1<q$,  the decay $q^{-\delta}$ in \eqref{eq:75} is not possible. In order to see this it suffices to take $P(m_1, m_2)=m_1^2m_2$. A proof of Theorem \ref{thm:main0} for polynomials of this type as well as Conjecture \ref{con:1} will require a deeper understanding and substantially new methods. We believe that the
proof of Theorem \ref{thm:main} is an important contribution towards
understanding Conjecture \ref{con:1} that may shed new light on the
general case and either lead to its full resolution or to a
counterexample.  The second and fourth authors plan to pursue this problem in the future.

\subsection{In Memoriam}
It was a great privilege and an unforgettable experience for the
second and fourth authors to know and work with Elias M. Stein
(January 13, 1931 -- December 23, 2018) and Jean Bourgain (February
28, 1954 -- December 22, 2018).  Eli and Jean had an immeasurable
effect on our lives and careers.  It was a very sad time for us when we
learned that Eli and Jean passed away within an interval of one day in
December 2018.  We miss our friends and collaborators  dearly.

We now briefly describe how the collaboration on this project arose.
In 2011 the second and fourth authors started to work on some aspects
of a multi-parameter circle method in the context of various discrete
multi-parameter operators. These efforts resulted in a draft on
estimates for certain two-parameter exponential sums. This draft was
sent to the first author sometime in the first part of 2016.  In
October 2016, when the second author was a member of the Institute for
Advanced Study, it was realized (during a discussion between the first
two authors) that the estimates from this draft are closely related to
a multi-parameter Vinogradov's mean value theorem. This was
interesting to the first author who at that time was
involved in developing the theory of decoupling.  We also realized that some
ideas of a multi-parameter circle method from the draft of the second
and fourth authors may be upgraded and used in attacking a
multi-parameter variant of the Bellow and Furstenberg problem
formulated in Conjecture \ref{con:0}. That was the first time when the
second, third and fourth authors learned about this conjecture and
unpublished observations of the first author from the late 1980's that
resulted in establishing pointwise convergence in \eqref{eq:26}.  This
was the starting point of our collaboration. At that time another
question arose, which is also related to this paper.  It
is interesting whether a sharp multi-parameter variant of Vinogradov's
mean value theorem can be proved using the recent developments in the
decoupling theory from \cite{BDG}.  A multi-parameter Vinogradov's
mean value theorem was investigated in \cite{ACK}, but the bounds are
not optimal. So the question is about adapting the methods from
\cite{BDG} to the multi-parameter setting in order to obtain sharp
bounds, and their applications in the exponential sum estimates.

A substantial part of this project was completed at the end of
November/beginning of December 2016, when the fourth author visited
Princeton University and the Institute for Advanced Study. At that
time we discussed (more or less) all tools that were needed to
establish Theorem \ref{thm:main} for the monomial
$P(m_1, m_2)=m_1^2m_2^3$. Then we were convinced that we could
establish Conjecture \ref{con:1} in full generality, but various
difficulties arose when we started to work out the details and we
ultimately only managed to prove Theorem \ref{thm:main} and Theorem \ref{thm:main0}. The second
and fourth authors decided to illustrate the arguments in the
two-parameter setting and the reason is twofold. On the one hand, we
wanted to avoid introducing heavy multi-parameter notation capturing
all combinatorial nuances arising in this project. On the other hand,
what is more important we wanted to
illustrate the spirit of our discussions that took place in 2016. For instance the arguments
presented in Section \ref{section:5} can be derived by using Weyl
differencing argument, which may be even simpler and can be easily
adapted to the multi-parameter setting, though our presentation is very
close to the arguments that we developed in 2016, and also motivates the
question about the role of decoupling theory in the multi-parameter
Vinogradov's mean value theorem that we have stated above.

\subsection*{Acknowledgments}
We thank Mei-Chu Chang and Elly Stein who supported the idea of
completing this work. We thank Terry Tao for a fruitful discussion in
February 2015 about the estimates for multi-parameter exponential sums
and writing a very helpful blog on this subject \cite{TaoBlog}.  We
also thank Agnieszka Hejna, Dariusz Kosz and Bartosz Langowski for
careful reading of earlier versions of this manuscript and their
helpful comments and corrections.  Finally, we thank the referees for
careful reading of the manuscript and useful remarks that led to the
improvement of the presentation.

\section{Notation and useful tools}\label{section:2}
We now set up notation that will be used throughout the paper. We also collect useful tools and  basic properties of oscillation semi-norms that will be used in the paper. 

\subsection{Basic notation}  The set of positive integers and nonnegative
integers will be denoted respectively by $\ZZ_+:=\{1, 2, \ldots\}$ and
$\NN:=\{0,1,2,\ldots\}$. For $d\in\ZZ_+$ the sets $\ZZ^d$, $\RR^d$, $\CC^d$ and $\TT^d:=\RR^d/\ZZ^d$
have standard meaning.
For any $x\in\RR$ we will use the floor and fractional part functions
\begin{align*}
\lfloor x \rfloor: = \max\{ n \in \ZZ : n \le x \},
\qquad \text{ and } \qquad
\{x\}:=x-\lfloor x \rfloor.
\end{align*}
For $x, y\in\RR$ we shall also
write $x \vee y := \max\{x,y\}$ and $x \wedge y := \min\{x,y\}$.  
We denote $\RR_+:=(0, \infty)$ and
for every $N\in\RR_+$ we set
\[
[N]:=(0, N]\cap\ZZ=\{1, \ldots, \lfloor N\rfloor\},
\]
and we will also write
\begin{align*}
\NN_{\le N}:= [0, N]\cap\NN,\ \: \quad &\text{ and } \quad
\NN_{< N}:= [0, N)\cap\NN,\\
\NN_{\ge N}:= [N, \infty)\cap\NN, \quad &\text{ and } \quad
\NN_{> N}:= (N, \infty)\cap\NN.
\end{align*}
For any $\tau>1$ we will consider the set
\begin{align*}
\DD_{\tau}:=\{\tau^n: n\in\NN\}.
\end{align*}

For $a = (a_1,\ldots, a_n) \in \ZZ^n$ and $q\ge 1$ an integer, we denote by $(a,q)$ the greatest
common divisor of $a$ and $q$; that is, the largest integer $d\ge1$ that divides $q$ and all the components
$a_1, \ldots, a_n$. Clearly any vector in $\QQ^n$ has a unique representation as $a/q$ with $q\in \ZZ_{+}$,
$a \in \ZZ^n$ and $(a,q)=1$.

We use $\ind{A}$ to denote the indicator function of a set $A$. If $S$ 
is a statement we write $\ind{S}$ to denote its indicator, equal to $1$
if $S$ is true and $0$ if $S$ is false. For instance $\ind{A}(x)=\ind{x\in A}$.

Throughout the paper $C>0$ is an absolute constant which may
change from occurrence to occurrence. For two nonnegative quantities
$A, B$ we write $A \lesssim B$ if there is an absolute constant $C>0$
such that $A\le CB$. We will write $A \simeq B$ when
$A \lesssim B\lesssim A$.  We will write $\lesssim_{\delta}$ or
$\simeq_{\delta}$ to emphasize that the implicit constant depends on
$\delta$. For a function $f:X\to \CC$ and positive-valued function
$g:X\to (0, \infty)$, we write $f = O(g)$ if there exists a constant
$C>0$ such that $|f(x)| \le C g(x)$ for all $x\in X$. We will also
write $f = O_{\delta}(g)$ if the implicit constant depends on
$\delta$.

\subsection{Summation by parts}
For any real numbers $u<v$ and any sequences $(a_n:n\in\ZZ)\subseteq \CC$ and $(b_n:n\in\ZZ)\subseteq \CC$  we will use the following version of the summation by parts formula
\begin{align}
\label{eq:294}
\sum_{n\in(u, v]\cap\ZZ}a_nb_n=S_vb_{\lfloor v\rfloor}+\sum_{n\in(u, v-1]\cap\ZZ}S_n(b_n-b_{n+1}),
\end{align}
where $S_w:=\sum_{k\in(u, w]\cap\ZZ}a_k$ for any $w>u$.

\subsection{Euclidean spaces} For  $d\in\ZZ_+$ the set $\Set{e_i\in\RR^d}{i\in[d]}$ denotes the standard basis in
$\RR^d$. The standard inner product and the corresponding Euclidean norm on $\RR^d$ are denoted by 
\begin{align*}
x\cdot\xi:=\sum_{k=1}^dx_k\xi_k, \quad \text{ and } \quad
\abs{x}:=\abs{x}_2:=\sqrt{\ipr{x}{x}}
\end{align*}
for every $x=(x_1,\ldots, x_d)$ and $\xi=(\xi_1, \ldots, \xi_d)\in\RR^d$. 

Throughout the paper the $d$-dimensional torus $\TT^d$, which unless otherwise stated will be  identified with
$[-1/2, 1/2)^d$, is a priori endowed with the   periodic norm
\begin{align*}
\norm{\xi}:=\Big(\sum_{k=1}^d \norm{\xi_k}^2\Big)^{1/2}
\qquad \text{for}\qquad
\xi=(\xi_1,\ldots,\xi_d)\in\TT^d,
\end{align*}
where $\norm{\xi_k}=\dist(\xi_k, \ZZ)$ for all $\xi_k\in\TT$ and
$k\in[d]$.  However, identifying $\TT^d$ with $[-1/2, 1/2)^d$, we
see that the norm $\norm{\:\cdot\:}$ coincides with the Euclidean norm
$\abs{\:\cdot\:}$ restricted to $[-1/2, 1/2)^d$.

\subsection{Smooth functions}
The partial derivative of a differentiable function $f:\RR^d\to\CC$ with respect to the
$j$-th variable $x_j$ will be denoted by
$\partial_{x_j}f=\partial_j f$, while for any multi-index
$\alpha\in\NN^d$ let $\partial^{\alpha}f$ denote the derivative operator 
$\partial^{\alpha_1}_{x_1}\cdots \partial^{\alpha_d}_{x_d}f=\partial^{\alpha_1}_1\cdots \partial^{\alpha_d}_df$
of total order $|\alpha|:=\alpha_1+\ldots+\alpha_d$.

Let  $\eta:\RR\to[0, 1]$ be a smooth and even cutoff function such that
\begin{align*}
\ind{[-1, 1]}\le\eta\le \ind{[-2, 2]}.
\end{align*}
For any $n, \xi\in\RR$ we define
\begin{align*}
\eta_{\le n}(\xi):=\eta(2^{-n}\xi).
\end{align*}
For any $\xi=(\xi_1,\ldots, \xi_d)\in\RR^d$ and
$i\in[d]$ we also define
\begin{align*}
\eta_{\le n}^{(i)}(\xi):=\eta_{\le n}(\xi_i).
\end{align*}
More generally, for any 
$A=\{i_1,\ldots, i_m\}\subseteq [d]$ for some $m\in[d]$, and numbers $n_{i_1},\ldots, n_{i_m}\in\RR$ corresponding to the set $A$  we will write
\begin{align}
\label{eq:78}
\eta_{\le n_{i_1},\ldots, \le n_{i_m}}^{A}(\xi):=\prod_{j=1}^m\eta_{\le n_{i_j}}(\xi_{i_j})=\prod_{j=1}^m\eta_{\le n_{i_j}}^{(i_j)}(\xi).
\end{align}
If the elements of the set $A$ are ordered increasingly
$1\le i_1<\ldots< i_m\le d$ we will also write
\begin{align*}
\eta_{\le n_{i_1},\ldots, \le n_{i_m}}^{(i_1,\ldots, i_m)}(\xi):=\eta_{\le n_{i_1},\ldots, \le n_{i_m}}^{A}(\xi)=\prod_{j=1}^m\eta_{\le n_{i_j}}(\xi_{i_j})=\prod_{j=1}^m\eta_{\le n_{i_j}}^{(i_j)}(\xi).
\end{align*}
If $n_{i_1}=\ldots= n_{i_m}=n\in\RR$ we will abbreviate $\eta_{\le n_{i_1},\ldots, \le n_{i_m}}^{A}$ to $\eta_{\le n}^{A}$ and $\eta_{\le n_{i_1},\ldots, \le n_{i_m}}^{(i_1,\ldots, i_m)}$ to $\eta_{\le n}^{(i_1,\ldots, i_m)}$.

\subsection{Function spaces}
All vector spaces in this paper will be defined over the complex numbers $\CC$. 
The triple $(X, \mathcal B(X), \mu)$
is a measure space $X$ with $\sigma$-algebra $\mathcal B(X)$ and
$\sigma$-finite measure $\mu$.  The space of all $\mu$-measurable
complex-valued functions defined on $X$ will be denoted by $L^0(X)$.
The space of all functions in $L^0(X)$ whose modulus is integrable
with $p$-th power is denoted by $L^p(X)$ for $p\in(0, \infty)$,
whereas $L^{\infty}(X)$ denotes the space of all essentially bounded
functions in $L^0(X)$.
These notions can be extended to functions taking values in a finite
dimensional normed vector space $(B, \|\cdot\|_B)$, for instance
\begin{align*}
L^{p}(X;B)
:=\big\{F\in L^0(X;B):\|F\|_{L^{p}(X;B)} \coloneqq \left\|\|F\|_B\right\|_{L^{p}(X)}<\infty\big\},
\end{align*}
where $L^0(X;B)$ denotes the space of measurable functions from $X$ to
$B$ (up to almost everywhere equivalence). Of course, if $B$ is
separable, these notions can be extended to infinite-dimensional
$B$. In this paper, we will always be able to work in
finite-dimensional settings by appealing to  standard approximation arguments.
In our case we will usually have  $X=\RR^d$ or
$X=\TT^d$ equipped with Lebesgue measure, and   $X=\ZZ^d$ endowed with 
counting measure. If $X$ is endowed with counting measure we will
abbreviate $L^p(X)$ to $\ell^p(X)$ and $L^p(X; B)$ to $\ell^p(X; B)$.

If $T : B_1 \to B_2$ is a continuous linear  map between two normed
vector spaces $B_1$ and  $B_2$, we use $\|T\|_{B_1 \to B_2}$ to denote its
operator norm.

\medskip

The following extension of the Marcinkiewicz--Zygmund inequality to
the Hilbert space setting will be very useful in Section \ref{section:6}.
\begin{lemma}
\label{lem:10}
Let $(X, \mathcal B(X), \mu)$ be a $\sigma$-finite measure space endowed with
a family $T=(T_m: m\in\NN)$ of bounded linear operators
$T_m:L^p(X)\to L^p(X)$ for some $p\in(0, \infty)$.
 Suppose that 
\begin{align*}
A_p(T):=\sup_{(\omega_m:m\in\NN)\in\{-1, 1\}^{\NN}}\norm[\Big]{\sum_{m\in\NN}\omega_mT_m}_{L^p\to L^p}<\infty.
\end{align*}
Then there is a constant $C_p>0$ such that for every sequence
$(f_j: j\in\NN)\in L^p(X;\ell^2(\NN))$ we have
\begin{align}
\label{eq:175}
\norm[\Big]{\big(\sum_{j\in\NN}\sum_{m\in\NN}\abs{T_mf_j}^2\big)^{1/2}}_{L^p(X)}
\le C_pA_p(T)
\norm[\Big]{\big(\sum_{j\in\NN}\abs{f_j}^2\big)^{1/2}}_{L^p(X)}.
\end{align}
The index set $\NN$ in the inner sum  of \eqref{eq:175} can be
replaced by any other countable set and the result  remains valid.
\end{lemma}
The proof of Lemma \ref{lem:10} can be found in \cite{MST1}.

\subsection{Fourier transform}  
We shall write $\ex(z)=e^{2\pi {\bm i} z}$ for
every $z\in\CC$, where ${\bm i}^2=-1$. Let $\calF_{\RR^d}$ denote the Fourier transform on $\RR^d$ defined for
any $f \in L^1(\RR^d)$ and for any $\xi\in\RR^d$ as
\begin{align*}
\calF_{\RR^d} f(\xi) := \int_{\RR^d} f(x) \ex(x\cdot\xi) d x.
\end{align*}
If $f \in \ell^1(\ZZ^d)$ we define the discrete Fourier
transform (Fourier series) $\calF_{\ZZ^d}$, for any $\xi\in \TT^d$, by setting
\begin{align*}
\calF_{\ZZ^d}f(\xi): = \sum_{x \in \ZZ^d} f(x) \ex(x\cdot\xi).
\end{align*}
Sometimes we shall abbreviate $\calF_{\ZZ^d}f$ to $\hat{f}$.

Let $\GG=\RR^d$ or $\GG=\ZZ^d$. The corresponding dual groups are $\GG^*=(\RR^d)^*=\RR^d$ or $\GG^*=(\ZZ^d)^*=\TT^d$ respectively.
For any bounded function $\mathfrak m: \GG^*\to\CC$ and a test function $f:\GG\to\CC$ we define the Fourier multiplier operator  by 
\begin{align}
\label{eq:100}
T_{\GG}[\mathfrak m]f(x):=\int_{\GG^*}\ex(-\xi\cdot x)\mathfrak m(\xi)\calF_{\GG}f(\xi)d\xi, \quad \text{ for } \quad x\in\GG.
\end{align}
One may think that $f:\GG\to\CC$ is a compactly supported function on $\GG$ (and smooth if $\GG=\RR^d$) or any other function for which \eqref{eq:100} makes sense.

Let $\RR_{\le d}[{\rm x}_1,\ldots, {\rm x}_n]$ be the vector space of all polynomials on $\RR^n$ of degree at most $d\in\ZZ_+$, which is equipped with the norm $\|P\|:=\sum_{0\le|\beta|\le d}|c_{\beta}|$ whenever
\begin{align*}
P(x)=\sum_{0\le|\beta|\le d}c_{\beta}x_1^{\beta_1}\cdots x_n^{\beta_n} \quad \text{ for } \quad x=(x_1,\ldots, x_n)\in\RR^n.
\end{align*}
We now formulate a multidimensional variant of the van der Corput lemma for polynomials that will be useful in our further applications.
\begin{proposition}\label{thm:CCW}
For each $d, n\in\ZZ_+$ there exists a constant $C_{d, n}>0$ such that
for any  $P\in \RR_{\le d}[{\rm x}_1,\ldots, {\rm x}_n]$ with $P(0) = 0$, one has
\begin{align*}
  \bigg|\int_{[0, 1]^n}\ex(P(x))dx\bigg|\le C_{d, n}\|P\|^{-1/d}.
\end{align*}
\end{proposition}
The proof of Proposition \ref{thm:CCW} can be found in \cite[Corollary 7.3., p. 1008]{CCW}, see also \cite[Section 1]{ACK}.

\subsection{Comparing sums to integrals} A well-known but useful lemma comparing
sums to integrals is the following.  The proof can be found in \cite[Chapter V]{Zygmund}, see also \cite{V}.

\begin{lemma}\label{sum-integral} Suppose $f:[a,b] \to \RR$ is $C^1$ such that $f'$ is monotonic and $|f'(s)| \le 1/2$ on $[a,b]$. Then there is an 
absolute constant $A$ such that
\begin{align*}
\Big| \sum_{a<n\le b}\ex(f(n)) \ - \ \int_a^b \ex(f(s)) ds \Big| \ \le \ A.
\end{align*}
\end{lemma}

\subsection{Coordinatewise order $\preceq$}  For any $x=(x_1,\ldots, x_k)\in\RR^k$ and $y=(y_1,\ldots, y_k)\in\RR^k$ 
we say  $x\preceq y$ if an only if $x_i\le y_i$ for each $i\in[k]$.
We also write $x\prec y$ if an only if $x\preceq y$ and
$x\neq y$, and $x\prec_{\rm s} y$ if an only if $x_i< y_i$ for each
$i\in[k]$. Let $\II\subseteq \RR^k$ be an index set
such that $\# \II\ge2$ and 
for every $J\in\ZZ_+\cup\{\infty\}$
define the set
\begin{align}
\label{eq:77}
\mathfrak S_J(\II):
=
\Set[\big]{(t_i:i\in\NN_{\le J})\subseteq \II}{t_{0}\prec_{\rm s}
t_{1}\prec_{\rm s}\ldots \prec_{\rm s}t_{J}},
\end{align}
where $\NN_{\le \infty}:=\NN$.
In other
words, $\mathfrak S_J(\II)$ is a family of all strictly increasing
sequences (with respect to the coordinatewise order) of length $J+1$
taking their values in the set $\II$.

\subsection{Oscillation semi-norms}
Let $\II\subseteq \RR^k$ be an index set such that $\#{\II}\ge2$. Let $(\mathfrak a_{t}(x): t\in\II)\subseteq\CC$ be a $k$-parameter
family of measurable functions defined on $X$. For any
$\JJ\subseteq \II$ and a sequence
$I=(I_i : i\in\NN_{\le J}) \in \mathfrak S_J(\II)$ the multi-parameter
oscillation semi-norm is defined by
\begin{align}
\label{eq:102}
O_{I, J}(\mathfrak a_{t}(x): t \in \JJ):=
\Big(\sum_{j=0}^{J-1}\sup_{t\in \BB[I,j]\cap\JJ}
\abs{\mathfrak a_{t}(x) - \mathfrak a_{I_j}(x)}^2\Big)^{1/2},
\end{align}
where
$\BB[I,i]:=[I_{i1}, I_{(i+1)1})\times\ldots\times[I_{ik}, I_{(i+1)k})$
is a box determined by the element $I_i=(I_{i1}, \ldots, I_{ik})$ of
the sequence $I\in \mathfrak S_J(\II)$. 
In order to avoid problems with measurability we always assume that
$\II\ni t\mapsto \mathfrak a_{t}(x)\in\CC$ is continuous for
$\mu$-almost every $x\in X$, or $\JJ$ is countable. We also use the
convention that the supremum taken over the empty set is zero.
\begin{remark}
\label{rem:1}
Some remarks concerning the definition of oscillation semi-norms are in order.
\begin{enumerate}[label*={\arabic*}.]
\item Clearly, $O_{I, J}(\mathfrak a_{t}: t \in \JJ)$ defines a semi-norm.

\item  Let $\II\subseteq \RR^k$ be an index set such that $\#{\II}\ge2$, and let $\JJ_1, \JJ_2\subseteq \II$ be disjoint. Then for any family $(\mathfrak a_t:t\in\II)\subseteq \CC$,  any $J\in\ZZ_+$ and any $I\in\mathfrak S_J(\II)$ one has
\begin{align}
\label{eq:116}
O_{I, J}(\mathfrak a_{t}: t\in\JJ_1\cup\JJ_2)\le O_{I, J}(\mathfrak a_{t}: t\in\JJ_1)+O_{I, J}(\mathfrak a_{t}: t\in\JJ_2).
\end{align}

\item Let $\II\subseteq \RR^k$ be a countable index set such that $\#{\II}\ge2$ and $\JJ\subseteq \II$. 
Then 
 for any family $(\mathfrak a_t:t\in\II)\subseteq \CC$, any  $J\in\ZZ_+$, any  $I\in\mathfrak S_J(\II)$  one has
 \begin{align*}
 O_{I, J}(\mathfrak a_{t}: t \in \JJ)\lesssim \Big(\sum_{t\in\II}|\mathfrak a_{t}|^2\Big)^{1/2}.
 \end{align*}
\item Let $\II\subseteq \RR^k$ be a countable index set such that $\#{\II}\ge2$.
For $l\in[k]$,
let $\proj_l:\RR^k \to \RR$ be the $l$th coordinate projection. 
Note that for any family $(\mathfrak a_t:t\in\II)\subseteq \CC$, any  $J\in\ZZ_+$, any  $I\in\mathfrak S_J(\II)$ and any $l\in[k]$ one has
\begin{align}
\label{eq:38}
\begin{split}
&O_{I, J}(\mathfrak a_{t}: t\in\II)=\Big(\sum_{j=0}^{J-1}\sup_{t\in\BB[I,j]\cap\II}|\mathfrak a_{t}-\mathfrak a_{I_j}|^2\Big)^{1/2}\\
&\hspace{1cm}\lesssim \Big(\sum_{t_l\in\proj_l(\II)}\sup_{\substack{(t_1,\ldots, t_{l-1},  t_{l+1},\ldots, t_k)\in\prod_{i\in[k]\setminus\{l\}}\proj_i(\II)\\(t_1,\ldots, t_{l-1}, t_l, t_{l+1},\ldots, t_k)\in\II}}|\mathfrak a_{(t_1,\ldots, t_{l-1}, t_l, t_{l+1},\ldots, t_k)}|^2\Big)^{1/2},
\end{split}
\end{align}
where $\proj_l(\II) \subset \RR$ is the image of $\II$ under $\proj_l$.
Inequality \eqref{eq:38} will be repeatedly used in Section \ref{section:7}. It is important to note that the parameter $t\in\II$ in the definition of oscillations and the sequence $I\in\mathfrak S_J(\II)$ both take values in $\II$.

\item We also recall the definition of $\rho$-variations. For any $\mathbb I\subseteq \RR$, any family $(\mathfrak a_t: t\in\mathbb I)\subseteq \CC$, and any exponent
$1 \leq \rho < \infty$, the $\rho$-variation semi-norm is defined to be
\begin{align*}
V^{\rho}( \mathfrak a_t: t\in\mathbb I):=
\sup_{J\in\ZZ_+} \sup_{\substack{t_{0}<\dotsb<t_{J}\\ t_{j}\in\mathbb I}}
\Big(\sum_{j=0}^{J-1}  |\mathfrak a_{t_{j+1}}-\mathfrak a_{t_{j}}|^{\rho} \Big)^{1/\rho},
\end{align*}
where the supremum is taken over all finite increasing sequences in
$\mathbb I$.

It is clear that for any $\II\subseteq \RR$  such that $\#{\II}\ge2$,
any $J\in\ZZ_+\cup\{\infty\}$ and any sequence $I=(I_i : i\in\NN_{\le J}) \in \mathfrak S_J(\II)$ one has
\begin{align}
\label{eq:137}
O_{I, J}(\mathfrak a_{t}: t \in \II)\le V^{\rho}( \mathfrak a_t: t\in\mathbb I),
\end{align}
whenever $1\le \rho\le 2$.

\item Inequality \eqref{eq:137} allows us to
deduce the Rademacher--Menshov inequality for oscillations, which
asserts that for any  $j_0, m\in\NN$ so that $j_0< 2^m$ and any
sequence of complex numbers $(\mathfrak a_k: k\in\NN)$, any $J\in[2^m]$ and any $I\in\mathfrak S_J([j_0, 2^m))$ we have
\begin{align}
\label{eq:164}
\begin{split}
O_{I, J}(\mathfrak a_{j}: j_0\leq j< 2^m)&\le V^{2}( \mathfrak a_j: j_0\leq j< 2^m)\\
&\leq \sqrt{2}\sum_{i=0}^m\Big(\sum_{j=0}^{2^{m-i}-1}\big|\sum_{\substack{k\in U_{j}^i\\
U_{j}^i\subseteq [j_0, 2^m)}} \mathfrak a_{k+1}-\mathfrak a_{k}\big|^2\Big)^{1/2},
\end{split}
\end{align}
where $U_j^i:=[j2^i, (j+1)2^i)$ for any $i, j\in\ZZ$. The latter inequality in \eqref{eq:164} immediately follows from 
\cite[Lemma 2.5., p. 534]{MSZ2}. Inequality \eqref{eq:164} will be used in Section \ref{section:6}.

\item For any $p\in[1, \infty]$ and for any family $(\mathfrak a_t:t\in\NN^k)\subseteq \CC$ of $k$-parameter measurable functions on $X$, one has
\begin{align}
\label{eq:150}
\begin{split}
\qquad \sup_{J\in\ZZ_+}\sup_{I\in \mathfrak S_J(\NN^k)}\norm[\big]{O_{I, J}(\mathfrak a_{t}: t \in \NN^k)}_{L^p(X)}\le &  \ 2 \
\big\|\sup_{t\in \NN^k}|\mathfrak a_{t}|\big\|_{L^p(X)}\\
&+
\sup_{J\in\ZZ_+}\sup_{I\in \mathfrak S_J(\ZZ_+^k)}\norm[\big]{O_{I, J}(\mathfrak a_{t}: t \in \ZZ_+^k)}_{L^p(X)}.
\end{split}
\end{align}
This easily follows from the definition of the set $\mathfrak S_J(\NN^k)$, see \eqref{eq:77}.
\item  For any $\II\subseteq\RR$ with
$\#{\II}\ge2$ and any sequence
$I=(I_i : i\in\NN_{\le J}) \in \mathfrak S_J(\II)$ of length $J\in\ZZ_+\cup\{\infty\}$ we define the diagonal sequence 
$\bar{I}=(\bar{I}_i : i\in\NN_{\le J}) \in \mathfrak S_J(\II^k)$ by setting 
$\bar{I}_i=(I_i,\ldots,I_i)\in\II^k$ for each $i\in\NN_{\le J}$. Then for any $\JJ\subseteq \II^k$ one has
\begin{align*}
\sup_{I\in \mathfrak S_J(\II)}\norm[\big]{O_{\bar{I}, J}(\mathfrak a_{t}: t \in \JJ)}_{L^p(X)}
\le
\sup_{I\in \mathfrak S_J(\II^k)}\norm[\big]{O_{I, J}(\mathfrak a_{t}: t \in \JJ)}_{L^p(X)}.
\end{align*}

\end{enumerate}
\end{remark}

It is not difficult to  show that oscillation semi-norms
always dominate maximal functions. 
\begin{proposition}
\label{prop:5}
Assume that $k\in\ZZ_+$ and let $(\mathfrak a_{t}: t\in\RR^k)\subseteq\CC$
be a $k$-parameter family of measurable functions on $X$. Let $\II\subseteq \RR$ and $\#{\II}\ge2$, then
for every $p\in[1, \infty]$ we have
\begin{align}
\label{eq:135}
\norm[\big]{\sup_{t \in (\II\setminus\{\sup\II\})^k}\abs{\mathfrak a_{t}}}_{L^p(X)}
\le \sup_{t \in \II^k}\norm{\mathfrak a_{t}}_{L^p(X)}
+\sup_{J\in\ZZ_+}\sup_{I\in \mathfrak S_J(\II)}\norm[\big]{O_{\bar{I}, J}(\mathfrak a_{t}: t \in \II^k)}_{L^p(X)},
\end{align}
where $\bar{I}\in\mathfrak S_J(\II^k)$ is the diagonal sequence
corresponding to a sequence $I\in \mathfrak S_J(\II)$ as in Remark
\ref{rem:1}.
\end{proposition}
A remarkable feature of the oscillation semi-norms is that they imply
pointwise convergence, which is formulated precisely in the
following proposition.

\begin{proposition}
\label{prop:4}
Let $(X, \calB(X), \mu)$ be a $\sigma$-finite measure space. For
$k\in\ZZ_+$ let $(\mathfrak a_{t}: t\in\NN^k)\subseteq\CC$ be a
$k$-parameter family of measurable functions on $X$. Suppose that
there is $p\in[1, \infty)$ and a constant $C_p>0$ such that
\begin{align*}
\sup_{J\in\ZZ_+}\sup_{I\in \mathfrak S_J(\NN)}
\norm[\big]{O_{\bar{I}, J}(\mathfrak a_{t}: t \in \NN^k)}_{L^p(X)}\le C_p<\infty.
\end{align*}
where $\bar{I}\in\mathfrak S_J(\NN^k)$ is the diagonal sequence
corresponding to a sequence $I\in \mathfrak S_J(\NN)$ as in Remark
\ref{rem:1}.  Then the limit
\begin{align*}
\lim_{\min\{t_1,\ldots, t_k\}\to\infty}\mathfrak a_{(t_1,\ldots, t_k)} 
\end{align*}
exists  $\mu$-almost everywhere on $X$.
\end{proposition}
For detailed proofs of Proposition \ref{prop:5} and Proposition \ref{prop:4}, we refer to \cite{MSW-survey}.

\section{Basic reductions and ergodic theorems: Proof of Theorem \ref{thm:main}}
\label{section:3}
This section is intended to establish Theorem \ref{thm:main} for general measure-preserving systems by reducing the matter to the integer shift system.
 We first briefly explain that the oscillation inequality \eqref{eq:284} from item (iv) of 
Theorem \ref{thm:main} implies conclusions from items (i)--(iii) of this theorem. 
\subsection{Proof of Theorem \ref{thm:main}(iii)} Assuming Theorem \ref{thm:main}(iv) with $\tau=2$ and invoking Proposition \ref{prop:5} (this permits us to dominate maximal functions by oscillations) we see that for every $p\in(1, \infty)$ there is a constant $C_p>0$ such that for any $f\in L^p(X)$ one has
\begin{align}
\label{eq:143}
\big\|\sup_{M_1, M_2\in\DD_2}|A_{M_1, M_2; X}^{P}f|\big\|_{L^p(X)}\lesssim_{p, P}\|f\|_{L^p(X)}.
\end{align}
But for any $f\ge0$ we have also a simple pointwise bound 
\begin{align*}
\sup_{M_1, M_2\in\ZZ_+}A_{M_1, M_2; X}^{P}f\lesssim \sup_{M_1, M_2\in\DD_2}A_{M_1, M_2; X}^{P}f,
\end{align*}
which in view of \eqref{eq:143} gives \eqref{eq:230} as claimed.\qed 

\subsection{Proof of Theorem \ref{thm:main}(ii)}
We fix $p\in(1, \infty)$ and $f\in L^p(X)$. We can also assume that $f\ge0$.
Using \eqref{eq:284} with $\tau=2^{1/s}$ for every $s\in\ZZ_+$
and invoking Proposition \ref{prop:4} we conclude that 
there is $f_{s}^*\in L^p(X)$ such that
\begin{align*}
\lim_{\min\{n_1, n_2\}\to\infty}A_{2^{n_1/s}, 2^{n_2/s}; X}^{P}f(x)=f_{s}^*(x)
\end{align*}
$\mu$-almost everywhere on $X$ for every $s\in\ZZ_+$.
It is not difficult to see that $f^*_{1}=f^*_{s}$ for all $s\in\ZZ_+$, 
since $\mathbb D_2\subseteq \mathbb D_{2^{1/s}}$.
Now for each $s\in\ZZ_+$ and
each $M_1, M_2\in\ZZ_+$ let $n_{M_i}^i \in \NN$ be 
 such that $2^{n_{M_i}^i/s}\le M_i<2^{(n_{M_i}^i+1)/s}$ for $i\in[2]$. Then we may conclude
\begin{align*}
2^{-2/s}f^*_1(x)\le
\liminf_{\min\{M_1, M_2\}\to\infty}A_{M_1, M_2; X}^{P}f(x)\le
\limsup_{\min\{M_1, M_2\}\to\infty}A_{M_1, M_2; X}^{P}f(x)\le
2^{2/s}f^*_1(x).
\end{align*}
Letting $s\to \infty$ we obtain
\begin{align*}
\lim_{\min\{M_1, M_2\}\to\infty}A_{M_1, M_2; X}^{P}f(x)=f^*_1(x)
\end{align*}
$\mu$-almost everywhere on $X$. This completes the proof of Theorem \ref{thm:main}(ii).\qed

\subsection{Proof of Theorem \ref{thm:main}(i)} Finally,  pointwise convergence from Theorem \ref{thm:main}(ii) combined with the maximal inequality \eqref{eq:230} the  and dominated convergence theorem gives norm convergence for any $f\in L^p(X)$ with $1<p<\infty$. This completes the proof of Theorem \ref{thm:main}. \qed

\subsection{Proof of Theorem \ref{thm:main} in the degenerate case}\label{deg} It is perhaps worth remarking that the proof of Theorem \ref{thm:main} is fairly easy when $P\in\ZZ[{\rm m}_1, {\rm m}_2]$ is degenerate in the sense that it can be written as $P({\rm m}_1, {\rm m}_2)=P_1({\rm m}_1)+P_2({\rm m}_2)$, where $P_1\in\ZZ[{\rm m}_1]$ and $P_2\in\ZZ[{\rm m}_2]$ such that $P_1(0)=P_2(0)=0$ (see \eqref{eq:66}).  It suffices to prove \eqref{eq:284}. The crucial observation is the following identity
\begin{align}
\label{eq:46}
A_{M_1; X}^{P_1({\rm m}_1)}A_{M_2; X}^{P_2({\rm m}_2)}f=A_{M_2; X}^{P_2({\rm m}_2)}A_{M_1; X}^{P_1({\rm m}_1)}f=A_{M_1, M_2; X}^{P({\rm m}_1, \rm m_2)}f.
\end{align}
Recall from \cite{MST1} that for every $p\in(1, \infty)$ there is $C_p>0$ such that for every $f=(f_{\iota}:\iota\in\NN)\in L^p(X; \ell^2(\NN))$ and $i\in[2]$ one has
\begin{align}
\label{eq:37}
\Big\|\Big(\sum_{\iota\in\NN}\sup_{M_i\in\ZZ_+}\big|A_{M_i; X}^{P_i({\rm m}_i)}f_{\iota}\big|^2\Big)^{1/2}\Big\|_{L^p(X)}\le C_p\|f\|_{L^p(X; \ell^2)}.
\end{align}
Moreover from \cite{MSS}, it was proved  that for every $p\in(1, \infty)$ there is $C_p>0$ such that for every $f\in L^p(X)$ and $i\in[2]$ one has
\begin{align}
\label{eq:45}
\sup_{J\in\ZZ_+}\sup_{I\in\mathfrak S_J(\ZZ_+) }\big\|O_{I, J}(A_{M_i; X}^{P_i({\rm m}_i)}f:  M_i\in\ZZ_+)\|_{L^p(X)}\le C_p\|f\|_{L^p(X)}.
\end{align}
By \eqref{eq:46} for every $J\in\ZZ_+$, $I\in\mathfrak S_J(\ZZ_+^2)$ and  $j\in\NN_{<J}$ one can write 
\begin{align*}
\sup_{(M_1, M_2)\in\BB[I,j]}&\big|A_{M_1, M_2; X}^{P({\rm m}_1, \rm m_2)}f-A_{I_{j1}, I_{j2}; X}^{P({\rm m}_1, \rm m_2)}f\big|\\
&\le
\sup_{M_1\in\ZZ_+}\big|A_{M_1; X}^{P_1({\rm m}_1)}\big(\sup_{I_{j2}\le M_2<I_{(j+1)2}}|A_{M_2; X}^{P_2({\rm m}_2)}f-A_{I_{j2}; X}^{P_2({\rm m}_2)}f|\big)\\
&+
\sup_{M_2\in\ZZ_+}\big|A_{M_2; X}^{P_2({\rm m}_2)}\big(\sup_{I_{j1}\le M_1<I_{(j+1)1}}|A_{M_1; X}^{P_1({\rm m}_1)}f-A_{I_{j1}; X}^{P_1({\rm m}_1)}f|\big)\big|.
\end{align*}
Using this inequality with the vector-valued maximal inequality \eqref{eq:37} and one-parameter oscillation inequality \eqref{eq:45} one obtains
\begin{multline}
\label{eq:64}
\sup_{J\in\ZZ_+}\sup_{I\in\mathfrak S_J(\ZZ_+^2) }\big\|O_{I, J}(A_{M_1, M_2; X}^{P({\rm m}_1, \rm m_2)}f: M_1, M_2\in\ZZ_+)\|_{L^p(X)}\\
\lesssim \sum_{i\in[2]}
\sup_{J\in\ZZ_+}\sup_{I\in\mathfrak S_J(\ZZ_+) }\big\|O_{I, J}(A_{M_i; X}^{P_i({\rm m}_i)}f:  M_i\in\ZZ_+)\|_{L^p(X)}\lesssim_p\|f\|_{L^p(X)}.
\end{multline}
This completes the proof of Theorem \ref{thm:main} in the degenerate case. From now on we will additionally  assume that  $P\in\ZZ[{\rm m}_1, {\rm m}_2]$ is non-degenerate.\qed

\subsection{Reductions to truncated averages}
We have seen that the proof  of Theorem \ref{thm:main} has been reduced to proving the oscillation inequality \eqref{eq:284}. 
We begin with certain general reductions that will simplify our further arguments.   Let us fix our measure-preserving
transformations $T_1, \ldots, T_d$, our polynomials ${\mathcal P} = \{P_1,\ldots, P_d\}\subset {\mathbb Z}[{\rm m}_1,\ldots,{\rm m}_k]$ and define a truncated version of the average \eqref{eq:229} by
\begin{align}
\label{eq:229'}
\tilde{A}_{M_1,\ldots, M_k; X}^{{\mathcal P}}f(x):=\EE_{m\in R_{M_1,\ldots, M_k}}f(T_1^{P_1(m)}\cdots T_d^{P_d(m)}x), \qquad x\in X,
\end{align}
where
\[
R_{M_1,\ldots, M_k}:= ([M_1]\setminus[\tau^{-1}M_1])\times\cdots\times([M_k]\setminus[\tau^{-1}M_k])
\]
is a rectangle in $\ZZ^k$.

We will abbreviate $\tilde{A}_{M_1,\ldots, M_k; X}^{{\mathcal P}}$ to $\tilde{A}_{M; X}^{{\mathcal P}}$
and $R_{M_1,\ldots, M_k}$ to $R_M$ whenever $M=(M_1,\ldots, M_k)\in\ZZ_+^k$. We now show that the $L^p(X)$ norms of the oscillation semi-norms  associated with the averages from \eqref{eq:229} and \eqref{eq:229'} have comparable norms in the following sense.

\begin{proposition}
\label{prop:7}
Let $d, k\in\ZZ_+$  be given. Let $(X, \mathcal B(X), \mu)$ be a $\sigma$-finite measure space equipped with a family of commuting  invertible and  measure-preserving transformations $T_1,\ldots, T_d:X\to X$. Let 
${\mathcal P} = \{P_1,\ldots, P_d\}\subset \ZZ[{\rm m}_1, \ldots, {\rm m}_k]$, $M=(M_1,\ldots,M_k)$ and let $A_{M; X}^{{\mathcal P}}$ and $\tilde{A}_{M; X}^{{\mathcal P}}$ be the corresponding averaging operators defined respectively in \eqref{eq:229} and \eqref{eq:229'}.  For every $\tau>1$ and every $1\le  p\le \infty$  there is a finite constant $C:=C_{d, k, p, \tau}>0$ such that for any  $f\in L^p(X)$ one has
\begin{align}
\label{eq:19}
\begin{split}
\norm[\big]{\sup_{M\in\DD_{\tau}^k}|A_{M; X}^{{\mathcal P}}|}_{L^p(X)}
\le
C\norm[\big]{\sup_{M\in\DD_{\tau}^k}|\tilde{A}_{M; X}^{\mathcal P}|}_{L^p(X)}.
\end{split}
\end{align}
An oscillation variant of \eqref{eq:19} also holds
\begin{align}
\label{eq:23}
\begin{split}
&\sup_{J\in\ZZ_+}\sup_{I\in \mathfrak S_J(\DD_{\tau}^k)}
\norm[\big]{O_{I, J}(A_{M; X}^{\mathcal P}:M\in\DD_{\tau}^k)}_{L^p(X)}\\
&\hspace{1cm}\le C\sup_{J\in\ZZ_+}\sup_{I\in \mathfrak S_J(\DD_{\tau}^k)}
\norm[\big]{O_{I, J}(\tilde{A}_{M; X}^{\mathcal P}f:M\in\DD_{\tau}^k)}_{L^p(X)}+C\norm{f}_{L^p(X)}.
\end{split}
\end{align}
\end{proposition}
\begin{proof}
The proof will proceed in two steps. We begin with some general observations which will permit us to simplify further arguments leading to the proofs of \eqref{eq:19} and \eqref{eq:23}.
\paragraph{\bf Step 1} Suppose that
$(\mathfrak a_{m}: m\in\ZZ_+^k)$ is a $k$-parameter sequence of measurable functions on $X$. Then for $M=(M_1,\ldots, M_k)=(\tau^{n_1},\ldots, \tau^{n_k})\in\DD_{\tau}^k$ one can write 
\begin{align*}
\sum_{m\in Q_{M_1,\ldots, M_k}}\mathfrak a_{m}=\sum_{(l_1,\ldots, l_k)\in\NN_{\le n_1}\times\cdots\times\NN_{\le n_k}}\sum_{m\in R_{\tau^{l_1},\ldots, \tau^{l_k}}}\mathfrak a_{m},
\end{align*}
and
\begin{align*}
\sum_{(l_1,\ldots, l_k)\in\NN_{\le n_1}\times\cdots\times\NN_{\le n_k}}\frac{|R_{\tau^{l_1},\ldots, \tau^{l_k}}|}{|Q_{\tau^{n_1},\ldots, \tau^{n_k}}|}\lesssim_{k,\tau}1. 
\end{align*}
Combining these two estimates one sees that
\begin{align}
\label{eq:24}
\big\|\sup_{M\in\DD_{\tau}^k}|\EE_{m\in Q_M}\mathfrak a_{m}|\big\|_{L^p(X)}\lesssim_{k, \tau}\big\|\sup_{M\in\DD_{\tau}^k}|\EE_{m\in R_M}\mathfrak a_{m}|\big\|_{L^p(X)}.
\end{align}
Applying \eqref{eq:24} with $\mathfrak a_{m}(x)=f(T_1^{P_1(m)}\cdots T_d^{P_d(m)}x)$ we obtain \eqref{eq:19}.
\paragraph{\bf Step 2} As before  let
$(\mathfrak a_{m}: m\in\ZZ_+^k)$ be a $k$-parameter sequence of measurable functions on $X$. For $l\in\NN_{\le k}$ and $M=(M_1,\ldots, M_k)=(\tau^{n_1},\ldots, \tau^{n_k})\in\DD_{\tau}^k$ define the sets
\begin{align*}
B_M^{l}:=\prod_{i=1}^{l}([M_i]\setminus[\tau^{-1}M_i])\times\prod_{i=l+1}^k [M_{i}]
\quad \text{ and } \quad
D_M^{l}:=\prod_{i=1}^{l-1}([M_i]\setminus[\tau^{-1}M_i])\times [\tau^{-1}M_{l}]\times\prod_{i=l+1}^k [M_{i}].
\end{align*}
Note that $B_M^0=Q_M$ and $B_M^k=R_M$, and $B_M^{l-1}=B_M^{l}\cup D_M^{l}$. Moreover, for $l\in[k]$ one sees 
\begin{align}
\label{eq:31}
\begin{split}
&\sup_{J\in\ZZ_+}\sup_{I\in \mathfrak S_J(\DD_{\tau}^k)}
\norm[\big]{O_{I, J}(\EE_{m\in B_M^{l-1}}\mathfrak a_{m}:M\in\DD_{\tau}^k)}_{L^p(X)}\\
&\hspace{3cm}\le \sup_{J\in\ZZ_+}\sup_{I\in \mathfrak S_J(\DD_{\tau}^k)}
\norm[\big]{O_{I, J}(u_{M_l}\EE_{m\in B_M^{l}}\mathfrak a_{m}:M\in\DD_{\tau}^k)}_{L^p(X)}\\
&\hspace{3cm}+\sup_{J\in\ZZ_+}\sup_{I\in \mathfrak S_J(\DD_{\tau}^k)}
\norm[\big]{O_{I, J}(v_{M_l}\EE_{m\in D_M^{l}}\mathfrak a_{m}:M\in\DD_{\tau}^k)}_{L^p(X)},
\end{split}
\end{align}
where
\begin{align*}
u_{M_l}:=\frac{|B_M^{l}|}{|B_M^{l-1}|}=\frac{\lfloor M_l\rfloor-\lfloor \tau^{-1}M_l\rfloor}{\lfloor M_l\rfloor}
\quad\text{ and } \quad
v_{M_l}:=\frac{|D_M^{l}|}{|B_M^{l-1}|}=
\frac{\lfloor \tau^{-1}M_l\rfloor}{\lfloor M_l\rfloor}.
\end{align*}
Considering $\tilde{u}_{M_l}:=u_{M_l}-1+\tau^{-1}$ and $\tilde{v}_{M_l}:=v_{M_l}-\tau^{-1}$ we see that
\begin{align*}
\sum_{M_l\in\DD_{\tau}}\tilde{u}_{M_l}^2\lesssim_{\tau}1,
\quad\text{ and } \quad
\sum_{M_l\in\DD_{\tau}}\tilde{v}_{M_l}^2\lesssim_{\tau}1.
\end{align*}
Thus using \eqref{eq:38} one sees that 
\begin{align}
\label{eq:32}
\begin{split}
&\sup_{J\in\ZZ_+}\sup_{I\in \mathfrak S_J(\DD_{\tau}^k)}
\norm[\big]{O_{I, J}(\tilde{u}_{M_l}\EE_{m\in B_M^{l}}\mathfrak a_{m}:M\in\DD_{\tau}^k)}_{L^p(X)}\lesssim_{\tau}
\big\|\sup_{M\in\DD_{\tau}^k}|\EE_{m\in Q_M}\mathfrak a_{m}|\big\|_{L^p(X)},\\
&\sup_{J\in\ZZ_+}\sup_{I\in \mathfrak S_J(\DD_{\tau}^k)}
\norm[\big]{O_{I, J}(\tilde{v}_{M_l}\EE_{m\in D_M^{l}}\mathfrak a_{m}:M\in\DD_{\tau}^k)}_{L^p(X)}\lesssim_{\tau}
\big\|\sup_{M\in\DD_{\tau}^k}|\EE_{m\in Q_M}\mathfrak a_{m}|\big\|_{L^p(X)}.
\end{split}
\end{align}
By \eqref{eq:150} there is $C_{p, \tau}>0$ such that
\begin{align}
\label{eq:33}
\begin{split}
\sup_{J\in\ZZ_+}\sup_{I\in \mathfrak S_J(\DD_{\tau}^k)}
\norm[\big]{O_{I, J}(v_{M_l}\EE_{m\in D_M^{l}}\mathfrak a_{m}:M\in\DD_{\tau}^k)}_{L^p(X)}\le C_{p, \tau}\big\|\sup_{M\in\DD_{\tau}^k}|\EE_{m\in Q_M}\mathfrak a_{m}|\big\|_{L^p(X)}\\
+\sup_{J\in\ZZ_+}\sup_{I\in \mathfrak S_J(\DD_{\tau}^k)}
\norm[\big]{O_{I, J}(v_{M_l}\EE_{m\in B_M^{l-1}}\mathfrak a_{m}:M\in\DD_{\tau}^k)}_{L^p(X)}.
\end{split}
\end{align}
Finally combining \eqref{eq:31}, \eqref{eq:32} and \eqref{eq:33} one obtains the following bootstrap inequality
\begin{multline*}
\sup_{J\in\ZZ_+}\sup_{I\in \mathfrak S_J(\DD_{\tau}^k)}
\norm[\big]{O_{I, J}(\EE_{m\in B_M^{l-1}}\mathfrak a_{m}:M\in\DD_{\tau}^k)}_{L^p(X)}\le C_{p, \tau}\big\|\sup_{M\in\DD_{\tau}^k}|\EE_{m\in Q_M}\mathfrak a_{m}|\big\|_{L^p(X)}\\
+\tau^{-1}\sup_{J\in\ZZ_+}\sup_{I\in \mathfrak S_J(\DD_{\tau}^k)}
\norm[\big]{O_{I, J}(\EE_{m\in B_M^{l-1}}\mathfrak a_{m}:M\in\DD_{\tau}^k)}_{L^p(X)}\\
+
\tau^{-1}(\tau-1)
\sup_{J\in\ZZ_+}\sup_{I\in \mathfrak S_J(\DD_{\tau}^k)}
\norm[\big]{O_{I, J}(\EE_{m\in B_M^{l}}\mathfrak a_{m}:M\in\DD_{\tau}^k)}_{L^p(X)},
\end{multline*}
which  immediately yields
\begin{align}
\label{eq:34}
\begin{gathered}
\sup_{J\in\ZZ_+}\sup_{I\in \mathfrak S_J(\DD_{\tau}^k)}
\norm[\big]{O_{I, J}(\EE_{m\in B_M^{l-1}}\mathfrak a_{m}:M\in\DD_{\tau}^k)}_{L^p(X)}\le C_{p, \tau}\big\|\sup_{M\in\DD_{\tau}^k}|\EE_{m\in Q_M}\mathfrak a_{m}|\big\|_{L^p(X)}\\
+\sup_{J\in\ZZ_+}\sup_{I\in \mathfrak S_J(\DD_{\tau}^k)}
\norm[\big]{O_{I, J}(\EE_{m\in B_M^{l}}\mathfrak a_{m}:M\in\DD_{\tau}^k)}_{L^p(X)}.
\end{gathered}
\end{align}
Iterating \eqref{eq:34} $k$ times and using \eqref{eq:24} to control the maximal function, we conclude that
\begin{align}
\label{eq:134}
\begin{gathered}
\sup_{J\in\ZZ_+}\sup_{I\in \mathfrak S_J(\DD_{\tau}^k)}
\norm[\big]{O_{I, J}(\EE_{m\in Q_M}\mathfrak a_{m}:M\in\DD_{\tau}^k)}_{L^p(X)}\le C_{p, \tau}\big\|\sup_{M\in\DD_{\tau}^k}|\EE_{m\in R_M}\mathfrak a_{m}|\big\|_{L^p(X)}\\
+\sup_{J\in\ZZ_+}\sup_{I\in \mathfrak S_J(\DD_{\tau}^k)}
\norm[\big]{O_{I, J}(\EE_{m\in R_M}\mathfrak a_{m}:M\in\DD_{\tau}^k)}_{L^p(X)}.
\end{gathered}
\end{align}
Finally, using \eqref{eq:134} with $\mathfrak a_{m}(x)=f(T_1^{P_1(m)}\cdots T_d^{P_d(m)}x)$ and invoking Proposition \ref{prop:5} (to control the maximal function from \eqref{eq:134} by oscillation semi-norms) we obtain \eqref{eq:23} as desired.
\end{proof}

Now using Proposition \ref{prop:7} we can reduce the oscillation inequality \eqref{eq:284} from Theorem \ref{thm:main} to establishing the following result for non-degenerate polynomials in the sense of \eqref{eq:66}.

\begin{theorem}
\label{thm:main'}
Let $(X, \mathcal B(X), \mu)$ be a $\sigma$-finite measure space
equipped with an invertible measure-preserving transformation
$T:X\to X$.  Let $P\in\ZZ[\rm m_1, \rm m_2]$ be a non-degenerate
polynomial such that $P(0, 0)=0$.  Let
$\tilde{A}_{M; X}^{P}f$ with
$M=(M_1,M_2)$ be the average defined in \eqref{eq:229'} with $d=1$,
$k=2$, and $P_1 =P$.
If $1<p<\infty$ and $\tau>1$, and $\DD_{\tau}:=\{\tau^n:n\in\NN\}$,  then one has
\begin{align}
\label{eq:284'}
\qquad \qquad\sup_{J\in\ZZ_+}\sup_{I\in\mathfrak S_J(\DD_{\tau}^2) }\big\|O_{I, J}(\tilde{A}_{M_1, M_2; X}^{P}f: M_1, M_2\in\DD_{\tau})\|_{L^p(X)}\lesssim_{p, \tau, P}\|f\|_{L^p(X)}.
\end{align}
 The implicit constant in \eqref{eq:284'} can be taken to depend only on $p, \tau, P$.
\end{theorem}

\subsection{Reduction to the integer shift system}
As mentioned in Example \ref{ex:1} the integer shift system is the most important for pointwise convergence problems. For $T=S_1$, for any $x\in\ZZ$ and for any finitely supported function $f:\ZZ\to\CC$, we may
write
\begin{align}
\label{eq:1}
\tilde{A}_{M_1, M_2; \ZZ, S_1}^{P}f(x)=\EE_{m\in R_{M_1, M_2}}f(x-P(m_1, m_2)).
\end{align}
We shall also abbreviate $\tilde{A}_{M_1, M_2; \ZZ, S_1}^{P}$ to $\tilde{A}_{M_1, M_2; \ZZ}^{P}$.
In fact, we will be able to deduce Theorem \ref{thm:main'} from its integer counterpart.

\begin{theorem}
\label{thm:main''}
Let $P\in\ZZ[\rm m_1, \rm m_2]$ be a non-degenerate polynomial (see \eqref{eq:66}) such that
$P(0, 0)=0$.
Let $\tilde{A}_{M_1, M_2; \ZZ}^{P}f$ be the  average defined  in \eqref{eq:1}.
If $1<p<\infty$ and $\tau>1$, and $\DD_{\tau}:=\{\tau^n:n\in\NN\}$, then one has
\begin{align}
\label{eq:2}
\sup_{J\in\ZZ_+}\sup_{I\in\mathfrak S_J(\DD_{\tau}^2) }\big\|O_{I, J}(\tilde{A}_{M_1, M_2; \ZZ}^{P}f: M_1, M_2\in\DD_{\tau})\|_{\ell^p(\ZZ)}\lesssim_{p, \tau, P}\|f\|_{\ell^p(\ZZ)},
\end{align}
The implicit constant in \eqref{eq:2} can be taken to depend only on $p, \tau, P$.
\end{theorem}

We immediately see that Theorem \ref{thm:main''} is a special case of Theorem \ref{thm:main'}. However, it is also a standard matter, in view of the Calder{\'o}n transference principle \cite{Cald}, that this implication can be reversed and so in order to prove \eqref{eq:284'}, it suffices to establish \eqref{eq:2}. This reduction is important since we can use Fourier methods in the integer setting which are not readily available in abstract measure spaces.

From now on we will focus our attention on establishing Theorem \ref{thm:main''}.

\section{``Backwards'' Newton diagram: Proof of Theorem \ref{thm:main''} }\label{section:4}
The ``backwards'' Newton diagram $N_P$ of a nontrivial polynomial
$P\in\RR[\rm m_1, \rm m_2]$, 
\begin{align}
\label{eq:12}
P(m_1, m_2):=\sum_{\gamma_1, \gamma_2}c_{\gamma_1, \gamma_2}m_1^{\gamma_1}m_2^{\gamma_2}, \quad \text{ with} \quad c_{0, 0}=0,
\end{align}
is defined as the closed convex hull of the set
\begin{align*}
\bigcup_{(\gamma_1, \gamma_2)\in S_P}\{(x+\gamma_1, y+\gamma_2)\in\RR^2: x\le 0, y\le 0\},
\end{align*}
where $S_P:=\{(\gamma_1, \gamma_2)\in\NN\times\NN: c_{\gamma_1, \gamma_2}\neq0\}$ denotes the set of non-vanishing coefficients of $P$.

Let $V_P\subseteq S_P$ be the set of vertices (corner points) of
$N_P$. Suppose that $V_P:=\{v_1, \ldots, v_r\}$ where 
$v_j=(v_{j,1}, v_{j,2})$ satisfies
$v_{j, 1}<v_{j+1, 1}$, and $v_{j+1, 2}<v_{j, 2}$ for
each $j\in[r]$.

Let $\omega_0=(0, 1)$ and
$\omega_r=(1, 0)$ and for $j\in[r-1]$, let
$\omega_j=(\omega_{j, 1}, \omega_{j,2})$ denote a normal vector
to the edge $\overline{v_jv_{j+1}}:=v_{j+1}-v_j$ such that
$\omega_{j, 1}, \omega_{j, 2}$ are positive integers (the choice is not
unique but it is not an issue here). Observe that the slopes of the
lines along $\omega_j$'s are decreasing as $j$ increases
since $N_P$ is convex. The convexity of $N_P$ also yields
that
\begin{align}\label{nd1}
    \omega_j\cdot(v-v_j)\le0 \qquad \mbox{and}\qquad \omega_{j-1}\cdot(v-v_j)\le0 \ \ \ \ ({\rm with \ one \ inequality \ strict}),
\end{align}
for all $v\in S_P\setminus \{v_j\}$ and $j\in[r]$. Now for $j\in[r]$ let us define
\begin{align*}
W(j):=\bigcap_{v\in  S_P\setminus\{v_j\}}\{(a, b)\in \ZZ_+\times\ZZ_+: (a, b)\cdot(v-v_j)<0\},
\end{align*}
which is the intersection of various half planes. If ${\rm V}_P=\{v_1\}$
 then we simply define  $W(1)=\ZZ_+\times\ZZ_+$.

\begin{remark}
\label{rem:3}
Obviously if $1\le i<j\le r$ then $W(i)\cap W(j)=\emptyset$. Indeed, if $(a, b)\in W(i)\cap W(j)$, then $(a, b)\cdot(v-v_i)<0$ 
for all $v\in S_P\setminus\{v_i\}$ and $(a, b)\cdot(v-v_j)<0$ for all $v \in S_P\setminus\{v_j\}$. In particular $(a, b)\cdot(v_j-v_i)<0$ and $(a, b)\cdot(v_i-v_j)<0$ which is impossible.

\end{remark}

 \begin{lemma}\label{lemnd1}
 For $j\in[r]$ we have
 \begin{align*}
    W(j)=\{(a, b)\in \ZZ_+\times\ZZ_+:\  \exists _{\alpha, \beta>0}\ (a, b)=\alpha\omega_{j-1}+\beta\omega_{j} \}.
\end{align*}
 \end{lemma}
 \begin{proof}
 The convexity of $N_P$ implies that the normals $\omega_{j-1}, \omega_{j}$ are  linearly independent, therefore for
every $(a, b)\in \ZZ_+\times\ZZ_+$, 
there are $\alpha, \beta$ such that $(a, b)=\alpha \omega_{j-1}+\beta\omega_{j}$. We only need to show that $(a, b)\in W(j)$ if and only if $\alpha, \beta>0$. Firstly suppose that $(a, b)\in W(j)$. Then $(a, b)\cdot(v-v_j)<0$ for all $v\in  S_P\setminus\{v_j\}$. In particular $(a, b)\cdot(v_{j+1}-v_j)=(\alpha\omega_{j-1}+\beta\omega_{j})\cdot(v_{j+1}-v_j)<0$. But this implies that $\alpha\omega_{j-1}\cdot(v_{j+1}-v_j)<0$, since  $\omega_j\cdot(v_{j+1}-v_j)=0$. This immediately gives that $\alpha>0$, provided that $j\in[r-1]$, since $\omega_{j-1}\cdot(v_{j+1}-v_j)\le0$ by \eqref{nd1}. When $j=r$ then $\alpha>0$ since $\omega_{r}=(1, 0)$ and $0<b=(a, b)\cdot(0, 1)=(\alpha\omega_{r-1}+\beta\omega_{r})\cdot(0, 1)=\alpha\omega_{r-1}\cdot(0, 1)=\alpha \omega_{r-1, 2}$. Similarly taking $v=v_{j-1}$ for $1<j\le r$ we obtain $\beta>0$. When $j=1$ then $\beta>0$ because $\omega_{0}=(0, 1)$ and $0<a=(a, b)\cdot(1, 0)=(\alpha\omega_{0}+\beta\omega_{1})\cdot(1, 0)=\beta\omega_{1}\cdot(1, 0)=\beta \omega_{1,1}$. Conversely, if $\alpha>0$ and $\beta>0$ then for any $v\in S_P\setminus\{v_j\}$ we have
 $(a, b)\cdot(v-v_j)=\alpha\omega_{j-1}\cdot(v-v_j)+\beta\omega_{j}\cdot(v-v_j)<0$, since $\omega_{j-1}\cdot(v-v_j)\le 0$ and $\omega_{j}\cdot(v-v_j)\le 0$, with at least one inequality strict.
\end{proof}
Lemma \ref{lemnd1} means that $W(j)$ consists of those lattice points of $\ZZ_+\times\ZZ_+$ which are within the cone centered at the origin with the boundaries determined by the lines along the normals
$\omega_{j-1}$ and  $\omega_{j}$ respectively.  Now for $j\in[r]$, we set
\begin{align*}
S(j):=
\{(a, b)\in \NN\times\NN:\  \exists_{\alpha\ge0,\beta\ge0}\ (a, b)=\alpha\omega_{j-1}+\beta\omega_{j}\}.
\end{align*}

\begin{remark}
\label{rem:2-0}
Some comments are in order.
\begin{enumerate}[label*={\arabic*}.]
\item Having defined the sets $S(j)$ for $j\in[r]$ it is not difficult to see that
\begin{align}
\label{eq:132}
\bigcup_{j=1}^rS(j)=\NN\times\NN.
\end{align}
\item We note that for $(a, b)\in S(j)$ we have $(a, b)\cdot(v-v_j)\le 0$ for all $v\in S_P$ by \eqref{nd1}. However, the strict inequality may not be achieved even for $v\not=v_j$.
\item If $r\ge2$, then by construction of the sets $S(j)$ one sees that if $(a, b)\in S(j)$ then
\begin{align}
\label{eq:209}
\frac{\omega_{j,2}}{\omega_{j,1}}a\le b\le \frac{\omega_{j-1,2}}{\omega_{j-1,1}}a
\end{align}
for any $1< j< r$; and if $j=1$ or $j=r$ one has respectively
\begin{align}
\label{eq:210}
\frac{\omega_{1,2}}{\omega_{1,1}}a\le b<\infty, \qquad \text{ and } \qquad
0\le b\le \frac{\omega_{r-1,2}}{\omega_{r-1,1}}a.
\end{align}
\item If $r=1$ and $(a, b)\in S(1)$, then $0\le a, b<\infty$. 
\end{enumerate}
\end{remark}

Now for any given $(a, b)\in S(j)$ we try to determine
$\alpha$ and $\beta$ explicitly. Let $A_j:=[\omega_{j-1}|\omega_j]$ be the matrix whose column vectors are the normals $\omega_{j-1}, \omega_j$. Then
\begin{align*}
    \left(
  \begin{array}{c}
    a \\
    b \\
  \end{array}
\right)=
\left(
  \begin{array}{cc}
    \omega_{j-1, 1} & \omega_{j,1} \\
    \omega_{j-1, 2} & \omega_{j, 2} \\
  \end{array}
\right)
\left(
  \begin{array}{c}
    \alpha \\
    \beta \\
  \end{array}
\right).
\end{align*}
The convexity of $N_P$ (and the orientation we chose) ensures that $\det A_j<0$. Taking $d_j:=-\det A_j>0$ one has
\begin{align*}
    \left(
  \begin{array}{c}
    \alpha \\
    \beta \\
  \end{array}
\right)=\frac{1}{\det A_j}
\left(
  \begin{array}{cc}
    \omega_{j,2} & -\omega_{j,1} \\
    -\omega_{j-1, 2} & \omega_{j-1, 1} \\
  \end{array}
\right)
\left(
  \begin{array}{c}
    a \\
    b \\
  \end{array}
\right)=\frac{1}{d_j}
\left(
  \begin{array}{c}
    -a\omega_{j,2} \ +\  b\omega_{j,1} \\
    a\omega_{j-1, 2} - b\omega_{j-1, 1} \\
  \end{array}
\right).
\end{align*}
We have chosen the components of $\omega_{j-1}$ and $\omega_{j}$ to be non-negative integers, therefore for $j\in[r-1]$ (keeping in mind that $\alpha,\beta\ge0$ and $d_j>0$) we may rewrite 
\begin{align*}
S(j)=\{(a, b)\in\NN\times\NN\colon\exists_{(t_1, t_2)\in\NN\times\NN}\ (a, b)=\frac{t_1}{d_j}\omega_{j-1}+\frac{t_2}{d_j}\omega_{j}\}.
\end{align*}
We allow $t_1$ to be zero when $j=r$. 

We now split $S(j)$ into $S_1(j)$ and $S_2(j)$, where
\begin{align*}
    S_1(j):=\{(a, b)\in S(j): (a, b)=\frac{(n+N)}{d_j}\omega_{j-1}+\frac{N}{d_j}\omega_{j},\ n\in\NN, \ N\in\NN\},\\
    S_2(j):=\{(a, b)\in S(j): (a, b)=\frac{N}{d_j}\omega_{j-1}+\frac{(n+N)}{d_j}\omega_{j},\ n\in\NN, \ N\in\NN\}.
\end{align*}
We can further decompose
\begin{align*}
S_1(j)=\bigcup_{N\in\NN}S_1^N(j),
\qquad \text{ and } \qquad
S_2(j)=\bigcup_{N\in\NN}S_2^N(j),
\end{align*}
where
\begin{align}
\label{eq:42}
\begin{split}
    S_1^N(j):=&\{(a, b)\in S(j): (a, b)=\frac{(n+N)}{d_j}\omega_{j-1}+\frac{N}{d_j}\omega_{j},\ n\in\NN\},\\
    S_2^N(j):=&\{(a, b)\in S(j): (a, b)=\frac{N}{d_j}\omega_{j-1}+\frac{(n+N)}{d_j}\omega_{j},\ n\in\NN\}.
    \end{split}
\end{align}
\begin{lemma}
\label{lem:30}
For each $j\in[r]$  there exists $\sigma_j>0$ such that for every $v\in S_P\setminus\{v_j\}$ one has
\begin{align}
\label{eq:47}
(a, b)\cdot (v-v_j)
\le -\sigma_j N
\end{align}
for all $(a, b)\in S_1^N(j)$.
The same conclusion is true for  $S_2^N(j)$.
\end{lemma}
\begin{proof}
For every $(a, b)\in S_1^N(j)$ we can write
\begin{align*}
(a, b)=\frac{(n+N)}{d_j}\omega_{j-1}+\frac{N}{d_j}\omega_{j}
=\frac{n}{d_j}\omega_{j-1}+\frac{N}{d_j}(\omega_{j}+\omega_{j-1})
\end{align*}
for some $n\in\NN$. By \eqref{nd1} we have
\begin{align*}
(v-v_j)\cdot (\omega_{j}+\omega_{j-1})<0
\end{align*}
for all $v\in S_P\setminus\{v_j\}$,
since $\omega_{j-1}$ and $\omega_{j}$ are linearly independent. Taking
\begin{align*}
\sigma_j:=\min_{v\in S_P\setminus\{v_j\}}\frac{1}{d_j}(v_j-v)\cdot (\omega_{j}+\omega_{j-1})>0,
\end{align*}
one sees, by \eqref{nd1} again, that
\begin{align*}
(a, b)\cdot (v-v_j)
=\frac{n}{d_j}\omega_{j-1}\cdot (v-v_j)+\frac{N}{d_j}(\omega_{j}+\omega_{j-1})\cdot (v-v_j)\le -\sigma_j N
\end{align*}
for all $(a, b)\in S_1^N(j)$. This immediately yields \eqref{eq:47} and the proof is finished.
\end{proof}
For any $\tau>1$ using the decomposition \eqref{eq:132} we may write
\begin{align}
\label{eq:5}
\DD_{\tau}\times\DD_{\tau}=\bigcup_{j=1}^r\SS_{\tau}(j),
\end{align}
where
\begin{align}
\label{eq:6}
\SS_{\tau}(j):=\{(\tau^{n_1}, \tau^{n_2})\in \DD_{\tau}\times\DD_{\tau}: (n_1, n_2)\in S(j)\},
\quad \text{ for } \quad j\in[r].
\end{align}
 Using \eqref{eq:42} we can further write
 \begin{align}
 \label{eq:48}
\SS_{\tau}(j)=\bigcup_{N\in\NN}\SS_{\tau, 1}^N(j)\cup\bigcup_{N\in\NN}\SS_{\tau, 2}^N(j),
\end{align}
where for any $j\in[r]$ one has
\begin{align}
\label{eq:49}
\begin{split}
\SS_{\tau, 1}^N(j):=&\{(\tau^{n_1}, \tau^{n_2})\in \DD_{\tau}\times\DD_{\tau}: (n_1, n_2)\in S_1^N(j) \},\\
\SS_{\tau, 2}^N(j):=&\{(\tau^{n_1}, \tau^{n_2})\in \DD_{\tau}\times\DD_{\tau}: (n_1, n_2)\in S_2^N(j) \}.
\end{split}
\end{align}

In view of decomposition \eqref{eq:5} our aim will be to restrict the estimates for oscillations to sectors from \eqref{eq:6}.

\begin{theorem}
\label{thm:main'''}
Let $P\in\ZZ[\rm m_1, \rm m_2]$ be a non-degenerate polynomial (see \eqref{eq:66}) such that
$P(0, 0)=0$. Let $r\in\ZZ_+$ be the number of corners in the corresponding Newton diagram $N_P$.
Let $f\in \ell^p(\ZZ)$ for some $1\le  p\le \infty$, and let $\tilde{A}_{M_1, M_2; \ZZ}^{P}f$ be the  average defined  in \eqref{eq:1}.
If $1<p<\infty$ and $\tau>1$ and $j\in[r]$, and $\SS_{\tau}(j)$ is a sector from \eqref{eq:6}, then one has
\begin{align}
\label{eq:3}
\sup_{J\in\ZZ_+}\sup_{I\in\mathfrak S_J(\SS_{\tau}(j))}\big\|O_{I, J}(\tilde{A}_{M_1, M_2; \ZZ}^{P}f: (M_1, M_2)\in\SS_{\tau}(j))\|_{\ell^p(\ZZ)}\lesssim_{p, \tau, P}\|f\|_{\ell^p(\ZZ)}.
\end{align}
 The implicit constant in \eqref{eq:3} may only depend on $p, \tau, P$.
\end{theorem}

The proof of Theorem \ref{thm:main'''} is postponed to Section \ref{section:7}. However, assuming momentarily Theorem \ref{thm:main'''} we can derive Theorem \ref{thm:main''}.
\begin{proof}[Proof of Theorem \ref{thm:main''}] 
Assume that \eqref{eq:3} holds for all $j\in[r]$. By \eqref{eq:5} and \eqref{eq:116} one has 
\begin{multline*}
\sup_{J\in\ZZ_+}\sup_{I\in\mathfrak S_J(\DD_{\tau}^2) }\big\|O_{I, J}(\tilde{A}_{M; \ZZ}^{P}f: M\in\DD_{\tau}^2)\|_{\ell^p(\ZZ)}
\lesssim \sum_{j\in[r]}\sup_{J\in\ZZ_+}\sup_{I\in\mathfrak S_J(\DD_{\tau}^2) }\big\|O_{I, J}(\tilde{A}_{M; \ZZ}^{P}f: M\in\SS_{\tau}(j))\|_{\ell^p(\ZZ)}.
\end{multline*}
\paragraph{\bf Step 1} If suffices to show that for every $j\in[r]$, every $J\in\ZZ_+$ and every $I\in\mathfrak S_J(\DD_{\tau}^2)$, one has
\begin{multline}
\label{eq:7}
\Big\|\Big(\sum_{i\in\NN_{<J}}\sup_{M\in\BB[I,i]\cap\SS_{\tau}(j)}|\tilde{A}_{M; \ZZ}^{P}f-\tilde{A}_{I_{i}; \ZZ}^{P}f|^2\Big)^{1/2}\Big\|_{\ell^p(\ZZ)}\\
\lesssim \sum_{j\in[r]}
\sup_{J\in\ZZ_+}\sup_{I\in\mathfrak S_J(\SS_{\tau}(j))}\big\|O_{I, J}(\tilde{A}_{M; \ZZ}^{P}f: M\in\SS_{\tau}(j))\big\|_{\ell^p(\ZZ)}.
\end{multline}
We can assume that $J>Cr$ for a large $C>0$, otherwise the estimate in
\eqref{eq:7} easily follows from maximal function estimates.  Let us fix a sequence
$I = (I_0, \ldots, I_J) \in \mathfrak S_J(\DD_{\tau}^2)$ and a sector
$\SS_{\tau}(j)$. Let
$\omega_*:=\max\{\omega_{i1}, \omega_{i2}: i\in[r]\}$ and we split the set
$\NN_{<J}$  into $O(r)$ sparse sets $\JJ_1,\ldots, \JJ_{O(r)}\subset \NN_{<J}$, where each
$\JJ\in\{\JJ_1,\ldots, \JJ_{O(r)}\}$ satisfies the separation condition:
\begin{align}
\label{eq:8}
\log_{\tau}I_{i_21}-\log_{\tau}I_{(i_1+1)1}\ge 100r\omega_*
\quad \text{ and } \quad
\log_{\tau}I_{i_22}-\log_{\tau}I_{(i_1+1)2}\ge 100r\omega_*
\end{align}
for every $i_1, i_2\in\JJ$ such that $i_1< i_2$. Our task now is to establish \eqref{eq:7} with the 
summation over $\JJ$ satisfying \eqref{eq:8}  in place of $\NN_{<J}$ in the sum on the left-hand side of \eqref{eq:7}.

\paragraph{\bf Step 2}  To every element $I_i = (I_{i1}, I_{i2})$ with
$i\in\NN_{<J}$ in the sequence $I$ (which say lies in the sector
$\SS_{\tau}(j_i)$), we associate at most one point
$P_i(j) \in \SS_{\tau}(j)$ in the following way. If $j_i<j$ and the
box $\BB[I,i]$ intersects the sector $\SS_{\tau}(j)$, then the box
intersects the sector along the bottom edge. We set
$P_i(j) = (I_i^j, I_{i2})$ where $I_i^j$ be the least element in
$\DD_{\tau}$ such that $(I_i^j, I_{i2}) \in \SS_{\tau}(j)$. If $j<j_i$
and the box $\BB[I,i]$ intersects the sector $\SS_{\tau}(j)$, then it
intersects the sector along the left edge. We set
$P_i(j) = (I_{i1}, {\tilde{I}}_i^j)$ where ${\tilde{I}}_i^j$ be the
least element in $\DD_{\tau}$ such that
$(I_{i1}, {\tilde{I}}_i^j) \in \SS_{\tau}(j)$.  Finally if $j_i=j$, we
set $P_i(j) = I_i$. The sequence $P(j) = (P_i(j): i\in\NN_{\le J'})$ forms a strictly increasing sequence 
lying in $\mathfrak S_{J'}(\SS_{\tau}(j))$ for some $J'\le J$ and each
$P_i(j)=(P_{i1}(j), P_{i2}(j))$ is the least element among all the
elements $(M_1,M_2) \in \BB[I,i] \cap \SS_{\tau}(j)$.

\paragraph{\bf Step 3} We now produce a sequence of length at most $r+2$, which will allow us to move from $I_i$ to $P_i(j)$ when $I_i\neq P_i(j)$.
More precisely, we claim that there exists a sequence $u^i:=(u^i_m : m\in \NN_{<m_{I_i}})\subset \DD_{\tau}^2$ for some $m_{I_i}\in[r+1]$, with the property that
\begin{align}
\label{eq:16}
u_0^i\succ u_1^i\succ\ldots \succ u_{m_{I_i}-1}^i,
\quad \text{ and } \quad u_{m_{I_i}-1}^i\prec u_{m_{I_i}}^i,
\end{align}
where $(u_0^i, u_{m_{I_i}}^i)=(I_i, P_i(j))$ or  $(u_0^i, u_{m_{I_i}}^i)=(P_i(j), I_i)$.
Moreover,  two consecutive elements $u_m^i, u_{m+1}^i$ of this sequence belong to a unique sector $\SS_{\tau}(j_{u_m^i})$ except the elements  $u_{m_{I_i}-2}^i, u_{m_{I_i}-1}^i$ and $u_{m_{I_i}-1}^i, u_{m_{I_i}}^i$, which may belong to the same sector. Suppose now that $\BB[I,i] \cap \SS_{\tau}(j)\neq\emptyset$ and $I_i\in \SS_{\tau}(j_i)$ and $j_i< j$.  Let $u_{0}^i:=I_i$ be the starting point. Suppose that the elements $u_{0}^i\succ u_{1}^i\succ \ldots\succ u_{m}^i$  have been chosen for some $m\in\NN_{< r}$ so that $u_{s}^i$ lies on  the bottom boundary ray of $\SS_{\tau}(j_{i}+s-1)$ and $u_{s}^i\prec u_{s-1}^i$ for each $s\in[m]$. Then we take $u_m^i$ and move southwesterly to  $u_{m+1}^i$, 
the nearest  point on the bottom boundary ray of $\SS_{\tau}(j_i+m)$ such that $u_{m+1}^i\prec u_{m}^i$. Continuing this way after $m_{I_i}-1=j-j_i+1\le r$ steps we arrive at $u_{m_{I_i}-1}^i\in \SS_{\tau}(j)$ which will allow us to reach the last point   of this sequence  $u_{m_{I_i}}^i:=P_i(j)$ as claimed in \eqref{eq:16}. Assume now that $\BB[I,i] \cap \SS_{\tau}(j)\neq\emptyset$ and $I_i\in \SS_{\tau}(j_i)$ and $j_i> j$. We start from the point $u_0^i:=P_i(j)$ and proceed exactly the same  as in the previous case until we reach the point $u_{m_{I_i}}^i:=I_i$.
\paragraph{\bf Step 4} To complete the proof we use the sequence from \eqref{eq:16} for each $i\in\JJ$, and observe that
\begin{multline*}
\Big\|\Big(\sum_{i\in\JJ}\sup_{M\in\BB[I,i]\cap\SS_{\tau}(j)}|\tilde{A}_{M; \ZZ}^{P}f-\tilde{A}_{I_{i}; \ZZ}^{P}f|^2\Big)^{1/2}\Big\|_{\ell^p(\ZZ)}
\le
\Big\|\Big(\sum_{i\in\JJ}\sup_{M\in\BB[P(j), i]\cap\SS_{\tau}(j)}|\tilde{A}_{M; \ZZ}^{P}f-\tilde{A}_{I_{i}; \ZZ}^{P}f|^2\Big)^{1/2}\Big\|_{\ell^p(\ZZ)}\\
\lesssim_r
\Big\|\Big(\sum_{i\in\JJ}\sup_{M\in\BB[P(j), i]\cap\SS_{\tau}(j)}|\tilde{A}_{M; \ZZ}^{P}f-\tilde{A}_{P_{i}(j); \ZZ}^{P}f|^2\Big)^{1/2}\Big\|_{\ell^p(\ZZ)}\\
+\Big\|\Big(\sum_{i\in\JJ}\sum_{m\in \NN_{<m_{I_i}}}|\tilde{A}_{u_{m+1}^i; \ZZ}^{P}f-\tilde{A}_{u_m^i; \ZZ}^{P}f|^2\Big)^{1/2}\Big\|_{\ell^p(\ZZ)}.
\end{multline*}
Clearly, the first norm is dominated by the right-hand side of \eqref{eq:7}. The same is true for the second norm. It follows from the fact that for two consecutive integers $i_1< i_2$ such that $\BB[I,{i_1}] \cap \SS_{\tau}(j)\neq\emptyset$ and 
$\BB[I,{i_2}] \cap \SS_{\tau}(j)\neq\emptyset$, if we have $u_{j_1}^{i_1}$ and $u_{j_2}^{i_2}$ belonging to the same sector, they must satisfy $u_{j_1}^{i_1}\prec u_{j_2}^{i_2}$ by the separation condition \eqref{eq:8}.
 This complete the proof of the theorem.
\end{proof}

\section{Exponential sum estimates}\label{section:5}
This section is intended to establish certain double exponential sum estimates which will be used later. We begin by recalling the classical  Weyl inequality with a logarithmic loss.

\begin{proposition}
\label{prop:20}
Let $d\in\ZZ_+$, $d\ge2$ and let $P\in\RR[\rm m]$ be such that $P(m):=c_dm^d+\ldots+c_1m$.
Then there exists a constant $C>0$ such that for every $M\in\ZZ_+$ the following is true.
Suppose that for some $2\le j\le d$ there are  $a, q\in\ZZ$ such that $1\le q\le M^j$ and  $(a, q)=1$ and  
\begin{align*}
\Big|c_j-\frac a q\Big|\le\frac 1 {q^2}.
\end{align*}
Then for $\sigma(d):=2d^2-2d+1$ one has
\begin{align}
\label{eq:w4}
\Big|\sum_{m=1}^M \ex(P(m))\Big|\le
CM\log(2M)\bigg(\frac{1}{q}+\frac{1}{M}+\frac q{M^j}\bigg)^{\frac{1}{\sigma(d)}}.
\end{align}
\end{proposition}
For the proof we refer to \cite[Theorem 1.5]{W}. The range of summation in \eqref{eq:w4} can be shifted to any segment of length $M$ without affecting the bound. We will also recall a simple lemma from \cite[Lemma A.15, p. 53]{MSZ3}, see also \cite[Lemma 1, p. 1298]{SW2}, which  follows from the Dirichlet principle.

\begin{lemma}
\label{lem:21}
Let $\theta\in\RR$ and $Q\in\ZZ\setminus\{0\}$.
Suppose that
\[
\abs[\Big]{\theta - \frac{a}{q}}
\leq
\frac{1}{q^2}
\]
for some integers $0\le a<q \leq M$
with $(a,q)=1$ for some $M\ge1$.
Then there is a reduced fraction $a'/q'$ so that $(a', q') = 1$ and
\[
\abs[\Big]{Q \theta - \frac{a'}{q'}}
\leq
\frac{1}{2q'M}
\]
with $q/(2|Q|) \leq q' \leq 2 M$.
\end{lemma}

We now extend Weyl's inequality in Proposition \ref{prop:20} to include the $j=1$ case. 

\begin{proposition}
\label{prop:20'}
Let $d\in\ZZ_+$ and let $P\in\RR[\rm m]$ be such that $P(m):=c_dm^d+\ldots+c_1m$.
Then there exists a constant $C>0$ such that for every $M\in\ZZ_+$ the following is true.
Suppose that for some $1\le j\le d$ there are  $a, q\in\ZZ$ such that $1\le q\le M^j$ and  $(a, q)=1$ and  
\begin{align}
\label{eq:27}
\Big|c_j-\frac a q\Big|\le\frac 1 {q^2}.
\end{align}
Then for certain $\tau(d)\in\ZZ_+$ one has
\begin{align}
\label{eq:w4'}
\Big|\sum_{m=1}^M \ex(P(m))\Big|\le
CM\log(2M)\bigg(\frac{1}{q}+\frac{1}{M}+\frac q{M^j}\bigg)^{\frac{1}{\tau(d)}}.
\end{align}
\end{proposition}

\begin{proof}
We first assume that $d=1$. Then $P(m)=c_1m$ and $j=1$. We can also assume that
$q\ge 2$ otherwise \eqref{eq:w4'} is obvious. Now it is easy to see that
\begin{align*}
\Big|\sum_{m=1}^M \ex(c_1m)\Big|\le \frac{1}{\|c_1\|}\lesssim q.
\end{align*}
Thus \eqref{eq:w4'} holds with $\tau(1)=1$. Now we assume that
$d\ge 2$. If \eqref{eq:27} holds for some $2\le j\le d$ then
\eqref{eq:w4'} follows from Proposition \ref{prop:20} with
$\tau(d)=\sigma(d)$, where $\sigma(d)$ is the exponent as in
\eqref{eq:w4}. Hence we can assume that $j=1$. Define
$\kappa:=\min\{q, M/q\}$, and let $\chi\in(0, (4d)^{-1})$. We may 
assume that $\kappa>100$ otherwise \eqref{eq:w4'} obviously follows.
For every $2\le j'\le d$, by Dirichlet's principle, there is a reduced
fraction $a_{j'}/q_{j'}$ such that
\begin{align}\label{cj}
\Big|c_{j'}-\frac{a_{j'}}{ q_{j'}}\Big|\le\frac{\kappa^{\chi}}{q_{j'} M^{j'}}
\end{align}
with $(a_{j'}, q_{j'})$ and $1\le q_j'\le M^{j'}\kappa^{-\chi}$. We may assume
that $1\le q_{j'}\le \kappa^{\chi}$ for all $2\le j'\le d$, since
otherwise the claim follows  from \eqref{eq:w4} with
$\tau(d)=\lceil \sigma(d)\chi^{-1}\rceil$. Let
$Q:=\lcm\{q_{j'}: 2\le j'\le d\}\le \kappa^{d\chi}$ and note that $Q \le M$ follows from the definition of $\kappa$.
We have
\begin{align*}
\Big|\sum_{m=1}^M \ex(P(m))\Big|&\le
\sum_{r=1}^Q\Big|\sum_{-\frac{r}{Q}< \ell\le \frac{M-r}{Q}} \ex(P(Q\ell+r))\Big|\\
&=\sum_{r=1}^Q\Big|\sum_{U< \ell\le V} A_{\ell} B_{\ell}\Big|,
\end{align*}
where $U:=-\frac{r}{Q}$, $V:=\frac{M-r}{Q}$ and $A_{\ell}:=\ex(c_1Q\ell)$ and
$$
B_{\ell}:=\ex\big(\sum_{j'=2}^d c_{j'} (Q\ell+r)^{j'} \big) \ = \ \ex\big(\sum_{j'=2}^d \alpha_{j'} (Q\ell + r)^{j'} +
\sum_{j'=2}^d \frac{a_{j'}}{q_{j'}} r^{j'}\big)
$$ 
where $\alpha_{j'} := c_{j'} - a_{j'}/q_{j'}$ satisfies the estimate \eqref{cj}.
Using the summation by parts
formula \eqref{eq:294} we obtain
\begin{align*}
\sum_{U < \ell\le V} A_{\ell} B_{\ell}=S_{V}B_{\lfloor V\rfloor} + \sum_{\ell\in(U, V-1]\cap\ZZ}S_{\ell}(B_{\ell}-B_{\ell+1}),
\end{align*}
with $S_{\ell}:=\sum_{k\in(U, \ell]\cap\ZZ}A_k$. 

From above, since $Q \le M$, we see that
\begin{align*}
|B_{\ell+1}-B_{\ell}| \lesssim \kappa^{\chi}QM^{-1}.
\end{align*}
By Lemma \ref{lem:21}  (with $M=q$) there is a reduced fraction $a'/q'$ such that $(a', q')=1$ and
\begin{align*}
\Big|c_1Q-\frac{a'}{q'}\Big|\le\frac{ 1 }{2qq'}
\qquad \text{ and } \qquad \kappa^{1-d\chi}/2\le q'\le 2q\le 2M/\kappa.
\end{align*}
Hence $q'\ge \kappa^{1-d\chi}/2\ge 2$ and so
\begin{align*}
|S_{\ell}|\lesssim \frac{1}{\|c_1Q\|}\lesssim q'\lesssim M/\kappa.
\end{align*}
Consequently, we conclude that
\begin{align*}
\Big|\sum_{m=1}^M \ex(P(m))\Big|\lesssim M\kappa^{-1/2}.
\end{align*}
This implies \eqref{eq:w4'} with $\tau(d)=2$ and the proof of Proposition \ref{prop:20'} is complete.
\end{proof}

We shall also use the Vinogradov mean value theorem. A detailed exposition of Vinogradov's method can be found in \cite[Section 8.5, p. 216]{IK}, see also \cite{W}. We shall follow \cite{IK}.  For each integers $s\ge1$ and   $k, N\ge2$ and for  $\lambda_1,\ldots, \lambda_k\in\ZZ$
let $J_{s, k}(N; \lambda_1, \ldots, \lambda_k)$
denote the number of solutions to the system of $k$ inhomogeneous equations in $2s$ variables given by
\begin{align}
\label{eq:57}
\left\{
\begin{array}{c}
  x_1+\ldots+x_s-y_1-\ldots-y_s=\lambda_1\\
 x_1^2+\ldots+x_s^2-y_1^2-\ldots-y_s^2=\lambda_2\\
   \vdots  \\
 x_1^k+\ldots+x_s^k-y_1^k-\ldots-y_s^k=\lambda_k,
\end{array}
\right.
\end{align}
where $x_j, y_j\in[N]$ for every $j\in[s]$. 
The number $J_{s, k}(N; \lambda_1, \ldots, \lambda_k)$ can be expressed in terms of a certain exponential sum. Let $R_k(x):=(x, x^2, \ldots, x^k)\in\RR^k$ denote the moment curve for $x\in\RR$. For  $\xi=(\xi_1,\ldots, \xi_k)\in\RR^k$ define the
exponential sum
\begin{align*}
S_k(\xi; N):=\sum_{n=1}^N\ex(\xi\cdot R_k(n))=\sum_{n=1}^N\ex(\xi_1n+\ldots+\xi_kn^k).
\end{align*}
One easily obtains
\begin{align}
  \label{eq:w1}
|S_k(\xi; N)|^{2s}=\sum_{|\lambda_1|\le sN}\ldots\sum_{|\lambda_k|\le sN^k}J_{s, k}(N; \lambda_1, \ldots, \lambda_k)\ex(\xi\cdot\lambda),
\end{align}
which by the Fourier inversion formula gives
\begin{align}
  \label{eq:w9}
J_{s, k}(N; \lambda_1, \ldots, \lambda_k)=\int_{[0,
  1)^k}|S_k(\xi; N)|^{2s}\ex(-\xi\cdot\lambda)d\xi.
\end{align}
 Moreover, from \eqref{eq:w9} one has
\begin{equation}\label{eq:w9.5}
J_{s, k}(N;\lambda_1, \ldots, \lambda_k)\le J_{s, k}(N):=J_{s, k}(N; 0, \ldots, 0),
\end{equation}
where the number $J_{s, k}(N)$ represents the number of solutions to the system of
$k$ homogeneous equations in $2s$ variables as in \eqref{eq:57} with $\lambda_1=\ldots=\lambda_k=0$.

Vinogradov's mean value theorem can be formulated as follows:
\begin{theorem}
\label{thm:vino}
For all integers $s\ge1$ and $k\ge2$ and any $\varepsilon>0$ there is a constant $C_{\varepsilon}>0$ such that for every integer $N\ge2$ one has
\begin{align}
\label{eq:w10}
J_{s, k}(N)\le C_{\varepsilon}\big(N^{s+\varepsilon}+N^{2s-\frac{k(k+1)}{2}+\varepsilon}\big).
\end{align}
Moreover, if additionally $s>\frac{1}{2}k(k+1)$ then there is a constant $C>0$ such that
\begin{align}
\label{eq:w10'}
J_{s, k}(N)\le C N^{2s-\frac{k(k+1)}{2}}.
\end{align}
\end{theorem}
Apart from the $N^{\varepsilon}$ loss in \eqref{eq:w10}, this bound
is known to be sharp. Inequality \eqref{eq:w10} is fairly simple
for $k=2$ and follows from elementary estimates for the divisor
function. The conclusion of Theorem \ref{thm:vino} for $k\ge3$, known as Vinogradov's mean value theorem, 
was a central problem in analytic number theory and 
had been open until recently.  The cubic case $k=3$ was solved by Wooley
\cite{Wol} using the efficient congruencing method. The case for any
$k\ge3$ was solved by the first author with Demeter and Guth \cite{BDG}
using the decoupling method. Not long afterwards, Wooley \cite{Wol1} also showed that the efficient congruencing method  can be used to solve the Vinogradov mean value conjecture for all $k\ge3$.   In fact, later we will only use
\eqref{eq:w10'}, which easily follows from \eqref{eq:w10}, the details can be found in \cite[Section 5]{BDG}.

\subsection{Double Weyl's inequality} Let $K_1, K_2\in\NN$, $M_1, M_2\in\ZZ_+$ satisfy $K_1< M_1$ and $K_2< M_2$. Let $Q\in\RR[\rm m_1, \rm m_2]$ be given 
and define double exponential sums by
\begin{align}
\label{eq:58}
S_{K_1, M_1, K_2, M_2}(Q):=&\sum_{m_1=K_1+1}^{M_1}\sum_{m_2=K_2+1}^{M_2}\ex(Q(m_1, m_2)),\\
\label{eq:87}
S_{K_1, M_1, K_2, M_2}^1(Q):=&\sum_{m_1=K_1+1}^{M_1}\Big|\sum_{m_2=K_2+1}^{M_2}\ex(Q(m_1, m_2))\Big|,\\
\label{eq:213}
S_{K_1, M_1, K_2, M_2}^2(Q):=&\sum_{m_2=K_2+1}^{M_2}\Big|\sum_{m_1=K_1+1}^{M_1}\ex(Q(m_1, m_2))\Big|.
\end{align}
If $K_1=K_2=0$ we will abbreviate \eqref{eq:58}, \eqref{eq:87} and \eqref{eq:213} respectively to
\begin{align}\label{00}
S_{M_1,  M_2}(Q),
\quad   \quad
S_{M_1,  M_2}^1(Q),
\quad  \text{ and } \quad
S_{M_1,  M_2}^2(Q).
\end{align}
By the triangle inequality we have
\begin{align}
\label{eq:226}
|S_{K_1, M_1, K_2, M_2}(Q)|\le S_{K_1, M_1, K_2, M_2}^1(Q),
\quad \text{ and } \quad
|S_{K_1, M_1, K_2, M_2}(Q)|\le S_{K_1, M_1, K_2, M_2}^2(Q).
\end{align}
We now provide  estimates for \eqref{eq:58}, \eqref{eq:87} and \eqref{eq:213} in the spirit of Proposition \ref{prop:20} above.
We first recall a technical lemma from \cite[Chapter IV, Lemma 5, p. 82]{Kar}.

\begin{lemma}
\label{lem:23}
Let $\alpha\in\RR$ and suppose that there are $a\in\ZZ, q\in\ZZ_+$ such that  $(a, q)=1$ and
\[
\Big| \alpha  \ - \ \frac{a}{q} \Big| \ \le \ \frac{1}{q^2}.
\]
Then for every $\beta \in \RR$, $U>0$ and $P\ge1$ one has
\begin{align}
\label{eq:w15}
\sum_{n=1}^P\min\bigg\{U, \frac{1}{\|\alpha n+\beta\|}\bigg\}\le
6\bigg(1+\frac{P}{q}\bigg)(U+q\log q).
\end{align}
\end{lemma}
Estimate \eqref{eq:w15} will be useful in the proof of the following counterpart of Weyl's inequality for double sums. 
\begin{proposition}
\label{prop:22}
Let $d_1, d_2\in\ZZ_+$ and $Q\in\RR[\rm m_1, \rm m_2]$ be such that 
\begin{align*}
Q(m_1, m_2):=\sum_{\gamma_1=0}^{d_1}\sum_{\gamma_2=0}^{d_2}
c_{\gamma_1, \gamma_2}m_1^{\gamma_1}m_2^{\gamma_2}, \quad \text{ and } \quad c_{0, 0}=0.
\end{align*}
Then there exists a constant $C>0$ such that for every
$K_1, K_2\in\NN$, $M_1, M_2\in\ZZ_+$ satisfying $K_1\le M_1$ and
$K_2\le M_2$ the following holds.  Suppose that for some
$1 \le \rho_1\le d_1$ and $1\le \rho_2\le d_2$ there are
$a_{\rho_1, \rho_2}\in\ZZ, q_{\rho_1, \rho_2}\in\ZZ_+$ such that
$(a_{\rho_1, \rho_2}, q_{\rho_1, \rho_2})=1$ and
\begin{align}
\label{eq:w6}
\Big|c_{\rho_1, \rho_2}-\frac{a_{\rho_1, \rho_2}}{q_{\rho_1, \rho_2}}\Big|\le\frac{1}{q_{\rho_1, \rho_2}^2}.
\end{align}
Set $k_i:=d_i(d_i+1)$ for $i\in[2]$, $M_{-} := \min(M_1^{\rho_1}, M_2^{\rho_2})$ and $M_{+} := \max(M_1^{\rho_1}, M_2^{\rho_2})$.
Then for $i\in [2]$,
\begin{align}
\label{eq:86}
S_{K_1, M_1, K_2, M_2}^i(Q)
&\le CM_1M_2\bigg( \frac{1}{M_{-}}+\frac{q_{\rho_1, \rho_2}\log q_{\rho_1, \rho_2}}{M_1^{\rho_1}M_2^{\rho_2}}+\frac{1}{q_{\rho_1, \rho_2}}+\frac{\log q_{\rho_1, \rho_2}}{M_{+}}\bigg)^{\frac{1}{4k_1k_2}}.
\end{align}
In view of \eqref{eq:226} estimates \eqref{eq:86} clearly hold for $|S_{K_1, M_1, K_2, M_2}(Q)|$.
\end{proposition}

\begin{remark}
\label{rem:2-0'} The bracketed expression in \eqref{eq:86} is equal to $\min(A,B)$ where
$$
A \ = \ \frac{1}{M_2^{\rho_2}}+\frac{q_{\rho_1, \rho_2}\log q_{\rho_1, \rho_2}}{M_1^{\rho_1}M_2^{\rho_2}}+\frac{1}{q_{\rho_1, \rho_2}}+\frac{\log q_{\rho_1, \rho_2}}{M_1^{\rho_1}}
$$
and
$$
B \ = \ \frac{1}{M_1^{\rho_1}}+\frac{q_{\rho_1, \rho_2}\log q_{\rho_1, \rho_2}}{M_1^{\rho_1}M_2^{\rho_2}}+\frac{1}{q_{\rho_1, \rho_2}}+\frac{\log q_{\rho_1, \rho_2}}{M_2^{\rho_2}}.
$$
\end{remark}

Multi-parameter exponential sums were extensively investigated over the years. The best source about this subject is \cite{ACK}. However, here we need bounds as in \eqref{eq:86}, which will allow us to 
gain logarithmic factors on minor arcs (see Proposition \ref{prop:23}) in contrast to polynomial factors, which were obtained in \cite{ACK}. We prove Proposition \ref{prop:22} by giving an argument based on an iterative application of the Vinogradov mean value theorem.

\begin{proof}[Proof of Proposition \ref{prop:22}]
We only prove \eqref{eq:86} for $i=1$. The proof of \eqref{eq:86} for $i=2$ can be obtained similarly by symmetry. 
To prove inequality \eqref{eq:86} when $i=1$ we shall follow \cite[Section 8.5., p. 216]{IK} and proceed in five steps.
\paragraph{\bf Step 1.}
For $i\in[2]$ let us define the $d_i$-dimensional box
\[
\mathcal B_{d_i}(M_{i}):=\Big(\prod_{j=1}^{d_i}[-k_iM_i^{j}, k_iM_i^{j}]\Big)\cap\ZZ^{d_i}.
\] 
Observe that
\begin{align*}
Q(m_1, m_2)=
\sum_{\gamma_2=0}^{d_2}c_{\gamma_2}(m_1)m_2^{\gamma_2}=c(m_1)\cdot R_{d_2}(m_2)+c_0(m_1),
\end{align*}
where for $\gamma_2\in[d_2]\cup\{0\}$ one has
\begin{align*}
c(m_1):=(c_1(m_1),\ldots, c_{d_2}(m_1)) \qquad\text{ and } \qquad
c_{\gamma_2}(m_1):=\sum_{\gamma_1=0}^{d_1}
c_{\gamma_1, \gamma_2}m_1^{\gamma_1}.
\end{align*}
Recall that $R_{d_2}(m_2) = (m_2, m_2^2, \ldots, m_2^{d_2})$.
By \eqref{eq:87} we note that
\begin{align*}
S_{K_1, M_1, K_2, M_2}^1(Q)
\le S_{M_1, M_2}^1(Q)+S_{M_1,K_2}^1(Q)
\lesssim \max_{N_2\in[M_2]}S_{M_1,N_2}^1(Q).
\end{align*}
For any $k_2\in\ZZ_+$, by
H\"older's inequality and by \eqref{eq:w1},  we obtain
\begin{align}
\label{eq:w12}
\begin{split}
S_{K_1, M_1, K_2, M_2}^1(Q)^{2k_2}&\lesssim M_1^{2k_2-1}\max_{N_2\in[M_2]}
\sum_{m_1=1}^{M_1}|S_{d_2}(c(m_1); N_2)|^{2k_2}\\
&=M_1^{2k_2-1}\max_{N_2\in[M_2]}\sum_{u\in \mathcal B_{d_2}(N_2)}
J_{k_2, d_2}(N_2; u) \sum_{m_1=1}^{M_1}
\ex(c(m_1)\cdot u).
\end{split}
\end{align}
\paragraph{\bf Step 2.} We see that
\[
c(m_1)\cdot u=\sum_{\gamma_1=0}^{d_1}
\sum_{\gamma_2=1}^{d_2}c_{\gamma_1,
  \gamma_2}u_{\gamma_2}m_1^{\gamma_1}
=\beta^1(u)\cdot R_{d_1}(m_1)+\beta_0^1(u),
\]
where for $u=(u_1,\ldots, u_{d_2})\in\ZZ^{d_2}$ and $\gamma_1\in[d_1]\cup\{0\}$ we set
\[
\beta^1(u):=(\beta_1^1(u),\ldots, \beta_{d_1}^1(u))\qquad\text{ and } \qquad\beta_{\gamma_1}^1(u):=\sum_{\gamma_2=1}^{d_2}c_{\gamma_1, \gamma_2}u_{\gamma_2}.
\]
Similarly for  $v=(v_1,\ldots, v_{d_1})\in\ZZ^{d_1}$ and $\gamma_2\in[d_2]\cup\{0\}$, we also set
\[
\beta^2(v):=(\beta_1^2(v),\ldots, \beta_{d_2}^2(v))\qquad\text{ and } \qquad\beta_{\gamma_2}^2(v):=\sum_{\gamma_1=1}^{d_1}c_{\gamma_1, \gamma_2}v_{\gamma_1}.
\]
This implies, raising both sides of \eqref{eq:w12} to power $2k_1$ for
any $k_1\in\ZZ_+$, that 
\begin{align}
\label{eq:w11}
\begin{split}
S_{K_1, M_1, K_2, M_2}^1(Q)^{4k_1k_2}
&\lesssim M_1^{4k_1k_2-2k_1}\max_{N_2\in[M_2]}\Big(\sum_{u\in \mathcal B_{d_2}(N_2)}
J_{k_2, d_2}(N_2; u) |S_{d_1}(\beta^1(u); M_1)|\Big)^{2k_1}\\
&\lesssim M_1^{4k_1k_2-2k_1}M_2^{4k_1k_2-2k_2}\max_{N_2\in[M_2]}\sum_{u\in \mathcal B_{d_2}(N_2)}
J_{k_2, d_2}(N_2; u) |S_{d_1}(\beta^1(u); M_1)|^{2k_1}.  
\end{split}
\end{align}
In \eqref{eq:w11} we used H\"older's inequality and 
\[
\sum_{u\in \mathcal B_{d_2}(N_2)}
 J_{k_2, d_2}(N_2; u)=N_2^{2k_2}.
\]
\paragraph{\bf Step 3.} For $v=(v_1, \ldots, v_{d_1})\in\ZZ^{d_1}$ we have
\begin{align}\label{eq:w11'}
\beta^1(u)\cdot v=\sum_{\gamma_2=1}^{d_2}\sum_{\gamma_1=1}^{d_1}c_{\gamma_1, \gamma_2}v_{\gamma_1}u_{\gamma_2}=\beta^2(v)\cdot u.
\end{align}
Applying \eqref{eq:w1} and \eqref{eq:w9.5} to the last sum in \eqref{eq:w11} we obtain
\begin{align}\label{eq:w11.5}
\nonumber \max_{N_2\in[M_2]}\sum_{u\in \mathcal B_{d_2}(N_2)}&
J_{k_2, d_2}(N_2; u) |S_{d_1}(\beta^1(u); M_1)|^{2k_1}\\
\nonumber &\le J_{k_2,d_2}(M_2) \sum_{u\in \mathcal B_{d_2}(M_2)}
|S_{d_1}(\beta^1(u);M_1)|^{2k_1}\\
\nonumber &= J_{k_2,d_2}(M_2) \sum_{u\in \mathcal B_{d_2}(M_2)}\sum_{v\in \mathcal B_{d_1}(M_1)} 
J_{k_1, d_1}(M_1; v)\ex(\beta^1(u)\cdot v)\\
&\le J_{k_1, d_1}(M_1) J_{k_2,d_2}(M_2) \sum_{v\in \mathcal B_{d_1}(M_1)} 
\big|\sum_{u\in \mathcal B_{d_2}(M_2)} \ex(\beta^2(v)\cdot u)\big|,
\end{align}
where we used \eqref{eq:w11'} in the last inequality.
In a slightly more involved process, we now obtain a different estimate for the last sum in \eqref{eq:w11}.
We apply \eqref{eq:w1} twice to obtain
\begin{align*}
\max_{N_2\in[M_2]}\sum_{u\in \mathcal B_{d_2}(N_2)}&
J_{k_2, d_2}(N_2; u) |S_{d_1}(\beta^1(u); M_1)|^{2k_1}\\
&= \max_{N_2\in[M_2]}\sum_{u\in \mathcal B_{d_2}(N_2)}\sum_{v\in \mathcal B_{d_1}(M_1)} J_{k_2, d_2}(N_2; u) J_{k_1, d_1}(M_1; v)\ex(\beta^1(u)\cdot v)\\
&= \max_{N_2\in[M_2]}\sum_{v\in \mathcal B_{d_1}(M_1)}  J_{k_1, d_1}(M_1; v)
\sum_{u\in \mathcal B_{d_2}(N_2)} J_{k_2, d_2}(N_2; u)\ex(\beta^2(v)\cdot u)\\
&= \max_{N_2\in[M_2]}\sum_{v\in \mathcal B_{d_1}(M_1)}  J_{k_1, d_1}(M_1; v) |S_{d_2}(\beta^2(v);N_2)|^{2k_2}
\end{align*}
where we used \eqref{eq:w11'} in the penultimate equality. Hence by \eqref{eq:w1}, \eqref{eq:w9.5} and \eqref{eq:w11'},
\begin{align}\label{eq:w11.6}
\nonumber \max_{N_2\in[M_2]}\sum_{u\in \mathcal B_{d_2}(N_2)}&
J_{k_2, d_2}(N_2; u) |S_{d_1}(\beta^1(u); M_1)|^{2k_1}\\
\nonumber &\le J_{k_1,d_1}(M_1) \max_{N_2\in[M_2]}\sum_{v\in \mathcal B_{d_1}(M_1)} |S_{d_2}(\beta^2(v);N_2)|^{2k_2}\\
\nonumber &= J_{k_1, d_1}(M_1) \max_{N_2\in[M_2]}\sum_{u\in \mathcal B_{d_2}(N_2)} J_{k_2, d_2}(N_2; u)
\sum_{v\in \mathcal B_{d_1}(M_1)} \ex(\beta^1(u)\cdot v)\\
&\qquad \le J_{k_1, d_1}(M_1)J_{k_2, d_2}(M_2)\sum_{u\in \mathcal B_{d_2}(M_2)} \big|\sum_{v\in \mathcal B_{d_1}(M_1)} 
\ex(\beta^1(u)\cdot v) \big|.
\end{align}
\paragraph{\bf Step 4.} In this step we prove (for $q = q_{\rho_1,\rho_2}$)
\begin{align}\label{sum-I-bound}
\sum_{u\in \mathcal B_{d_2}(M_2)} \big|\sum_{v\in \mathcal B_{d_1}(M_1)} 
\ex(\beta^1(u)\cdot v) \big| \lesssim \prod_{j=1}^2 M_j^{\frac{d_j(d_j+1)}{2}} 
\bigg( \frac{1}{M_2^{\rho_2}}+\frac{q\log q}{M_1^{\rho_1}M_2^{\rho_2}}+\frac{1}{q}+\frac{\log q}{M_1^{\rho_1}}\bigg)
\end{align}
and 
\begin{align}\label{sum-II-bound}
\sum_{v\in \mathcal B_{d_1}(M_1)} \big|\sum_{u\in \mathcal B_{d_2}(M_2)} 
\ex(\beta^2(v)\cdot u) \big| \lesssim \prod_{j=1}^2 M_j^{\frac{d_j(d_j+1)}{2}} 
\bigg( \frac{1}{M_1^{\rho_1}}+\frac{q\log q}{M_1^{\rho_1}M_2^{\rho_2}}+\frac{1}{q}+\frac{\log q}{M_2^{\rho_2}}\bigg).
\end{align}
We only establish \eqref{sum-I-bound}. The symmetric bound \eqref{sum-II-bound} is similar. The exponential sum
$$
\sum_{v\in \mathcal B_{d_1}(M_1)}  \ex(\beta^1(u)\cdot v) = \prod_{\gamma_1=1}^{d_1} \sum_{|v_{\gamma_1}| \le k_1 M_1^{\gamma_1}}
\ex(\beta^1_{\gamma_1}(u) v_{\gamma_1})
$$
is a product of geometric series which we can easily evaluate to conclude
\begin{align*}
\nonumber &\sum_{u\in \mathcal B_{d_2}(M_2)}  \big|\sum_{v\in \mathcal B_{d_1}(M_1)} 
\ex(\beta^1(u)\cdot v) \big| \lesssim \sum_{u\in \mathcal B_{d_2}(M_2)} \prod_{\gamma_1=1}^{d_1}
\min\bigg\{2d_1M_1^{\gamma_1}, \frac{1}{\|\beta_{\gamma_1}^1(u)\|}\bigg\}\\
&\hspace{1cm}\le
(2d_1M_1)^{\frac{d_1(d_1+1)}{2}-\rho_1}  \sum_{u\in \mathcal
B_{d_2}(M_2)}
\min\bigg\{2d_1M_1^{\rho_1}, \frac{1}{\|\beta_{\rho_1}^1(u)\|}\bigg\}.
\end{align*}
Since \eqref{eq:w6} holds and 
\[
\beta_{\rho_1}^1(u)=c_{\rho_1,\rho_2}u_{\rho_2}+\beta(u), \qquad\text{ where } \qquad
\beta(u):=\sum_{\substack{\gamma_2=1\\\gamma_2\neq\rho_2}}^{d_2}c_{\rho_1, \gamma_2}u_{\gamma_2},
\]
we can apply \eqref{eq:w15} with $P=k_{2}M_2^{\rho_2}$, $U=2d_1M_1^{\rho_1}$ and $q=q_{\rho_1, \rho_2}$ and obtain
\begin{align*}
\sum_{|u_{\rho_2}|\le k_{2}M_2^{\rho_2}} \min\bigg\{2d_1M_1^{\rho_1}, \frac{1}{\|\beta_{\rho_1}^1(u)\|}\bigg\}
&\lesssim M_1^{\rho_1}+q\log q+\frac{M_1^{\rho_1}M_2^{\rho_2}}{q}+M_2^{\rho_2}\log q\\
&\lesssim M_1^{\rho_1}M_2^{\rho_2}\bigg( \frac{1}{M_2^{\rho_2}}+\frac{q\log q}{M_1^{\rho_1}M_2^{\rho_2}}+\frac{1}{q}+\frac{\log
q}{M_1^{\rho_1}}\bigg).
\end{align*}
Hence
$$
\sum_{u\in \mathcal B_{d_2}(M_2)} \big|\sum_{v\in \mathcal B_{d_1}(M_1)} 
\ex(\beta^1(u)\cdot v) \big| \lesssim \bigg(\prod_{j=1}^2 M_j^{\frac{d_j(d_j+1)}{2}} \bigg) 
\bigg( \frac{1}{M_2^{\rho_2}}+\frac{q\log q}{M_1^{\rho_1}M_2^{\rho_2}}+\frac{1}{q}+\frac{\log q}{M_1^{\rho_1}}\bigg),
$$
establishing \eqref{sum-I-bound}.
\paragraph{\bf Step 5} We use the bound \eqref{sum-I-bound} in \eqref{eq:w11.6} to conclude
\begin{align*}
S_{K_1, M_1, K_2, M_2}^1(Q)^{4k_1k_2}
&\lesssim  M_1^{4k_1k_2-2k_1}M_2^{4k_1k_2-2k_2}J_{k_1, d_1}(M_1)J_{k_2, d_2}(M_2)\\
&\qquad \times M_1^{\frac{d_1(d_1+1)}{2}}M_2^{\frac{d_2(d_2+1)}{2}}
\bigg( \frac{1}{M_2^{\rho_2}}+\frac{q\log q}{M_1^{\rho_1}M_2^{\rho_2}}+\frac{1}{q}+\frac{\log
q}{M_1^{\rho_1}}\bigg).
\end{align*}
From Vinogradov's mean value theorem (or more precisely from \eqref{eq:w10'} with $s=k_i:=d_i(d_i+1)$ and $k=d_i$ for $i\in[2]$)  we conclude
from \eqref{eq:w10'}, $J_{k_i, d_i}(M_i) \le C M_i^{3 k_i/2}, i=1,2$ and so
\begin{align*}
S_{K_1, M_1, M_2, M_2}^1(Q)^{4k_1k_2}
\lesssim M_1^{4k_1k_2}M_2^{4k_1k_2}\bigg( \frac{1}{M_2^{\rho_2}}+\frac{q\log q}{M_1^{\rho_1}M_2^{\rho_2}}+\frac{1}{q}+\frac{\log
q}{M_1^{\rho_1}}\bigg).  
\end{align*}
In a similar way, using \eqref{sum-II-bound} in \eqref{eq:w11.5}, we also have
\begin{align*}
S_{K_1, M_1, M_2, M_2}^1(Q)^{4k_1k_2}
\lesssim M_1^{4k_1k_2}M_2^{4k_1k_2}\bigg( \frac{1}{M_1^{\rho_1}}+\frac{q\log q}{M_1^{\rho_1}M_2^{\rho_2}}+\frac{1}{q}+\frac{\log
q}{M_2^{\rho_2}}\bigg).  
\end{align*}
Therefore $S_{K_1, M_1, M_2, M_2}^1(Q)^{4k_1k_2}$ is bounded from above by the minimum of these two bounds.
By Remark \ref{rem:2-0'}, 
this completes the proof of Proposition  \ref{prop:22}.
\end{proof}

\subsection{Double Weyl's inequality in the Newton diagram sectors}
Throughout this subsection we assume that $P\in\ZZ[\rm m_1, \rm m_2]$ and $P(0, 0)=0$.
Moreover, we assume that $P$ is non-degenerate in the sense of \eqref{eq:66}; see the remark below Theorem \ref{thm:main}.   
Then for every $\xi\in\RR$, we define a corresponding polynomial $P_{\xi}\in \RR[\rm m_1, \rm m_2]$ by setting
\begin{align}
\label{eq:211}
P_{\xi}(m_1, m_2):=\xi P(m_1, m_2).
\end{align}
It is clear to see that the backwards Newton diagrams of $P$ and $P_{\xi}$ are the same $N_P=N_{P_{\xi}}$. Let  $r\in\ZZ_+$ be the number of vertices in the backwards Newton diagram $N_{P}$. In view of \eqref{eq:209} and \eqref{eq:210} from  Remark \ref{rem:2-0} for $r\ge2$ we have
\begin{align}
\label{eq:107}
\begin{split}
&\log M_1\lesssim \log M_2 \quad \text{ if } \quad (M_1, M_2)\in\SS_{\tau}(1),\\
&\log M_1\simeq \log M_2  \quad \text{ if } \quad (M_1, M_2)\in\SS_{\tau}(j) \text{ for } 1<j<r,\\
&\log M_2\lesssim \log M_1 \quad \text{ if } \quad (M_1, M_2)\in\SS_{\tau}(r).
\end{split}
\end{align}
Consequently we may  define a quantity $M_{r, j}^*$ as follows.
If $r=1$, we simply set
\begin{align}
\label{eq:212'}
M_{1,1}^*:=M_1\vee M_2 \quad \text{ if }\quad (M_1, M_2)\in\SS_{\tau}(1)=\DD_{\tau}\times \DD_{\tau}.
\end{align}
If $r\ge2$,  we  set
\begin{align}
\label{eq:212}
M_{r, j}^*:=
\begin{cases}
M_2 & \text{ if } (M_1, M_2)\in\SS_{\tau}(1) \text{ for } j=1,\\
M_1\vee M_2 & \text{ if } (M_1, M_2)\in\SS_{\tau}(j) \text{ for } 1<j<r,\\
M_1 & \text{ if } (M_1, M_2)\in\SS_{\tau}(r)\text{ for } j=r.
\end{cases}
\end{align}
 The quantity $\log M_{r,j}^*$ will always  allow us to extract the larger parameter (larger up to a multiplicative constant as in \eqref{eq:107}) from $\log M_1$ and $\log M_2$.
 We estimate $|S_{K_1, M_1, K_2, M_2}(P_{\xi})|$ in terms of $\log M_{r, j}^*$, whenever $(M_1, M_2)\in\SS_{\tau}(j)$   for $j\in[r]$, and $(K_1, K_2)\in\NN^2$ satisfying $M_1\lesssim K_1\le M_1$ and $M_2\lesssim K_2\le M_2$.

\begin{proposition}
\label{prop:23}
Let $P_{\xi}\in \RR[\rm m_1, \rm m_2]$ be the polynomial in \eqref{eq:211} corresponding to a polynomial $P\in\ZZ[\rm m_1, \rm m_2]$ with the properties above. Let $r\in\ZZ_+$ be the number of vertices in the backwards Newton diagram $N_{P}$. Let $\tau>1$, $\alpha>1$, $j\in[r]$ be given. Let $v_j=(v_{j, 1}, v_{j, 2})$ be the vertex of the backwards Newton diagram $N_{P}$.
Then there exists a constant $\beta_0:=\beta_0(\alpha)>\alpha$ such that for every $\beta\in(\beta_0,\infty)\cap\ZZ_+$ we find  a constant $0<C=C(\alpha, \beta_0, \beta, j, \tau, P)<\infty$ such that  for every $(M_1, M_2)\in\SS_{\tau}(j)$ and $(K_1, K_2)\in\NN^2$ satisfying $M_1\lesssim K_1\le M_1$ and $M_2\lesssim K_2\le M_2$ the following holds. Suppose that there are
$a\in\ZZ, q\in\ZZ_+$ such that   $(a, q)=1$
and
\begin{align}
\label{eq:215}
(\log M_{r, j}^*)^{\beta}\lesssim
q\le M_1^{v_{j, 1}}M_2^{v_{j, 2}}(\log M_{r, j}^*)^{-\beta},
\end{align}
and 
\begin{align}
\label{eq:w6-1}
\Big|\xi-\frac{a}{q}\Big|\le\frac{(\log M_{r, j}^*)^{\beta}}{qM_1^{v_{j, 1}}M_2^{v_{j, 2}}},
\end{align}
where $M_{r, j}^*$ is defined in \eqref{eq:212}. Then one has
\begin{align}
\label{eq:227}
|S_{K_1, M_1, K_2, M_2}(P_{\xi})|\le C M_1M_2(\log M_{r, j}^*)^{-\alpha}.
\end{align}
\end{proposition}

\begin{proof}
We note that the following three scenarios may occur when $r>1$:
\begin{enumerate}[label*={\arabic*}.]
\item If $j=1$ we  have $v_{1, 1}=0$ or $v_1\in\ZZ_+\times\ZZ_+$. In this  case we also have  $\log M_1\lesssim \log M_2$.
\item If $j=r$ we  have $v_{r, 2}=0$ or $v_r\in\ZZ_+\times\ZZ_+$. In this  case we also have  $\log M_1\gtrsim \log M_2$.
\item If $1<j<r$ we have  $v_j\in\ZZ_+\times\ZZ_+$. In this case we also have  $\log M_1\simeq \log M_2$.
\end{enumerate}
Note that if $r=1$ then $\SS_{\tau}(1)=\DD_{\tau}\times\DD_{\tau}$ and $v_1\in\ZZ_+\times\ZZ_+$, since $P$ is non-degenerate in the sense of \eqref{eq:66}. Throughout the proof, in the case of  $r=1$, we will additionally assume that $\log M_1\le \log M_2$. Taking into account \eqref{eq:212'} and \eqref{eq:212} we can also assume that $\log M_1\vee \log M_2$ is sufficiently large, i.e. $\log M_1\vee \log M_2>C_{0}$, where $C_{0}=C_0(\alpha, \beta_0, j, \tau, P)>0$ is a large absolute constant. Otherwise, inequality  \eqref{eq:227} follows.
The proof will be divided into three steps. 
\vskip 5pt
\paragraph{\bf Step 1.}
We first establish \eqref{eq:227} when $j=1$ and $v_{1, 1}=0$ or $j=r$ and $v_{r, 2}=0$. Suppose that $j=1$ and $v_{1, 1}=0$ holds. The case when $j=r$ and $v_{r, 2}=0$ can be proved in a similar way so we omit the details. 
As we have seen above $\log M_1\lesssim \log M_2$. By \eqref{eq:215} and \eqref{eq:w6-1} we obtain
\begin{align*}
\Big|\xi-\frac{a}{q}\Big|\le\frac{1}{q^2}.
\end{align*}
Applying Lemma \ref{lem:21} with $Q= c_{0, v_{1, 2}}$ and $M=q$ we may find a fraction $a'/q'$ such that $(a', q')=1$ and $q/(2c_{0, v_{1, 2}})\le q'\le 2q$ and 
\begin{align*}
\Big|c_{0, v_{1, 2}}\xi-\frac{a'}{q'}\Big|\le\frac{1}{2q'q}\le \frac{1}{(q')^2}.
\end{align*}
Thus by Proposition \ref{prop:20'}, noting that $v_{1,2}\ge 1$, 
we obtain
\begin{align*}
|S_{K_1, M_1, K_2, M_2}(P_{\xi})|&\le S_{K_1, M_1, K_2, M_2}^1(P_{\xi})\\
&\lesssim M_1M_2\log(M_2)\bigg(\frac 1{q'}+\frac{1}{M_2}+\frac{q'}{M_2^{v_{1, 2}}}\bigg)^{\frac{1}{\tau(\deg P)}}\\
&\lesssim M_1M_2 (\log M_{r, j}^*)^{-\frac{\beta}{\tau(\deg P)}+1},
\end{align*}
since $\log M_{r, j}^*\simeq \log M_2$. It suffices to take $\beta>\tau(\deg P)(\alpha+1)$ and the claim in \eqref{eq:227} follows. 
\vskip 5pt
\paragraph{\bf Step 2.} 
We now establish \eqref{eq:227} when $1\le j\le r$ and $v_j\in\ZZ_+\times\ZZ_+$ (note that when $1<j<r$, we automatically have
$v_j\in\ZZ_+\times\ZZ_+$). If $r=1$ then we assume that $\log M_1\le \log M_2$.  If $r\ge2$ we will assume that $1\le j<r$, which 
gives that $\log M_1\lesssim \log M_2$. The case when $j=r$ can be proved in much the same way, (with the difference that $\log M_1\gtrsim \log M_2$), we omit the details. In this step, we additionally assume that  $M_1\le (\log M_{r, j}^*)^{\chi}$ for some  $0<\chi<\beta/(8\deg P)$ 
with $\beta$ to be specified later.

Notice that \eqref{eq:215} and \eqref{eq:w6-1} imply
\begin{align*}
\Big|\xi-\frac{a}{q}\Big|\le\frac{1}{q^2}.
\end{align*}
By  \eqref{eq:215}  and  $M_1\le (\log M_{r, j}^*)^{\chi}$ we conclude
\begin{align*}
(\log M_{r, j}^*)^{\beta}\le
q\le M_2^{v_{j, 2}}(\log M_{r, j}^*)^{-3\beta/4}
\end{align*}
since $\chi  < \beta/(8\deg P)$. 
We note that the polynomial $P$ can be written as
\begin{align*}
P(m_1, m_2)=P_{v_{j, 1}}(m_1)m_2^{v_{j, 2}}+\sum_{\substack{(\gamma_1, \gamma_2)\in S_P\\ \gamma_2\neq v_{j, 2}}}c_{\gamma_1, \gamma_2}m_1^{\gamma_1}m_2^{\gamma_2},
\end{align*}
where $P_{v_{j, 1}}\in\ZZ[\rm m_1]$ and $\deg P_{v_{j, 1}}=v_{j, 1}$.

Observe that for every $1\le m_1\le M_1\le (\log M_{r, j}^*)^{\chi}$ one has
\begin{align*}
|P_{v_{j, 1}}(m_1)|\le \#S_P\max_{(\gamma_1, \gamma_2)\in S_P}|c_{\gamma_1, \gamma_2}|M_1^{\deg P}\lesssim_P
(\log M_{r, j}^*)^{\beta/4}.
\end{align*}

Applying Lemma \ref{lem:21} with $M=M_2^{v_{j, 2}}(\log M_{r, j}^*)^{-3\beta/4}$ and $Q=P_{v_{j, 1}}(m_1)$ for each $K_1< m_1\le M_1$ (noting that $P_{v_{j,1}}(m_1) \not = 0$ for large $m_1$) we find a fraction $a'/q'$ so that $(a', q')=1$ and $(\log M_{r, j}^*)^{3\beta/4}\lesssim q'\le 2M_2^{v_{j, 2}}(\log M_{r, j}^*)^{-3\beta/4}$ and
\begin{align*}
\Big|P_{v_{j, 1}}(m_1)\xi-\frac{a'}{q'}\Big|\le\frac{(\log M_{r, j}^*)^{3\beta/4}}{2q'M_2^{v_{j, 2}}}\le \frac{1}{(q')^2}.
\end{align*}
We apply Proposition \ref{prop:20'} for each $1\le m_1\le M_1$, noting that $v_{j,2}\ge 1$, to bound
\begin{align*}
\Big|\sum_{m_2=K_2+1}^{M_2}\ex(P_{\xi}(m_1, m_2))\Big|
\lesssim M_2\log(M_2)\bigg(\frac 1{q'}+\frac{1}{M_2}+\frac{q'}{M_2^{v_{j, 2}}}\bigg)^{\frac{1}{\tau(\deg P)}}\lesssim M_2 (\log M_{r, j}^*)^{-\frac{3\beta}{4\tau(\deg P)}+1},
\end{align*}
since $\log M_{r, j}^*\simeq \log M_2$ for  $j\in[r-1]$. It suffices to take $\beta>\frac{4}{3}\tau(\deg P)(\alpha+1)$ and \eqref{eq:227} follows. 
\vskip 5pt
\paragraph{\bf Step 3.}
As in the previous step $1\le j<r$ (or $r=1$ and $\log M_1\le \log M_2$)  and we now  assume that $(\log M_{r, j}^*)^{\chi}\le M_1\lesssim M_2$ for some  $0<\chi<\beta/(8\deg P)$, which will be further adjusted. The case when $j=r$ can be established in a similar fashion keping in mind that $\log M_1\gtrsim \log M_2$. In fact, we take $\chi:=\beta/(16\deg P)+1$, which forces $\beta>16\deg P$.   

Applying Lemma \ref{lem:21} with $Q=c_{v_{j, 1}, v_{j,2}}$ and $M=q$, we find a fraction $a'/q'$ so that $(a', q')=1$ and 
$(\log M_{r, j}^*)^{\beta}\lesssim_P q (2Q)^{-1} \le q'\le 2q$ and
\begin{align*}
\Big|c_{v_{j, 1}, v_{j,2}}\xi-\frac{a'}{q'}\Big| \ \le \  \frac{1}{(q')^2}.
\end{align*}
From Proposition \ref{prop:22}, we obtain (with $M_{-} = \min(M_1^{v_{j,1}}, M_2^{v_{j,2}})$ and $M_{+} = \max(M_1^{v_{j,1}}, M_2^{v_{j,2}})$)
\begin{align*}
|S_{K_1, M_1, K_2, M_2}(P_{\xi})|&\lesssim M_1M_2
\bigg( \frac{1}{M_{-}}+\frac{q'\log q'}{M_1^{v_{j, 1}}M_2^{v_{j, 2}}}+\frac{1}{q'}+\frac{\log
q'}{M_{+}}\bigg)^{\frac{1}{4(1+\deg P)^5}}\\
&\lesssim M_1M_2(\log M_{r, j}^*)^{-\frac{\beta}{64(1+\deg P)^5}}.
\end{align*}
Taking $\beta>64(1+\deg P)^5(\alpha+1)$ we obtain \eqref{eq:227}. This completes the proof of Proposition \ref{prop:23}.   
\end{proof}

\subsection{Estimates for double complete exponential sums}
In this subsection we provide estimates for double complete exponential sums in the spirit of Gauss.
We begin with a well-known bound  which is also a simple consequence of Proposition \ref{prop:22}.

\begin{lemma}[\cite{ACK}]
\label{lem:31}
Let $P\in\QQ[\rm m_1, \rm m_2]$ be a
polynomial as in \eqref{eq:12} and let 
$a_{\gamma_1, \gamma_2}\in\ZZ$ and $q\in\ZZ_+$ satisfy
 $c_{\gamma_1, \gamma_2}=a_{\gamma_1, \gamma_2}/q$ for each $(\gamma_1, \gamma_2)\in S_P$ 
such that
\begin{align*}
\gcd(\{a_{\gamma_1, \gamma_2}: (\gamma_1, \gamma_2)\in S_P\}\cup\{q\})=1.
\end{align*}
Consider the exponential sum $S_{q,q}$ from \eqref{00}.
 Then there are  $C>0$ and
$\delta\in(0, 1)$ such that
\begin{align}
\label{eq:97}
|S_{q, q}(P)|&\le Cq^{2-\delta}
\end{align}
holds. The constant $C$ can be taken to depend only on the degree of $P$.
\end{lemma}

 We now derive simple
consequences of Lemma \ref{lem:31} for exponential sums that arise in
the proof of our main result.  Let $P\in\ZZ[\rm m_1, \rm m_2]$ be such
that
\begin{align}
\label{eq:103}
P(m_1, m_2):=\sum_{(\gamma_1, \gamma_2)\in S_P}c_{\gamma_1, \gamma_2}^Pm_1^{\gamma_1}m_2^{\gamma_2},
\end{align}
where $c_{(0,0)}^P=0$. We additionally assume that $P$ 
is  non-degenerate (see the remark below Theorem \ref{thm:main}). That is, we have
$S_{P}\cap(\ZZ_{+}\times\ZZ_{+})\neq\emptyset$.
Using the definition of $P_{\xi}$ from \eqref{eq:211}, we define the complete exponential sum by 
\begin{align}
\label{eq:323}
G(a/q):=\frac{1}{q^2}\sum_{r_1=1}^q\sum_{r_2=1}^q\ex(P_{a/q}(r_1, r_2)),\qquad  a/q\in\QQ,
\end{align}
and we also have partial complete exponential sums defined by
\begin{align}
\label{eq:325}
\begin{split}
G_{m_1}^1(a/q):=\frac{1}{q}\sum_{r_2=1}^q\ex(P_{a/q}(m_1, r_2)), \qquad a/q\in\QQ,\; m_1\in\ZZ,\\
G_{m_2}^2(a/q):=\frac{1}{q}\sum_{r_1=1}^q\ex(P_{a/q}(r_1, m_2)), \qquad a/q\in\QQ,\; m_2\in\ZZ.
\end{split}
\end{align}
\begin{proposition}
\label{prop:32}
Let $P\in\ZZ[\rm m_1, \rm m_2]$ be a polynomial  as in \eqref{eq:103} which is non-degenerate
(that is, $S_{P}\cap(\ZZ_{+}\times\ZZ_{+})\neq\emptyset$). Then there is $C_P>0$ and $\delta\in(0, 1)$ such that the following inequalities hold. If $a/q\in\QQ$ and $(a, q)=1$, then
\begin{align}
\label{eq:104}
|G(a/q)|\le C_P \, q^{-\delta}.
\end{align}
Moreover,  for every sufficiently large $K_1, M_1\in\ZZ_+$ depending on $P$ one has
\begin{align}
\label{eq:105}
\frac{1}{M_1}\sum_{m_1=K_1 + 1}^{M_1}|G_{m_1}^1(a/q)|\le C_P \, q^{-\delta},
\end{align}
and similarly  for every sufficiently large $K_2, M_2\in\ZZ_+$ depending on $P$ one has
\begin{align}
\label{eq:106}
\frac{1}{M_2}\sum_{m_2=K_2 + 1}^{M_2}|G_{m_2}^2(a/q)|\le C_P \, q^{-\delta}.
\end{align}
\end{proposition}

\begin{proof}
We prove Proposition \ref{prop:32} in two steps. 
\paragraph{\bf Step 1} In this step we establish \eqref{eq:104}. Fix $a/q\in\QQ$
such that $(a, q)=1$.  For any $(\gamma_1, \gamma_2)\in S_P$ we let
$a_{\gamma_1, \gamma_2}:=ac_{\gamma_1, \gamma_2}^P/(c_{\gamma_1, \gamma_2}^P, q)$
and $q_{\gamma_1, \gamma_2}:=q/(c_{\gamma_1, \gamma_2}^P, q)$.  Now
with this notation we see that
\begin{align*}
P_{a/q}(r_1, r_2)=Q(r_1, r_2):=\sum_{\gamma_1=0}^{d_1}\sum_{\gamma_2=0}^{d_2}\frac{a_{\gamma_1, \gamma_2}}{q_{\gamma_1, \gamma_2}}r_1^{\gamma_1}r_2^{\gamma_2},
\end{align*}
for some integers $d_1, d_2\ge1$. Furthermore,
$G(a/q) = q^{-2} S_{q,q}(Q)$, see \eqref{eq:323}. We take $(\rho_1, \rho_2)\in S_{P}\cap(\ZZ_{+}\times\ZZ_{+})\neq\emptyset$ and 
use \eqref{eq:97},  which yields
\begin{align*}
|G(a/q)|=q^{-2}|S_{q, q}(Q)|\lesssim_P q^{-\delta}.
\end{align*}
This completes the proof of \eqref{eq:104}.
\vskip 5pt
\paragraph{\bf Step 2}
We only prove \eqref{eq:105}, the proof of \eqref{eq:106} is exactly the same. We fix  $a/q\in\QQ$ such that $(a, q)=1$, and  
we also fix $(\rho_1, \rho_2)\in S_{P}\cap(\ZZ_+\times\ZZ_+)\neq\emptyset$.  Using Lemma \ref{lem:21} we find a reduced fraction $a_{\rho_1, \rho_2}/q_{\rho_1, \rho_2}$ so that $(a_{\rho_1, \rho_2},q_{\rho_1, \rho_2})=1$ and
\begin{align*}
\Big|\frac{a c_{\rho_1, \rho_2}^P}{q}-\frac{a_{\rho_1, \rho_2}}{q_{\rho_1, \rho_2}}\Big|\le\frac{1}{2q_{\rho_1, \rho_2}q}
\end{align*}
with $q/(2c_{\rho_1, \rho_2})\le q_{\rho_1, \rho_2}\le 2q$. We fix $\chi>0$ and assume first that $M_1\ge q^{\chi}$. Appealing to inequality \eqref{eq:86} with $M_2=q$ we obtain for some $\delta\in(0, 1)$ that
\begin{align*}
\frac{1}{M_1}\sum_{m_1=K_1+1}^{M_1}|G_{m_1}^1(a/q)|\lesssim_P q^{-\delta}
\end{align*}

We now establish a similar bound assuming that $M_1< q^{\chi}$ for a sufficiently small $\chi>0$, which will be specified momentarily. Our polynomial $P$ from \eqref{eq:103} can be rewritten as
\begin{align*}
P(m_1, m_2)=\sum_{\gamma_2=1}^{d_2}P_{\gamma_2}(m_1)m_2^{\gamma_2}+P_{0}(m_1),
\end{align*}
for some $d_2\ge1$ where $P_{\gamma_2}\in\ZZ[\rm m_1]$ and $\deg P_{\gamma_2}\le \deg P$.  Take $0<\chi<\frac{1}{10\deg P}$, and
observe that for every $1\le \gamma_2\le d_2$ and  for every $1\le m_1\le M_1\le q^{\chi}$ one has
\begin{align}
\label{eq:61}
|P_{\gamma_2}(m_1)|\le \#S_P\max_{(\gamma_1, \gamma_2)\in S_P}|c_{\gamma_1, \gamma_2}|M_1^{\deg P}\le
q^{1/4},
\end{align}
whenever $q$ is sufficiently large in terms of the coefficients of $P$.

Assume first that $d_2\ge2$, and we may take $\rho_2=d_2$. Applying Lemma \ref{lem:21} with $Q=P_{\rho_2}(m_1)$ for each $K_1< m_1\le M_1$ (noting that $P_{\rho_2}(m_1) \not= 0$  for sufficiently  large $m_1\ge K_1$), we find a fraction $a'/q'$ so that $(a', q')=1$ and $\frac{1}{2}q^{3/4}\le q'\le 2q$ and
\begin{align*}
\Big|P_{\rho_2}(m_1)\frac{a}{q}-\frac{a'}{q'}\Big|\le\frac{1}{2q'q}\le \frac{1}{(q')^2}.
\end{align*}
Then we apply Proposition \ref{prop:20'} for each $K_1 < m_1\le M_1$, which gives
\begin{align*}
|G_{m_1}^1(a/q)|
\lesssim \log(2q)\bigg(\frac 1{q'}+\frac 1 q+\frac{q'}{q^{d_2}}\bigg)^{\frac{1}{\tau(d_2)}}\lesssim  (\log q)q^{-\frac{3}{4\tau(d_2)}}\lesssim q^{-\delta},
\end{align*}
for some $\delta\in(0, 1)$ 
and \eqref{eq:105} follows, since $d_2\ge2$.

Assume now that $d_2=1$, then
\begin{align*}
\frac{1}{M_1}\sum_{m_1=K_1 + 1}^{M_1}|G_{m_1}^1(a/q)|= \frac{1}{M_1}\#\{K_1 < m_1\le M_1: P_{1}(m_1)\equiv 0 \bmod q \}=0,
\end{align*}
in view of \eqref{eq:61}, which ensures that $\{m_1\in [M_1]: P_{1}(m_1)\equiv 0 \bmod q \}=\emptyset$.
\end{proof}
\section{Multi-parameter Ionescu--Wainger theory}
\label{section:6}

One of the most important ingredients in our argument is the
Ionescu--Wainger multiplier theorem \cite{IW}, see also \cite{M10},
and its vector-valued variant from \cite{MSZ3}, see  also \cite{TaoIW}. 
We begin with recalling the results from \cite{IW} and \cite{MSZ3} and
fixing necessary notation and terminology.

\subsection{Ionescu--Wainger multiplier theorem}
Let $\PP$ be the set of all prime numbers, and let $\rho\in(0, 1)$ be
a sufficiently small absolute constant.  We then define the natural
number
\[
D:=D_{\rho}:=\lfloor 2/\rho \rfloor + 1,
\]
and for any integer  $l\in\NN$, set
\begin{align*}
N_0 := N_0^{(l)} := \lfloor 2^{\rho l/2} \rfloor + 1,
\quad \text{ and } \quad
Q_0 := Q_0^{(l)} := (N_0!)^D.
\end{align*}
We also define the set
\begin{align*}
P_{\leq l}:=\big\{q = Qw: Q|Q_0 \text{ and } w\in W_{\le l}\cup\{1\}\big\},
\end{align*}
where
\begin{align*}
W_{\leq l}:=\bigcup_{k\in[D]}\bigcup_{(\gamma_1,\dots,\gamma_k)\in[D]^k}\big\{p_1^{\gamma_1} \cdots p_k^{\gamma_k}\colon
 p_1,\ldots, p_k\in(N_0^{(l)}, 2^l]\cap\PP \text{ are distinct}\big\}.
\end{align*}
In other words $W_{\leq l}$ is the set of all products of prime
factors from $(N_0^{(l)}, 2^l]\cap\PP$ of length at most $D$, at
powers between $1$ and $D$.

\begin{remark}
\label{rem:2}
For every $\rho\in(0, 1)$ there exists a large absolute constant $C_{\rho}\ge1$ such that the following elementary facts about the sets $P_{\leq l}$ hold:
\begin{enumerate}[label*={(\roman*)}]
\item \label{IW1}  If $l_1\le l_2$, then $P_{\leq l_1}\subseteq P_{\leq l_2}$.
\item \label{IW2}  One has $[2^l] \subseteq P_{\leq l} \subseteq [2^{C_\rho 2^{\rho l}}]$.
\item \label{IW3}  If $q \in P_{\leq l}$, then all factors of $q$ also lie in $P_{\leq l}$.
\item \label{IW0}  One has $Q_{\le l}:=\lcm  (P_{\leq l})\lesssim 2^{C_\rho2^l}$.
\end{enumerate}
\end{remark}
By property \ref{IW1} it makes sense to define $P_l := P_{\leq l} \backslash P_{\leq l-1}$, with the
convention that $P_{\leq l}$ is empty for negative $l$.  From
property \ref{IW2}, for all $q \in P_l$, we have
\begin{align}
\label{eq:371}
2^{l-1} < q \leq 2^{C_\rho 2^{\rho l}}.
\end{align}
Let $d\in\ZZ_+$ and define $1$-periodic   sets
\begin{align}
\label{eq:372}
\Sigma_{\leq l}^d := \Big\{ \frac{a}{q}\in(\QQ\cap\TT)^d:  q \in P_{\leq l} \text{ and } (a, q)=1\Big\},
\quad \text{ and } \quad
\Sigma_l^d := \Sigma_{\leq l}^d \backslash \Sigma_{\leq l-1}^d,
\end{align}
where $(a, q)=(a_1,\ldots, a_d, q)=1$ for any $a=(a_1,\ldots, a_d)\in\ZZ^d$.
Then by \eqref{eq:371} we see
\begin{align}
\label{eq:373}
\# \Sigma_{\leq l}^d \ \le \ 2^{C_\rho(d+1) 2^{\rho l}}.
\end{align}

Let $k\in\ZZ_+$ be fixed. For any finite family of fractions $\Sigma\subseteq (\TT\cap\QQ)^k$ and  a measurable function $\mathfrak m: \RR^k \to B$ taking its values in a separable Banach space $B$ which is supported on the unit
cube $[-1/2, 1/2)^k$, define a $1$-periodic extension of $\mathfrak m$ by
\begin{align*}
\Theta_{\Sigma}[\mathfrak m](\xi):=\sum_{a/q\in\Sigma}\mathfrak m(\xi-a/q), \qquad \xi\in\TT^k.
\end{align*}

We will also need to introduce the notion of  $\Gamma$-lifted extensions of $\mathfrak m$. For $d\in\ZZ_+$ consider $\Gamma:=\{i_1,\ldots, i_k\}\subseteq [d]$ of size $k\in[d]$. We define a $\Gamma$-lifted
$1$-periodic extension of $\mathfrak m$ by
\begin{align*}
\Theta_{\Sigma}^{\Gamma}[\mathfrak m](\xi):=\sum_{a/q\in\Sigma}\mathfrak m(\xi_{i_1}-a_1/q,\ldots, \xi_{i_k}-a_k/q),  \quad \text{ for } \quad \xi=(\xi_1,\ldots, \xi_d)\in\TT^d.
\end{align*}

We now recall the following  vector-valued Ionescu--Wainger
multiplier theorem from \cite{MSZ3,TaoIW}.

\begin{theorem}
\label{thm:IW}
Let $d\in\ZZ_+$ be given. For every $\rho\in(0, 1)$ and for every  $p \in (1,\infty)$, there exists an absolute constant
$C_{p, \rho, d}>0$, that depends only on $p$, $\rho$  and $d$,  such that, for every $l\in\NN$, the following
holds.
Let $0<\varepsilon_l \le 2^{-10 C_\rho 2^{2\rho l}}$, and let
$\mathfrak m: \RR^d \to L(H_0,H_1)$ be a measurable function supported on
$\varepsilon_{l}[-1/2, 1/2)^d$,  with values in the space $L(H_{0},H_{1})$ of bounded linear
operators between separable Hilbert spaces $H_{0}$ and $H_{1}$. Let  
\begin{align}
\label{eq:374}
\mathbf A_{p}:=\|T_{\RR^d}[\mathfrak m]\|_{L^{p}(\RR^d;H_0)\to L^{p}(\RR^d;H_1)}.
\end{align}
Then the $1$-periodic multiplier
\begin{align}
\label{eq:375}
\Theta_{\Sigma_{\le l}^d}[\mathfrak m](\xi)=\sum_{a/q \in\Sigma_{\le l}^d}
\mathfrak m(\xi - a/q)
\quad \text{ for } \quad \xi\in\TT^d,
\end{align}
where $\Sigma_{\le l}^d$ is the set of all reduced fractions in \eqref{eq:372}, satisfies
\begin{align}
\label{eq:376}
\|T_{\ZZ^d}[\Theta_{\Sigma_{\le l}^d}[\mathfrak m]]f\|_{\ell^p(\ZZ^d;H_1)}
\le C_{p,\rho, d}
\mathbf A_{p}
\|f\|_{\ell^p(\ZZ^d;H_0)}
\end{align}
for every $f\in \ell^p(\ZZ^d;H_0)$.
\end{theorem}

The advantage of applying Theorem~\ref{thm:IW} is that one can
directly transfer square function estimates from the continuous to the
discrete setting, which will be useful in Section \ref{section:7}.
The hypothesis \eqref{eq:374}, unlike the support hypothesis, is
scale-invariant, in the sense that the constant $\mathbf A_{p}$ does
not change when $\mathfrak m$ is replaced by $\mathfrak m(A\cdot)$ for
any invertible linear transformation $A:\RR^d\to \RR^d$.

Theorem
\ref{thm:IW} was originally established by Ionescu and Wainger
\cite{IW} in the scalar-valued setting with an extra factor $(l+1)^D$ 
in the right hand side of \eqref{eq:376}. Their proof is based on an
intricate inductive argument that exploits super-orthogonality
phenomena. A slightly different proof with factor $(l+1)$ in
\eqref{eq:376} was given in \cite{M10}. The latter proof, instead of
induction as in \cite{IW}, used certain recursive arguments, which
clarified the role of the underlying square functions and
orthogonalities (see also \cite[Section 2]{MSZ3}). The theorem in the context of super-orthogonality phenomena is
discussed in a survey by Pierce \cite{Pierce} in a much broader
context.  
Finally we refer to
the recent paper of Tao \cite{TaoIW}, where Theorem \ref{thm:IW} as stated above, with a uniform constant
$\mathbf A_{p}$, is established.

For future reference we also recall the sampling principle of
Magyar--Stein--Wainger from \cite{MSW}, which was an important
ingredient in the proof of Theorem \ref{thm:IW}.

\begin{proposition}
\label{prop:msw}
Let $d\in\ZZ_+$ be given. There exists an absolute constant $C>0$ such that the following holds.
Let $p \in [1,\infty]$ and $q\in\ZZ_+$, and let
$B_1, B_2$ be finite-dimensional Banach spaces.  Let
$\mathfrak m : \RR^d \to L(B_1, B_2)$ be a bounded operator-valued
function supported on $[-1/2,1/2)^d/q$ and let
$\mathfrak m^{q}_{\mathrm{per}}$ be the periodic multiplier
\[
\mathfrak m^{q}_{\mathrm{per}}(\xi) : = \sum_{n\in\ZZ^d} \mathfrak m(\xi-n/q),\qquad \xi\in\TT^d.
\]
Then
\[
\|T_{\ZZ^d}[\mathfrak m^{q}_{\mathrm{per}}]\|_{\ell^{p}(\ZZ^d;B_1)\to \ell^{p}(\ZZ^d;B_2)}\le
C\|T_{\RR^d}[\mathfrak m]\|_{L^{p}(\RR^d;B_1)\to L^{p}(\RR^d;B_2)}.
\]
\end{proposition}
The proof can be found in \cite[Corollary 2.1, pp. 196]{MSW}.
We also refer to \cite{MSZ1} for a generalization of Proposition \ref{prop:msw} to real interpolation spaces.
We emphasize that $B_1$ and  $B_2$ are general (finite dimensional) Banach spaces in Proposition \ref{prop:msw},
in contrast to the Hilbert space-valued multipliers appearing in Theorem \ref{thm:IW} and so Proposition \ref{prop:msw} 
includes maximal function formulations and can also accommodate oscillation semi-norms.

\subsection{One-parameter semi-norm variant of Theorem \ref{thm:IW}}
Let $\Lambda:=\{\lambda_1,\ldots,\lambda_k\}\subset\ZZ_+$ be a  set  of size $k\in\ZZ_+$ of natural exponents,
 and consider the associated one-parameter family of dilations which for every $x\in\RR^k$, is defined by
\begin{align*}
(0,\infty)\ni t\mapsto t\circ x:=(t^{\lambda_1}x_1,\ldots,t^{\lambda_k}x_k)\in\RR^k.
\end{align*}
Let  $\Upsilon:=(\Upsilon_n:\RR^k\to\CC: n\in\NN)$ be a sequence of measurable
functions which define a positive sequence of operators in the sense that for every $n\in\NN$, one has
\begin{align}
\label{eq:377}
T_{\RR^k}[\Upsilon_n]f\ge0 \quad\text{if}\quad f\ge0.
\end{align}
Furthermore suppose there exist  $C_{\Upsilon}>0$,
 $0<\delta_{\Upsilon}<1$ and $1<\tau\le 2$ such that for every $\xi\in\RR^k$ and
$n\in\NN$, one has
\begin{align}
\label{eq:378}
\abs{\Upsilon_n(\xi)}&\le
C_{\Upsilon}\min\big\{1, \abs{\tau^n\circ \xi}^{-\delta_{\Upsilon}}\big\},\\
\label{eq:379}
\abs{\Upsilon_n(\xi)-1}&\le
C_{\Upsilon}\min\big\{1, \abs{\tau^n\circ \xi}^{\delta_{\Upsilon}}\big\}.
\end{align}
The condition \eqref{eq:377} implies that the operator  $T_{\RR^k}[\Upsilon_n] f = f *\mu_n$ is convolution
with positive measure $\mu_n$ and condition \eqref{eq:379} implies $\Upsilon_n(0) = 1$ and so each $\mu_n$ 
is a probability measure. Hence for every
$p\in[1, \infty)$, 
\begin{align}
\label{eq:380}
A_p^{\Upsilon}:=\sup_{n\in\NN}\|T_{\RR^k}[\Upsilon_n]\|_{L^p(\RR^k)\to L^p(\RR^k)}\le 1.
\end{align}
In this generality, $L^p(\RR^k)$ estimates with $1<p\le \infty$ for the maximal function 
$\sup_{n\in\NN}|T_{\RR^k}[\Upsilon_n] f(x)|$ were obtained in \cite{DR}
and corresponding $r$-variational and jump inequalites were established in
\cite{jsw} (see also \cite{MSZ2}). Here we extend these results further.

For $d\in\ZZ_+$ consider $\Gamma:=\{i_1,\ldots, i_k\}\subseteq [d]$ of size $k\in[d]$ and define a $\Gamma$-lifted sequence of measurable
functions $\Upsilon^{\Gamma}:=(\Upsilon_n^{\Gamma}:\RR^d\to\CC: n\in\NN)$ by setting
\begin{align*}
\Upsilon_n^{\Gamma}(\xi):=\Upsilon_n(\xi_{i_1},\ldots, \xi_{i_k}) \quad \text{ for } \quad \xi=(\xi_1,\ldots, \xi_d)\in\RR^d.
\end{align*}

Our first main result is the following  one-parameter semi-norm variant of Theorem
\ref{thm:IW}.

\begin{theorem}
\label{thm:IW1}
Let $d\in\ZZ_+$ and $\Gamma\subseteq[d]$ of size $k\in[d]$ be given. Let $\Upsilon=(\Upsilon_n:\RR^k\to\CC: n\in\NN)$ be a
sequence of measurable functions satisfying conditions \eqref{eq:377},
\eqref{eq:378} and \eqref{eq:379}, and let
$\Upsilon^{\Gamma}:=(\Upsilon_n^{\Gamma}:\RR^d\to\CC: n\in\NN)$ be the corresponding $\Gamma$-lifted sequence.
For every
$\rho\in(0, 1)$ and for every $p \in (1,\infty)$, there exists an
absolute constant
$0<C=C(d, p, \rho, \tau, \Gamma, A_p^{\Upsilon}, C_{\Upsilon})<\infty$ such
that for every integer $l\in\NN$ and $m\le -10C_{\rho}2^{2\rho l}$
the following holds. If
\begin{align}
\label{eq:167}
\supp \Upsilon_n\subseteq 2^m[-1/2, 1/2)^k
\quad \text{ for all }\quad n\in\NN,
\end{align}
then for every
$f=(f_{\iota}:\iota\in\NN)\in \ell^p(\ZZ^d; \ell^2(\NN))$ one has
\begin{align}
\label{eq:381}
\sup_{J\in\ZZ_+}\sup_{I\in\mathfrak S_J(\NN)}
\norm[\bigg]{\Big(\sum_{\iota\in\NN}O_{I, J}\big(T_{\ZZ^d}\big[\Theta_{\Sigma_{\le l}^d}[\Upsilon_n^{\Gamma}\eta_{\le m}^{\Gamma^c}]\big]f_{\iota}:n\in\NN\big)^2\Big)^{1/2}}_{\ell^p(\ZZ^{d})}
\leq C
(l+1)
\|f\|_{\ell^p(\ZZ^{d}; \ell^2(\NN))},
\end{align}
with $\Theta_{\Sigma^d_{\le l}}$  defined in \eqref{eq:375}.
In particular, \eqref{eq:381} implies the maximal estimate
\begin{align*}
\norm[\bigg]{\Big(\sum_{\iota\in\NN}\sup_{n\in\NN}\big|T_{\ZZ^d}\big[\Theta_{\Sigma_{\le l}^d}[\Upsilon_n^{\Gamma}\eta_{\le m}^{\Gamma^c}]\big]f_{\iota}\big|^2\Big)^{1/2}}_{\ell^p(\ZZ^{d})}
\leq C
(l+1) 
\|f\|_{\ell^p(\ZZ^{d}; \ell^2(\NN))}.
\end{align*}
\end{theorem}

Some remarks about Theorem \ref{thm:IW1} are in order.
\begin{enumerate}[label*={\arabic*}.]
\item Theorem \ref{thm:IW1} is a semi-norm variant of the Ionescu--Wainger \cite{IW} theorem for oscillations. The proof below works also for $r$-variations or jumps in place of oscillations as well as for norms corresponding to real interpolation spaces. We refer to \cite{MSZ1} for definitions.
\item In practice Theorem \ref{thm:IW1} will be applied with $\Gamma=[d]$. However, the concept of  $\Gamma$-lifted sequences is introduced here for further references. 
\item A careful inspection of the proof below allows us to show that the conclusion of Theorem \ref{thm:IW1} also holds in $\RR^d$.
For every $d\in\ZZ_+$, every sequence $\Upsilon=(\Upsilon_n:\RR^d\to\CC: n\in\ZZ)$ 
 of measurable functions satisfying conditions \eqref{eq:377},
\eqref{eq:378}, \eqref{eq:379} and \eqref{eq:380}, and 
for every $p\in(1, \infty)$, there exists a constant $C>0$ such that for every
$f=(f_{\iota}:\iota\in\NN)\in L^p(\RR^d; \ell^2(\NN))$ one has
\begin{align}
\label{eq:108}
\sup_{J\in\ZZ_+}\sup_{I\in\mathfrak S_J(\ZZ)}
\norm[\bigg]{\Big(\sum_{\iota\in\NN}O_{I, J}\big(T_{\RR^d}[\Upsilon_n]f_{\iota}:n\in\ZZ\big)^2\Big)^{1/2}}_{L^p(\RR^{d})}
\leq C
\|f\|_{L^p(\RR^{d}; \ell^2(\NN))}.
\end{align}
An important feature of our approach is that we do not need to invoke the corresponding inequality for martingales  in the proof. This stands in a sharp contrast to variants of inequality \eqref{eq:108} involving $r$-variations, where all arguments to the best of our knowledge use the corresponding $r$-variational inequalities for martingales. 
\end{enumerate}

\begin{proof}[Proof of Theorem \ref{thm:IW1}]
Fix $p\in(1, \infty)$ and  a sequence 
$f=(f_\iota:\iota\in\NN)\in\ell^2(\ZZ^d; \ell^2(\NN))\cap\ell^p(\ZZ^d; \ell^2(\NN))$.
For each $l\in\NN$ define an integer
\begin{align}
\label{eq:383}
\kappa_l:=\big\lfloor \big(100 C_{\rho}+ \log_2(\delta_{\Upsilon}\log_2\tau )^{-1}\big)(l+1)\big\rfloor+2,
\end{align}
where $C_\rho$ is the constant from Remark \ref{rem:2}, see property \ref{IW0}.
By \eqref{eq:135} it only suffices to establish \eqref{eq:381}, which will follow from the oscillation inequalities respectively for small  scales
\begin{align}
\label{eq:384}
\sup_{J\in\ZZ_+}\sup_{I\in\mathfrak S_J(\NN_{<2^{\kappa_l}})}\norm[\bigg]{\Big(\sum_{\iota\in\NN}O_{I, J}\big(T_{\ZZ^d}\big[\Theta_{\Sigma_{\le l}^d}[\Upsilon_n^{\Gamma}\eta_{\le m}^{\Gamma^c}]\big]f_{\iota}:n\in\NN_{<2^{\kappa_l}}\big)^2\Big)^{1/2}}_{\ell^p(\ZZ^{d})}
\lesssim(l+1)
\|f\|_{\ell^p(\ZZ^{d}; \ell^2(\NN))},
\end{align}
and large scales
\begin{align}
\label{eq:385}
\sup_{J\in\ZZ_+}\sup_{I\in\mathfrak S_J(\NN_{\ge2^{\kappa_l}})}\norm[\bigg]{\Big(\sum_{\iota\in\NN}O_{I, J}\big(T_{\ZZ^d}\big[\Theta_{\Sigma_{\le l}^d}[\Upsilon_n^{\Gamma}\eta_{\le m}^{\Gamma^c}]\big]f_{\iota}:n\in\NN_{\ge2^{\kappa_l}}\big)^2\Big)^{1/2}}_{\ell^p(\ZZ^{d})}
\lesssim 
\|f\|_{\ell^p(\ZZ^{d}; \ell^2(\NN))},
\end{align}
\paragraph{\bf Step 1}
We now prove inequality \eqref{eq:384}. 
 We fix $J\in\ZZ_+$
and a sequence $I\in\mathfrak S_J(\NN_{<2^{\kappa_l}})$. Then, by the Rademacher--Menshov inequality \eqref{eq:164}, we see that
\begin{align*}
&\norm[\bigg]{\Big(\sum_{\iota\in\NN}O_{I, J}\big(T_{\ZZ^d}\big[\Theta_{\Sigma_{\le l}^d}[\Upsilon_n^{\Gamma}\eta_{\le m}^{\Gamma^c}]\big]f_{\iota}:n\in\NN_{<2^{\kappa_l}}\big)^2\Big)^{1/2}}_{\ell^p(\ZZ^{d})}\\
&\qquad \lesssim \sum_{v=0}^{\kappa_l}\norm[\bigg]{\Big(\sum_{\iota\in\NN}\sum_{u=0}^{2^{\kappa_l-v}-1}\abs[\big]{\sum_{n\in U_u^v}
T_{\ZZ^d}\big[\Theta_{\Sigma_{\le l}^d}[(\Upsilon_{n+1}^{\Gamma}-\Upsilon_{n}^{\Gamma})\eta_{\le m}^{\Gamma^c}]\big]f_{\iota}}^2\Big)^{1/2}}_{\ell^p(\ZZ^{d})},
\end{align*}
where $U_u^v=[u2^v, (u+1)2^v)\cap\ZZ$. Hence it suffices to prove
\begin{align}
\label{eq:152}
\norm[\bigg]{\Big(\sum_{\iota\in\NN}\sum_{u=0}^{2^{\kappa_l-v}-1}\abs[\big]{\sum_{n\in U_u^v}T_{\ZZ^d}\big[\Theta_{\Sigma_{\le l}^d}[(\Upsilon_{n+1}^{\Gamma}-\Upsilon_{n}^{\Gamma})\eta_{\le m}^{\Gamma^c}]\big]f_{\iota}}^2\Big)^{1/2}}_{\ell^p(\ZZ^{d})}\lesssim
\|f\|_{\ell^p(\ZZ^{d}; \ell^2(\NN))},
\end{align}
uniformly in $v$. 
By Theorem \ref{thm:IW} and by our choice of $\kappa_l$ in \eqref{eq:383}, since $m\le -10C_{\rho}2^{2\rho l}$, \eqref{eq:152} will follow
if for every sequence 
$(f_\iota:\iota\in\NN)\in L^2(\RR^d; \ell^2(\NN))\cap L^p(\RR^d; \ell^2(\NN))$,
\begin{align}
\label{eq:165}
\norm[\bigg]{\Big(\sum_{\iota\in\NN}\sum_{u=0}^{2^{\kappa_l-v}-1}\abs[\big]{\sum_{n\in U_u^v}
T_{\RR^d}[(\Upsilon_{n+1}^{\Gamma}-\Upsilon_{n}^{\Gamma})\eta_{\le m}^{\Gamma^c}]f_{\iota}}^2\Big)^{1/2}}_{L^p(\RR^{d})}\lesssim
\|f\|_{L^p(\RR^{d}; \ell^2(\NN))}
\end{align}
holds uniformly in $v$.

To
prove inequality \eqref{eq:165}, in view of Lemma \ref{lem:10}, it
suffices to show that for every $p\in(1, \infty)$ and for every
$f\in L^p(\RR^d)$ one has
\begin{align}
\label{eq:166}
\sup_{(\omega_u)\in\{-1, 1\}^{\NN}}\norm[\Big]
{\sum_{u=0}^{2^{\kappa_l-v}-1}
\omega_u \bigl[ T_{\RR^d}[(\Upsilon_{(u+1)2^v}^{\Gamma}-\Upsilon_{u2^v}^{\Gamma})\eta_{\le m}^{\Gamma^c}
]f\bigr]}_{L^p(\RR^d)}
\lesssim_p\norm{f}_{L^p(\RR^d)},
\end{align}
unformly in $v \in [0,\kappa_l]$ and $l$.
The proof of \eqref{eq:166}, using conditions 
\eqref{eq:377}, \eqref{eq:378}, \eqref{eq:379} and \eqref{eq:380}, follows from standard Littlewood--Paley theory as developed in
\cite{DR}. We refer for instance to \cite{MSZ2} for details in this context.

\paragraph{\bf Step 2}
We now prove inequality \eqref{eq:385}. By the  support condition \eqref{eq:167},  we may write (see 
property \ref{IW0} from Remark \ref{rem:2}) 
\begin{align*}
T_{\ZZ^d}\big[\Theta_{\Sigma_{\le l}^d}[\Upsilon_n^{\Gamma}\eta_{\le m}^{\Gamma^c}]\big]=
T_{\ZZ^d}\big[\Theta_{\Sigma_{\le l}^d}[\Upsilon_n^{\Gamma}(1-\eta_{\le-2^{2C_\rho l}}^{\Gamma})\eta_{\le m}^{\Gamma^c}]\big]
+
T_{\ZZ^d}\big[\Theta_{\Sigma_{\le l}^d}[\Upsilon_n^{\Gamma}\eta_{\le-2^{2C_\rho l}}^{\Gamma}\eta_{\le m}^{\Gamma^c}]\big],
\end{align*}
where $\eta_{\le-2^{2C_\rho l}}^{\Gamma}:=\prod_{i\in\Gamma}\eta_{\le-2^{2C_\rho l}}^{(i)}$, (see definition \eqref{eq:78}).
The proof of \eqref{eq:385} will be complete if we show \eqref{eq:385} with 
$T_{\ZZ^d}\big[\Theta_{\Sigma_{\le l}^d}[\Upsilon_n^{\Gamma}(1-\eta_{\le-2^{2C_\rho l}}^{\Gamma})\eta_{\le m}^{\Gamma^c}]\big]$, and 
$T_{\ZZ^d}\big[\Theta_{\Sigma_{\le l}^d}[\Upsilon_n^{\Gamma}\eta_{\le-2^{2C_\rho l}}^{\Gamma}\eta_{\le m}^{\Gamma^c}]\big]$ in place of 
$T_{\ZZ^d}\big[\Theta_{\Sigma_{\le l}^d}[\Upsilon_n^{\Gamma}\eta_{\le m}^{\Gamma^c}]\big]$.
To establish \eqref{eq:385}  with
$T_{\ZZ^d}\big[\Theta_{\Sigma_{\le l}^d}[\Upsilon_n^{\Gamma}(1-\eta_{\le-2^{2C_\rho l}}^{\Gamma})\eta_{\le m}^{\Gamma^c}]\big]$, it suffices to prove that for every $p\in(1, \infty)$ there exists $\delta_p\in(0, 1)$ such that for every $n\ge 2^{\kappa_l}$ and every $f=(f_\iota:\iota\in\NN)\in\ell^2(\ZZ^d; \ell^2(\NN))\cap\ell^p(\ZZ^d; \ell^2(\NN))$ one has
\begin{align}
\label{eq:92}
\norm[\bigg]{\Big(\sum_{\iota\in\NN}\big|T_{\ZZ^d}\big[\Theta_{\Sigma_{\le l}^d}[\Upsilon_n^{\Gamma}(1-\eta_{\le-2^{2C_\rho l}}^{\Gamma})\eta_{\le m}^{\Gamma^c}]\big]f_{\iota}\big|^2\Big)^{1/2}}_{\ell^p(\ZZ^{d})}
\lesssim\tau^{-\delta_p n} 
\|f\|_{\ell^p(\ZZ^{d}; \ell^2(\NN))}.
\end{align}
Inequality \eqref{eq:92}, in view of Lemma \ref{lem:10} and Theorem \ref{thm:IW}, can be reduced to showing that for every $p\in(1, \infty)$ there exists $\delta_p\in(0, 1)$ such that
\begin{align}
\label{eq:93}
\|T_{\RR^d}[\Upsilon_n^{\Gamma}(1-\eta_{\le-2^{2C_\rho l}}^{\Gamma})\eta_{\le m}^{\Gamma^c}]f\|_{L^p(\RR^d)}\lesssim
\tau^{-\delta_p n} 
\|f\|_{L^p(\RR^{d})}
\end{align}
holds for every $n\ge 2^{\kappa_l}$.
By interpolation it suffices to prove \eqref{eq:93} for $p=2$ and by Plancherel’s theorem, this reduces to showing that
\begin{align*}
|\Upsilon_n^{\Gamma}(\xi)(1-\eta_{\le-2^{2C_\rho l}}^{\Gamma}(\xi))\eta_{\le m}^{\Gamma^c}(\xi)|\lesssim \tau^{-\delta_{\Upsilon}n/2}\quad 
\end{align*}
holds uniformly in $\xi$ for all $n\ge 2^{\kappa_l}$. This follows from the definition of $\kappa_l$ and \eqref{eq:378}.
\paragraph{\bf Step 3}
We now establish \eqref{eq:385} with 
$T_{\ZZ^d}\big[\Theta_{\Sigma_{\le l}^d}[\Upsilon_n^{\Gamma}\eta_{\le-2^{2C_\rho l}}^{\Gamma}\eta_{\le m}^{\Gamma^c}]\big]$ in place of 
$T_{\ZZ^d}\big[\Theta_{\Sigma_{\le l}^d}[\Upsilon_n^{\Gamma}\eta_{\le m}^{\Gamma^c}]\big]$.
Taking $Q_{\le l}$ from property \ref{IW0}, note that 
\begin{align*}
T_{\ZZ^d}\big[\Theta_{\Sigma_{\le l}^d}[\Upsilon_n^{\Gamma}\eta_{\le-2^{2C_\rho l}}^{\Gamma}\eta_{\le m}^{\Gamma^c}]\big]=
T_{\ZZ^d}\big[\Theta_{Q_{\le l}^{-1}[Q_{\le l}]^k}^{\Gamma}[\Upsilon_n^{\Gamma}\eta_{\le-2^{2C_\rho l}}^{\Gamma}]\big]
T_{\ZZ^d}\big[\Theta_{\Sigma_{\le l}^d}[\eta_{\le m}^{[d]}]\big].
\end{align*}

Using this factorization 
it suffices to show that
\begin{align}
\label{eq:169}
\norm[\bigg]{\Big(\sum_{\iota\in\NN}\abs[\big]{T_{\ZZ^d}\big[\Theta_{\Sigma_{\le l}^d}[\eta_{\le m}^{[d]}]\big]f_{\iota}}^2\Big)^{1/2}}_{\ell^p(\ZZ^{d})}\lesssim
 \|f\|_{\ell^p(\ZZ^{d}; \ell^2(\NN))},
\end{align}
 and
\begin{align}
\label{eq:170}
\begin{gathered}
\sup_{J\in\ZZ_+}\sup_{I\in\mathfrak S_J(\NN_{\ge2^{\kappa_l}})}\norm[\bigg]{\Big(\sum_{\iota\in\NN}O_{I, J}\big(T_{\ZZ^d}\big[\Theta_{Q_{\le l}^{-1}[Q_{\le l}]^k}^{\Gamma}[\Upsilon_n^{\Gamma}\eta_{\le-2^{2C_\rho l}}^{\Gamma}]\big]
f_{\iota}:n\in\NN_{\ge2^{\kappa_l}}\big)^2\Big)^{1/2}}_{\ell^p(\ZZ^{d})}\\
\lesssim 
\|f\|_{\ell^p(\ZZ^{d}; \ell^2(\NN))}.
\end{gathered}
\end{align}
By Lemma \ref{lem:10} and Theorem \ref{thm:IW}, the bound \eqref{eq:169} follows from
\begin{align*}
\big\|T_{\RR^d}[\eta_{\le m}^{[d]}]f\big\|_{L^p(\RR^d)}\lesssim_{p}\|f\|_{L^p(\RR^d)},
\end{align*}
which clearly holds for all $p\in[1, \infty]$. To prove \eqref{eq:170} we can use the sampling principle formulated in Proposition \ref{prop:msw} to reduce matters to proving
\begin{align}
\label{eq:178}
\sup_{J\in\ZZ_+}\sup_{I\in\mathfrak S_J(\NN_{\ge2^{\kappa_l}})}\norm[\bigg]{\Big(\sum_{\iota\in\NN}O_{I, J}\big(T_{\RR^k}[\Upsilon_n]
f_{\iota}:n\in\NN_{\ge2^{\kappa_l}}\big)^2\Big)^{1/2}}_{L^p(\RR^{k})}
\lesssim 
\|f\|_{L^p(\RR^{k}; \ell^2(\NN))}.
\end{align}
To do this we carefully choose the finite dimensional Banach spaces $B_1$ and $B_2$ in Proposition \ref{prop:msw} to accommodate the oscillation semi-norm $O_{I,J}$. See the remark after Proposition \ref{prop:msw}.

\paragraph{\bf Step 4} Let $\eta$ be a smooth function with $\ind{[-1,1]^k} \le \eta \le \ind{[-\tau,\tau]^k}$
and set $\chi_{n}(\xi):=  \eta(\tau^{-n} \circ \xi)$ 
Using  conditions  \eqref{eq:377},
\eqref{eq:378}, \eqref{eq:379} and \eqref{eq:380} we see that Theorem B in   \cite{DR} implies
\begin{align}
\label{eq:179}
\norm[\bigg]{\Big(\sum_{n\in\NN}\big|T_{\RR^k}[\Upsilon_n-\chi_{-n}]
f\big|^2\Big)^{1/2}}_{L^p(\RR^k)}
\lesssim 
\|f\|_{L^p(\RR^k)}
\end{align}
for $1<p<\infty$ since $|\Upsilon_n(\xi) - \chi_{-n}(\xi)| \lesssim \min(|\tau^n\circ \xi|, |\tau^n \circ \xi|^{-1})^{\delta_{\Upsilon}}$ and both maximal functions $\sup_{n\in\NN} |T_{\RR^k}[\Upsilon_n]f|$ and $\sup_{n\in\NN} |T_{\RR^k}[\chi_{-n}]f|$ are both bounded on all $L^q(\RR^k)$ for all $1<q<\infty$.

Using Lemma \ref{lem:10}, we see that inequality \eqref{eq:179} reduces \eqref{eq:178} to proving
\begin{align}
\label{eq:181}
\sup_{J\in\ZZ_+}\sup_{I\in\mathfrak S_J(\NN)}\norm[\bigg]{\Big(\sum_{\iota\in\NN}O_{I, J}\big(T_{\RR^k}[\chi_{n}]
f_{\iota}:n\in\ZZ\big)^2\Big)^{1/2}}_{L^p(\RR^k)}
\lesssim 
\|f\|_{L^p(\RR^k; \ell^2(\NN))}.
\end{align}
To prove \eqref{eq:181} we note that for every $m < n$ we have
\begin{align*}
\chi_{m}\chi_{n}=  \chi_{m}.
\end{align*}
 We fix $J\in\ZZ_+$
and a sequence $I\in\mathfrak S_J(\NN)$. Then
\begin{align*}
O_{I, J}\big(T_{\RR^k}[\chi_{n}]
f_{\iota}:n\in\NN\big)&\lesssim\Big(\sum_{j=0}^{J-1}\sup_{I_j\le n<I_{j+1}}\big|T_{\RR^k}[\chi_{n}-\chi_{ I_{j}}]f_{\iota}\big|^2\Big)^{1/2}\\
&= \Big(\sum_{j=0}^{J-1}\sup_{I_j< n< I_{j+1}}\big|T_{\RR^k}[\chi_{n}]T_{\RR^k}[\chi_{I_{j+1}}-\chi_{I_{j}}]f_{\iota}\big|^2\Big)^{1/2}\\
&\le\Big(\sum_{j\in\NN}\sup_{n\in\ZZ}\big(\varphi_n*\big|T_{\RR^k}[\chi_{I_{j+1}}-\chi_{I_{j}}]f_{\iota}\big|\big)^2\Big)^{1/2},
\end{align*}
where $\varphi_n(x):=|T_{\RR^k}[\chi_{n}](x)|$.
Using this estimate and the Fefferman-Stein vector-valued maximal function estimate (see \cite{Stein}), we conclude that
\begin{multline}
\label{eq:196}
\sup_{J\in\ZZ_+}\sup_{I\in\mathfrak S_J(\NN)}\norm[\bigg]{\Big(\sum_{\iota\in\NN}O_{I, J}\big(T_{\RR^k}[\chi_{n}]
f_{\iota}:n\in\ZZ\big)^2\Big)^{1/2}}_{L^p(\RR^k)}\\
\le\sup_{I\in\mathfrak S_{\infty}(\NN)}\norm[\bigg]{\Big(\sum_{\iota\in\NN}\sum_{j\in\NN}\sup_{n\in\ZZ}\big(\varphi_n*\big|T_{\RR^k}[\chi_{ I_{j+1}}-\chi_{I_{j}}]f_{\iota}\big|\big)^2\Big)^{1/2}}_{L^p(\RR^k)}\\
\lesssim_{p}\sup_{I\in\mathfrak S_{\infty}(\NN)}\norm[\bigg]{\Big(\sum_{\iota\in\NN}\sum_{j\in\NN}\big|T_{\RR^k}[\chi_{I_{j+1}}-\chi_{I_{j}}]f_{\iota}\big|^2\Big)^{1/2}}_{L^p(\RR^k)}.
\end{multline}
As above, using Theorem B in \cite{DR}, we see that for every $p\in(1, \infty)$,
\begin{align}
\label{eq:168}
\sup_{I\in\mathfrak S_{\infty}(\NN)}\norm[\bigg]{\Big(\sum_{j\in\NN}\big|T_{\RR^k}[\chi_{I_{j+1}}-\chi_{I_{j}}]f\big|^2\Big)^{1/2}}_{L^p(\RR^k)}\lesssim_{p}\|f\|_{L^p(\RR^k)}.
\end{align}
Then invoking  \eqref{eq:168} and Lemma \ref{lem:10} we obtain
\begin{align}
\label{eq:188}
\sup_{I\in\mathfrak S_{\infty}(\NN)}\norm[\bigg]{\Big(\sum_{\iota\in\NN}\sum_{j\in\NN}\big|T_{\RR^k}[\chi_{I_{j+1}}-\chi_{I_{j}}]f_{\iota}\big|^2\Big)^{1/2}}_{L^p(\RR^k)}\lesssim_p\|f\|_{L^p(\RR^k; \ell^2(\NN))}.
\end{align}
Combining \eqref{eq:196} with \eqref{eq:188} we obtain the desired claim in \eqref{eq:181} and this completes the proof of Theorem \ref{thm:IW1}.
\end{proof}
\subsection{Multi-parameter semi-norm variant of Theorem \ref{thm:IW}} We will generalize Theorem \ref{thm:IW1} to the multi-parameter setting for a class of multipliers arising in our question. We formulate our main result in the two-parameter setting, but all arguments are adaptable to multi-parameter settings.

Let $P\in\RR[\rm m_1, \rm m_2]$ be a polynomial with $\deg P\ge2$ such that
\begin{align}
\label{eq:111}
P(m_1, m_2):=\sum_{(\gamma_1, \gamma_2)\in S_P}c_{\gamma_1, \gamma_2}m_1^{\gamma_1}m_2^{\gamma_2},
\end{align}
where $c_{(0,0)}=0$. In addition, we assume that $P$ 
is  non-degenerate in the sense that $S_P\cap(\ZZ_+\times\ZZ_+)\neq\emptyset$, see the remark below Theorem \ref{thm:main}.
Let $r\in\ZZ_+$ be the number of vertices in the backwards Newton diagram $N_{P}$ corresponding to the polynomial $P$ from \eqref{eq:111}.
For any vertex $v_j=(v_{j, 1}, v_{j, 2})$ of $N_P$ we denote the associated monomial by
\begin{align}
\label{eq:112}
P^j(m_1, m_2):=c_{(v_{j, 1}, v_{j, 2})}m_1^{v_{j, 1}}m_2^{v_{j, 2}}.
\end{align}
From Section \ref{section:4} (see Remark \ref{rem:2-0}) we know that $P^j$ is the main monomial in the sector $S(j)$ for $j\in[r]$.

We fix the lacunarity factor $\tau>1$. 
Throughout this subsection  we  allow all the implied constants to depend on $\tau$.
For real numbers $M_1, M_2\ge1$ and $\xi\in\RR$, we consider the multiplier
\begin{align}
\label{eq:197}
\mathfrak m_{M_1, M_2}^P(\xi):=\frac{1}{(1-\tau^{-1})^2}\int_{\tau^{-1}}^1\int_{\tau^{-1}}^1\ex(P_{\xi}(M_1y_1, M_2y_2))dy_1dy_2,
\end{align}
where recall $P_{\xi}\in\RR[\rm m_1, \rm m_2]$ is defined as
$P_{\xi}(m_1,m_2) = \xi P(m_1,m_2)$. 

As an application of Theorem \ref{thm:IW1} we obtain the following two-parameter  oscillation inequality.

\begin{theorem}
\label{thm:IW2}
Let $\tau>1$ be given and let $(\mathfrak m_{M_1, M_2}^{P}: (M_1, M_2)\in\DD_{\tau}\times\DD_{\tau})$ be the two-parameter sequence of multipliers from \eqref{eq:197} corresponding to the polynomial $P$ from \eqref{eq:111}. Let $r\in\ZZ_+$ be the number of vertices in the backwards Newton diagram $N_{P}$.
 For every
$\rho\in(0, 1)$ and  $p \in (1,\infty)$ and any $j\in[r]$, there exists an
absolute constant
$0<C=C(p, \rho, \tau, j, P)<\infty$ such
that for every integers $l\in\NN$ and $m\le -10C_{\rho}2^{2\rho l}$ and for every
$f=(f_{\iota}:\iota\in\NN)\in \ell^p(\ZZ; \ell^2(\NN))$, one has
\begin{align}
\label{eq:9}
\begin{gathered}
\sup_{J\in\ZZ_+}\sup_{I\in\mathfrak S_J(\SS_{\tau}(j))}
\norm[\bigg]{\Big(\sum_{\iota\in\NN}O_{I, J}\big(T_{\ZZ}\big[\Theta_{\Sigma_{\le l}}[\mathfrak m_{M_1, M_2}^{P}\eta_{\le m}]\big]f_{\iota}:(M_1, M_2)\in \SS_{\tau}(j)\big)^2\Big)^{1/2}}_{\ell^p(\ZZ)}\\
\leq C
(l+1) 
\|f\|_{\ell^p(\ZZ; \ell^2(\NN))},
\end{gathered}
\end{align}
with $\Theta_{\Sigma_{\le l}}$  defined in \eqref{eq:375}.
In particular, \eqref{eq:9} also implies the maximal estimate
\begin{align*}
\norm[\bigg]{\Big(\sum_{\iota\in\NN}\sup_{(M_1, M_2)\in \SS_{\tau}(j)}\big|T_{\ZZ}\big[\Theta_{\Sigma_{\le l}}[\mathfrak m_{M_1, M_2}^{P}\eta_{\le m}]\big]f_{\iota}\big|^2\Big)^{1/2}}_{\ell^p(\ZZ)}
\leq C
(l+1)
\|f\|_{\ell^p(\ZZ; \ell^2(\NN))}.
\end{align*}
\end{theorem}
Some remarks about Theorem \ref{thm:IW2} are in order.
\begin{enumerate}[label*={\arabic*}.]

\item Theorem \ref{thm:IW2} is the simplest instance of a
multi-parameter oscillation variant of the Ionescu--Wainger theorem
\cite{IW}.  More general variants of Theorem \ref{thm:IW2} can be
also proved, for instance, an analogue of Theorem \ref{thm:IW2} for the following multipliers
\[
\ \ \ \ \ \ \mathfrak m_{M_1, M_2}^P(\xi_1, \xi_2, \xi_3)=\int_{0}^1\int_{0}^1\ex(\xi_1(M_1y_1)+\xi_2(M_2y_2)+\xi_3 P(M_1y_1, M_2y_2))dy_1dy_2,
\]
can be established using the methods of the paper. However, this goes beyond the scope of this paper and will be
discussed in the future.
\item In contrast to the one-parameter theory, it is not clear whether
multi-parameter $r$-variational or jump counterparts of Theorem
\ref{thm:IW2} are available. As far as we know it is not even clear if
there are useful multi-parameter definitions of $r$-variational or
jump semi-norms. From this point of view the multi-parameter
oscillation semi-norm is an invaluable tool allowing us to handle
pointwise convergence problems in the multi-parameter setting.

\item A careful inspection of the proof
allows us to establish an analogue of Theorem \ref{thm:IW2} in the continuous setting. Namely, for every $p\in(1, \infty)$ there is a constant $C>0$ such that for every $f=(f_{\iota}:\iota\in\NN)\in L^p(\RR; \ell^2(\NN))$ one has
\begin{align*}
\begin{gathered}
\sup_{J\in\ZZ_+}\sup_{I\in\mathfrak S_J(\SS_{\tau}(j))}
\norm[\bigg]{\Big(\sum_{\iota\in\NN}O_{I, J}\big(T_{\RR}[\mathfrak m_{M_1, M_2}^{P}]f_{\iota}:(M_1, M_2)\in \SS_{\tau}(j)\big)^2\Big)^{1/2}}_{L^p(\RR)}\\
\leq C \|f\|_{L^p(\RR; \ell^2(\NN))}.
\end{gathered}
\end{align*}
\end{enumerate}

\begin{proof}[Proof of Theorem \ref{thm:IW2}]
We will only prove Theorem \ref{thm:IW2} for  $j=r=1$ or for $1\le j<r$ with $r\ge2$. The same argument can be used to prove the case for $j=r$. In view of \eqref{eq:135} it suffices to prove \eqref{eq:9}.
We divide the proof  into two steps to make the argument clearer.
\vskip 5pt
\paragraph{\bf Step 1} 
We prove that for every $p\in(1, \infty)$ and every $f=(f_{\iota}:\iota\in\NN)\in \ell^p(\ZZ; \ell^2(\NN))$ one has
\begin{align*}
\norm[\bigg]{\Big(\sum_{(M_1, M_2)\in \SS_{\tau}(j)}\sum_{\iota\in\NN}\big|T_{\ZZ}\big[\Theta_{\Sigma_{\le l}}[(\mathfrak m_{M_1, M_2}^{P}-\mathfrak m_{M_1, M_2}^{P^j})\eta_{\le m}]\big]f_{\iota}\big|^2\Big)^{1/2}}_{\ell^p(\ZZ)}
\lesssim
\|f\|_{\ell^p(\ZZ; \ell^2(\NN))}.
\end{align*}
Using \eqref{eq:48} and \eqref{eq:49} it suffices to prove that for every $p\in(1, \infty)$  there is $\sigma_{j, p}\in(0, 1)$ such that for every $N\in\NN$, $i\in[2]$ and every $f=(f_{\iota}:\iota\in\NN)\in \ell^p(\ZZ; \ell^2(\NN))$ one has
\begin{align}
\label{eq:50}
\norm[\bigg]{\Big(\sum_{(M_1, M_2)\in \SS_{\tau, i}^N(j)}\sum_{\iota\in\NN}\big|T_{\ZZ}\big[\Theta_{\Sigma_{\le l}}[(\mathfrak m_{M_1, M_2}^{P}-\mathfrak m_{M_1, M_2}^{P^j})\eta_{\le m}]\big]f_{\iota}\big|^2\Big)^{1/2}}_{\ell^p(\ZZ)}
\lesssim \tau^{-\sigma_{j, p}N}
\|f\|_{\ell^p(\ZZ; \ell^2(\NN))}.
\end{align}
We only prove \eqref{eq:50} for $i=1$ as the proof for $i=2$ is  the same. By the construction of the sets $\SS_{\tau, 1}^N(j)$, see definition \eqref{eq:49}, the problem becomes a one-parameter problem. Indeed, if $(M_1, M_2)\in \SS_{\tau, 1}^N(j)$, then $(M_1, M_2)=(\tau^{n_1}, \tau^{n_2})$ and
\begin{align*}
(n_1, n_2)=\frac{n}{d_j}\omega_{j-1}+\frac{N}{d_j}(\omega_{j}+\omega_{j-1})
\quad \text{ for some } \quad
n \in\ZZ_+.
\end{align*}
Defining $(n_1^k, n_2^k):=\frac{k}{d_j}\omega_{j-1}+\frac{N}{d_j}(\omega_{j}+\omega_{j-1})$ for any $k\in\ZZ_+$, inequality \eqref{eq:50}
can be written as
\begin{align*}
\norm[\bigg]{\Big(\sum_{k\in\ZZ_+}\sum_{\iota\in\NN}\big|T_{\ZZ}\big[\Theta_{\Sigma_{\le l}}[(\mathfrak m_{\tau^{n_1^k}, \tau^{n_2^k}}^{P}-\mathfrak m_{\tau^{n_1^k}, \tau^{n_2^k}}^{P^j})\eta_{\le m}]\big]f_{\iota}\big|^2\Big)^{1/2}}_{\ell^p(\ZZ)}
\lesssim \tau^{-\sigma_{j, p}N}
\|f\|_{\ell^p(\ZZ; \ell^2(\NN))}.
\end{align*}
By Lemma \ref{lem:10} and Theorem \ref{thm:IW} it suffices to prove that for every $p\in(1, \infty)$ 
there is $\sigma_{j, p}\in(0, 1)$ such that for every $N\in\NN$ and $f\in L^p(\RR)$, one has
\begin{align}
\label{eq:52}
\sup_{(\varepsilon_k: k\in\ZZ_+)\in\{0, 1\}^{\ZZ_+}}\norm[\Big]{\sum_{k\in\ZZ_+}\varepsilon_kT_{\RR}\big[(\mathfrak m_{\tau^{n_1^k}, \tau^{n_2^k}}^{P}-\mathfrak m_{\tau^{n_1^k}, \tau^{n_2^k}}^{P^j})\big]f}_{L^p(\RR)}
\lesssim \tau^{-\sigma_{j, p}N} 
\|f\|_{L^p(\RR)}.
\end{align}
By \eqref{eq:111}, \eqref{eq:112} and Lemma \ref{lem:30} we obtain
\begin{align*}
|P(\tau^{n_1}y_1, \tau^{n_2}y_2)-P^j(\tau^{n_1}y_1, \tau^{n_2}y_2)|&\le \sum_{(\gamma_1, \gamma_2)\in S_P\setminus\{v_j\}}|c_{\gamma_1, \gamma_2}|\tau^{(\gamma_1, \gamma_2)\cdot(n_1,n_2)}|y_1|^{\gamma_1}|y_2|^{\gamma_2}\\
&\le (\sup_{v\in S_P}|c_{v}|)\tau^{(n_1, n_2)\cdot v_j}\sum_{v\in S_P\setminus\{v_j\}} \tau^{(n_1, n_2)\cdot (v-v_j)}\\
& \le \#S_P(\sup_{v\in S_P}|c_{v}|)\tau^{(n_1, n_2)\cdot v_j}
\tau^{-\sigma_j N}
\end{align*}
whenever $|y_1|, |y_2|\le 1$, with $\sigma_j>0$ defined in \eqref{eq:47}. Consequently, we have
\begin{align}
\label{eq:53}
|\mathfrak m_{\tau^{n_1^k}, \tau^{n_2^k}}^{P}(\xi)-\mathfrak m_{\tau^{n_1^k}, \tau^{n_2^k}}^{P^j}(\xi)|\lesssim_P
\tau^{-\sigma_j N}(\tau^{(n_1^k, n_2^k)\cdot v_j}
|\xi|).
\end{align}
Moreover by van der Corput's lemma (Proposition \ref{thm:CCW}), we can find a $\delta_0\in (0, 1)$ such that
\begin{align}
\label{eq:59}
|\mathfrak m_{\tau^{n_1^k}, \tau^{n_2^k}}^{P}(\xi)-\mathfrak m_{\tau^{n_1^k}, \tau^{n_2^k}}^{P^j}(\xi)|\lesssim_P
(\tau^{(n_1^k, n_2^k)\cdot v_j}
|\xi|)^{-\delta_0}
\end{align}
for sufficiently large $N\in\NN$. A convex combination of \eqref{eq:53} and \eqref{eq:59} gives
\begin{align}
\label{eq:60}
|\mathfrak m_{\tau^{n_1^k}, \tau^{n_2^k}}^{P}(\xi)-\mathfrak m_{\tau^{n_1^k}, \tau^{n_2^k}}^{P^j}(\xi)|\lesssim_P
\tau^{-\sigma_j' N}\min\big\{(\tau^{(n_1^k, n_2^k)\cdot v_j}
|\xi|)^{\delta_0'}, (\tau^{(n_1^k, n_2^k)\cdot v_j}
|\xi|)^{-\delta_0'}\big\},
\end{align}
for some $\delta_0', \sigma_j'\in(0, 1)$.

Using \eqref{eq:60} and Plancherel's theorem we obtain \eqref{eq:52} for $p=2$. Standard Littlewood--Paley theory arguments 
(see for example Theorem D in \cite{DR}) allows us then to obtain \eqref{eq:52} for all $p\in(1, \infty)$. 
\vskip 5pt
\paragraph{\bf Step 2} The argument from the first step allows us to reduce matters to proving
\begin{align*}
\begin{gathered}
\sup_{J\in\ZZ_+}\sup_{I\in\mathfrak S_J(\DD_{\tau}^2)}
\norm[\bigg]{\Big(\sum_{\iota\in\NN}O_{I, J}\big(T_{\ZZ}\big[\Theta_{\Sigma_{\le l}}[\mathfrak m_{M_1, M_2}^{P^j}\eta_{\le m}]\big]f_{\iota}:(M_1, M_2)\in \DD_{\tau}^2\big)^2\Big)^{1/2}}_{\ell^p(\ZZ)}\\
\lesssim (l+1)
\|f\|_{\ell^p(\ZZ; \ell^2(\NN))}.
\end{gathered}
\end{align*}
We define a new one-parameter multiplier
\begin{align*}
\mathfrak g_{M}^{P^j}(\xi):=\frac{1}{(1-\tau^{-1})^2}\int_{\tau^{-1}}^1\int_{\tau^{-1}}^1
\ex(c_{(v_{j, 1}, v_{j, 2})} M\xi y_1^{v_{j, 1}}y_2^{v_{j, 2}})dy_1dy_2.
\end{align*}
Observe that by Theorem \ref{thm:IW1} we obtain
\begin{gather*}
\sup_{J\in\ZZ_+}\sup_{I\in\mathfrak S_J(\DD_{\tau}^2)}
\norm[\bigg]{\Big(\sum_{\iota\in\NN}O_{I, J}\big(T_{\ZZ}\big[\Theta_{\Sigma_{\le l}}[\mathfrak m_{M_1, M_2}^{P^j}\eta_{\le m}]\big]f_{\iota}:(M_1, M_2)\in \DD_{\tau}^2\big)^2\Big)^{1/2}}_{\ell^p(\ZZ)}\\
\le
\sup_{J\in\ZZ_+}\sup_{I\in\mathfrak S_J(\DD_{\tau})}
\norm[\bigg]{\Big(\sum_{\iota\in\NN}O_{I, J}\big(T_{\ZZ}\big[\Theta_{\Sigma_{\le l}}\mathfrak g_{M}^{P^j}\eta_{\le m}\big]f_{\iota}: M\in \DD_{\tau}\big)^2\Big)^{1/2}}_{\ell^p(\ZZ)}
\lesssim (l+1)
\|f\|_{\ell^p(\ZZ; \ell^2(\NN))}.
\end{gather*}
This completes the proof of the theorem.
\end{proof}

\section{Two-parameter circle method: Proof of Theorem \ref{thm:main'''}}
\label{section:7}
Throughout this section $\tau>1$ is fixed and we  allow all the implied constants to depend on $\tau$.  Let $P\in\ZZ[\rm m_1, \rm m_2]$ be a polynomial obeying
$P(0, 0)=0$, which is  non-degenerate in the sense that $S_P\cap(\ZZ_{+}\times\ZZ_{+})\neq\emptyset$; see \eqref{eq:66}.
For every real number $N\ge1$ define
\begin{align*}
\chi_N(x):=\frac{1}{|(\tau^{-1}N, N]\cap\ZZ|}\ind{(\tau^{-1}N, N]}(x), \quad x\in \RR.
\end{align*}
For every real numbers $M_1, M_2\ge1$ and $\xi\in\RR$ we consider the multiplier
\begin{align*}
m_{M_1, M_2}(\xi):=\sum_{m_1\in\ZZ}\sum_{m_2\in\ZZ}\ex(P_{\xi}(m_1, m_2))\chi_{M_1}(m_1)\chi_{M_2}(m_2),
\end{align*}
with $P_{\xi}(m_1,m_2) = \xi P(m_1,m_2)$. The corresponding partial multipliers are defined by
\begin{align}
\label{eq:128}
\begin{split}
m_{m_1, M_2}^1(\xi):=&\sum_{m_2\in\ZZ}\ex(P_{\xi}(m_1, m_2))\chi_{M_2}(m_2), \qquad m_1\in\ZZ,\\
m_{M_1, m_2}^2(\xi):=&\sum_{m_1\in\ZZ}\ex(P_{\xi}(m_1, m_2))\chi_{M_1}(m_1), \qquad m_2\in\ZZ.
\end{split}
\end{align}

We fix further notation and terminology. 
For functions $G:\QQ\cap\TT\to \CC$, $\mathfrak m:\TT\to\CC$, a finite set $\Sigma\subset \QQ\cap \TT$, any $n\in\ZZ$ and any $\xi\in\TT$ we define the following  $1$-periodic multiplier
\begin{align}\label{Phi}
\Phi_{\le n}^{\Sigma}[G, \mathfrak m](\xi):=\sum_{a/q\in\Sigma}G(a/q)\mathfrak m(\xi-a/q)\eta_{\le n}(\xi-a/q).
\end{align}

In a similar way, for any $l\in\NN$, $n\in\ZZ$, any $\xi\in\TT$  we define the following projection multipliers
(recall the definition of $\Sigma_{\le l}:=\Sigma_{\le l}^1$ from \eqref{eq:372})
\begin{align*}
\Delta_{\le l, \le n}(\xi):=\sum_{a/q\in\Sigma_{\le l}}\eta_{\le n}(\xi-a/q),
\qquad \text{ and } \qquad
\Delta_{\le l, \le n}^{c}(\xi):=1-\Delta_{\le l, \le n}(\xi).
\end{align*}

All these multipliers  will be applied with different choices of parameters. 
 For $\beta>0$, $M_1, M_2, M>0$, $N\ge0$, and $v=(v_1, v_2)\in\ZZ^2$ we define
\begin{align}
\label{eq:228}
l^{\beta}(M):= \log_2\big((\log_{\tau} M)^{\beta}\big),
\qquad\text{ and }\qquad
n_{M_1, M_2}^{v}(N):=\log_2 (M_1^{v_{ 1}}M_2^{v_{ 2}})-N.
\end{align}
Using \eqref{eq:228} we also set
\begin{align}
\label{eq:228'}
n_{M_1, M_2}^{v, \beta}(M):=n_{M_1, M_2}^{v}(l^{\beta}(M))=\log_2 (M_1^{v_{ 1}}M_2^{v_{ 2}}(\log_{\tau} M)^{-\beta}).
\end{align}
Definitions \eqref{eq:228} and \eqref{eq:228'} will be applied with
$v\in\ZZ^2$ being a vertex of the backwards Newton diagram $N_{P}$.
In this section we shall abbreviate $\mathfrak m_{M_1, M_2}^P$ to
\begin{align*}
\mathfrak m_{M_1, M_2}(\xi):=\frac{1}{(1-\tau^{-1})^2}\int_{\tau^{-1}}^1\int_{\tau^{-1}}^1\ex(P_{\xi}(M_1y_1, M_2y_2))dy_1dy_2, \qquad \xi\in\RR.
\end{align*}
We also define the following two partial  multipliers
\begin{align}
\label{eq:326}
\begin{split}
\mathfrak m_{m_1, M_2}^1(\xi):=\frac{1}{1-\tau^{-1}}\int_{\tau^{-1}}^1\ex(P_{\xi}(m_1, M_2y_2))dy_2, \qquad \xi\in\RR, \; m_1\in\ZZ,\\
\mathfrak m_{M_1, m_2}^2(\xi):=\frac{1}{1-\tau^{-1}}\int_{\tau^{-1}}^1\ex(P_{\xi}(M_1y_1, m_2))dy_1, \qquad \xi\in\RR, \; m_2\in\ZZ.
\end{split}
\end{align}

Our main result of this section is Theorem \ref{thm:maint}, which is a restatement of Theorem \ref{thm:main'''}.
\begin{theorem}
\label{thm:maint}
Let $r\in\ZZ_+$ be the number of vertices in the backwards Newton diagram $N_{P}$.  
Then for every $p\in(1, \infty)$ and $j\in[r]$ and  for every $f\in\ell^p(\ZZ)$ one has
\begin{align}
\label{eq:289}
\sup_{J\in\ZZ_+}\sup_{I\in\mathfrak S_J(\SS_{\tau}(j))}\|O_{I, J}(T_{\ZZ}[m_{M_1, M_2}]f: (M_1, M_2)\in\SS_{\tau}(j))\|_{\ell^p(\ZZ)}\lesssim_{p, \tau}\|f\|_{\ell^p(\ZZ)}.
\end{align}
\end{theorem}

The proof of Theorem \ref{thm:maint} is divided into several steps. We  apply iteratively  the classical  circle method,
taking into account the geometry of the backwards Newton diagram $N_P$. 
\subsection{Preliminaries}
The number of vertices $r\in\ZZ_+$ in the backwards Newton diagram $N_{P}$  is fixed. Let $v_j=(v_{j, 1}, v_{j, 2})$ denote the vertex of $N_{P}$ corresponding to  $j\in[r]$.

It suffices to establish inequality \eqref{eq:289} for
 $j=r=1$ assuming additionally that $\log M_1\le \log M_2$
when $(M_1, M_2)\in\SS_{\tau}(1)$, or for any $r\ge2$ and 
any $1\le j<r$. Both cases  
ensure that
\begin{align}
\label{eq:29}
\log M_1\lesssim \log M_2\quad\text{ whenever }\quad
(M_1, M_2)\in\SS_{\tau}(j),
\end{align}
which means that $M_1\le M_2^{K_j}$ for some $K_j>0$, see Remark \ref{rem:2-0}.
 The case when $j=r$ with $r\ge2$ can be proved in much the same way, with the difference that $\log M_1\gtrsim \log M_2$ whenever $(M_1, M_2)\in\SS_{\tau}(r)$. We only outline the most important changes, omitting the details, which can be easily adjusted using the arguments below.

 From now on $p\in(1, \infty)$ is fixed and we let $p_0\in(1, 2)$ be such that $p\in(p_0, p_0')$.   The proof will involve several parameters that have to be suitably adjusted to $p\in(p_0,p_0')$.

 We begin by setting
\begin{align*}
\theta_p:=\bigg(\frac{1}{p_0}-\frac{1}{\min\{p, p'\}}\bigg)\bigg(\frac{1}{p_0}-\frac{1}{2}\bigg)^{-1}\in(0, 1).
\end{align*}
We will take
\begin{align}
\label{eq:299}
\alpha>100\:\theta_p^{-1},
\qquad \text{ and } \qquad
\beta> 1000\max\big\{\delta^{-1}, (1+\deg P)^5\big\}(\alpha+1),
\end{align}
where $\beta\in\ZZ_+$ plays the role of the parameter $\beta\in\ZZ_+$ from Proposition \ref{prop:23}, and $\delta\in (0, 1)$ is the parameter that arises in the complete sum estimates, see Proposition  \ref{prop:32}.

Finally, we need the parameter $\rho>0$, introduced in the Ionescu--Wainger multiplier theorem (see Theorem \ref{thm:IW} as well as Theorem \ref{thm:IW1} and Theorem \ref{thm:IW2}), to satisfy
\begin{align}
\label{eq:301}
\rho\beta<\frac{1}{1000}.
\end{align}

\subsection{Minor arc estimates}
We first establish the minor arcs estimates. 
\begin{claim}
\label{claim:1}
For every $1\le j<r$ and for every $(M_1, M_2)\in\SS_{\tau}(j)$ one has
\begin{align}
\label{eq:54}
\|T_{\ZZ}[m_{M_1, M_2}\Delta^{c}_{\le l^{\beta}(M_2), \le -n_{M_1, M_2}^{v_j,\beta}(M_2)}]f\|_{\ell^2(\ZZ)}\lesssim_{ \tau} (\log M_2)^{-\alpha}\|f\|_{\ell^2(\ZZ)}, \qquad f\in\ell^2(\ZZ),
\end{align}
with $\alpha$ as in \eqref{eq:299}.
The same estimate holds when $j=r=1$, as long as $\log M_1\le \log M_2$.
\end{claim}

The case $j=r\ge2$ requires a minor modification. Keeping in mind that $\log M_2\lesssim \log M_1$, it suffices to establish
an analogue of \eqref{eq:54}. Namely, one has 
\begin{align*}
\|T_{\ZZ}[m_{M_1, M_2}\Delta^{ c}_{\le l^{\beta}(M_1), \le -n_{M_1, M_2}^{v_j,\beta}(M_1)}]f\|_{\ell^2(\ZZ)}\lesssim_{\tau} (\log M_1)^{-\alpha}\|f\|_{\ell^2(\ZZ)}, \qquad f\in\ell^2(\ZZ).
\end{align*}

\begin{proof}[Proof of Claim \ref{claim:1}]
Since  $\log M_1\lesssim \log M_2$, one has $\log M_{r, j}^*\simeq \log M_2$, where $M_{r, j}^*$ was defined  in \eqref{eq:212}. We can also assume that $M_2$ is a large number.  To prove \eqref{eq:54}, by Plancherel's theorem, it suffices to show for every $\xi\in\TT$ that
\begin{align}
\label{eq:288}
|m_{M_1, M_2}(\xi)\Delta_{\le l^{\beta}(M_2), \le -n_{M_1, M_2}^{v_j,\beta}(M_2)}^{c}(\xi)|\lesssim (\log M_2)^{-\alpha}.
\end{align}
For this purpose we use Dirichlet's principle to find a rational fraction $a_0/q_0$ such that $(a_0, q_0)=1$ and $1\le q_0 \le C M_1^{v_{j, 1}}M_2^{v_{j, 2}}\log(M_{r, j}^*)^{-\beta}=C 2^{n_{M_1, M_2}^{v_j, \beta}(M_{r, j}^*)}$ and
\begin{align*}
\Big|\xi-\frac{a_0}{q_0}\Big|\le\frac{\log(M_{r, j}^*)^{\beta}}{q_0CM_1^{v_{j, 1}}M_2^{v_{j, 2}}}\le \frac{1}{q_0^2}
\end{align*}
for a large constant $C>1$ to be specified later. If $q_0< \log(M_2)^{\beta}$ then $a_0/q_0\in \Sigma_{\le l^{\beta}(M_2)}$ and consequently the left-hand side of \eqref{eq:288} vanishes (if $C>1$ is large enough)  and there is nothing to prove. Thus we can assume that $ \log(M_{r, j}^*)^{\beta}\lesssim q_0\lesssim M_1^{v_{j, 1}}M_2^{v_{j, 2}}\log(M_{r, j}^*)^{-\beta}$. We now can apply Proposition \ref{prop:23} and obtain \eqref{eq:288} as claimed. 
\end{proof}

\subsection{Major arcs estimates}
Recalling  \eqref{Phi}
we begin with a simple approximation formula. 
\begin{lemma}
\label{lemma:3}
Suppose that  $1\le j<r$ and $(M_1, M_2)\in\SS_{\tau}(j)$. Then for every $0\le l, l'\le l^{\beta}(M_2)$  and $(M_1', M_2')\in\SS_{\tau}(j)$ and $m_1\simeq M_1'$ such that $1\le M_1'\le M_1$ and $2^{C_{\rho}2^{\rho l}}\le M_2'\le M_2$ one has 
\begin{align}
\label{eq:129}
\begin{split}
m_{m_1, M_2'}^1(\xi)\Delta_{\le l, \le -n_{M_1, M_2}^{v_j}(l')}(\xi)
=\Phi_{\le -n_{M_1, M_2}^{v_j}(l')}^{\Sigma_{\le l}}[G_{m_1}^1, \mathfrak m_{m_1, M_2'}^1](\xi)+O(2^{C_{\rho}2^{\rho l}}(M_{2}')^{-1}),
\end{split}
\end{align}
where  $n_{M_1, M_2}^{v}(N)$, $G_{m_1}^1$,  $m_{m_1, M_2}^1$ and $\mathfrak m_{m_1, M_2}^1$ were defined respectively in \eqref{eq:228'}, \eqref{eq:325},  \eqref{eq:128}  and \eqref{eq:326}.
In particular, \eqref{eq:129} immediately yields
\begin{align}
\label{eq:76}
\begin{split}
m_{M_1', M_2'}(\xi)\Delta_{\le l, \le -n_{M_1, M_2}^{v_j}(l')}(\xi)
=&\sum_{m_1\in\ZZ}\Phi_{\le -n_{M_1, M_2}^{v_j}(l')}^{\Sigma_{\le l}}[G_{m_1}^1, \mathfrak m_{m_1, M_2'}^1](\xi)\chi_{M_1'}(m_1)\\
&+O(2^{C_{\rho}2^{\rho l}}(M_{2}')^{-1}).
\end{split}
\end{align}
The same claims hold when $j=r=1$, as long as $\log M_1\le \log M_2$.
\end{lemma}

A similar conclusion holds when $j=r\ge2$. Taking into account that $\log M_2\lesssim \log M_1$ whenever 
$(M_1, M_2)\in\SS_{\tau}(r)$  and assuming that 
$0\le l, l'\le l^{\beta}(M_1)$,  one has for every $(M_1', M_2')\in\SS_{\tau}(j)$ and $m_2\simeq M_2'$ satisfying $2^{C_{\rho}2^{\rho l}}\le M_1'\le M_1$ and $1\le M_2'\le M_2$ that
\begin{align}
\label{eq:130}
\begin{split}
m_{M_1', m_2}^2(\xi)\Delta_{\le l, \le -n_{M_1, M_2}^{v_j}(l')}(\xi)
=\Phi_{\le -n_{M_1, M_2}^{v_j}(l')}^{\Sigma_{\le l}}[G_{m_2}^2, \mathfrak m_{M_1', m_2}^2](\xi)+O(2^{C_{\rho}2^{\rho l}}(M_{1}')^{-1}).
\end{split}
\end{align}
In particular, \eqref{eq:130} yields
\begin{align*}
m_{M_1', M_2'}(\xi)\Delta_{\le l, \le -n_{M_1, M_2}^{v_j}(l')}(\xi)
=&\sum_{m_2\in\ZZ}\Phi_{\le -n_{M_1, M_2}^{v_j}(l')}^{\Sigma_{\le l}}[G_{m_2}^2, \mathfrak m_{M_1', m_2}^2](\xi)\chi_{M_2'}(m_2)\\
&+O(2^{C_{\rho}2^{\rho l}}(M_{1}')^{-1}).
\end{align*}

\begin{proof}[Proof of Lemma \ref{lemma:3}]
For every $a/q\in\Sigma_{\le l}$,  we note 
\begin{align}
\label{eq:320}
\ex(P_{\xi}(m_1, m_2))=
\ex(P_{\xi-a/q}(m_1, qm + r_2))\ex(P_{a/q}(m_1, r_2)),
\end{align}
whenever $m_1\in\ZZ$, $m_2 = qm + r_2$ and $r_2\in\ZZ_q$.
Then, by \eqref{eq:320}, since $q\le 2^{C_{\rho}2^{\rho l}}\le M_2'$, we have
\begin{align}
\label{eq:321}
&\sum_{m_2\in\ZZ}\ex(P_{\xi}(m_1, m_2))\chi_{M_2'}(m_2)=
\sum_{r_2=1}^q\ex(P_{a/q}(m_1, r_2))\\
\nonumber&\hspace{3cm}\times\sum_{m\in \ZZ}\ex(P_{\xi-a/q}(m_1, qm+r_2))\chi_{M_2'}(qm + r_2).
\end{align}
The summation in $m$ ranges over $m_{*}\le m \le m_{**}$ where $m_{*}/m_{**}$ is minimal/maximal with respect to
$\tau^{-1}M_2' \le qm + r_2 \le M_2'$. We will use Lemma \ref{sum-integral} to compare  
$$
{\rm the \ sum} \ \  \sum_{m_{*} <m \le m_{**}} \ex(f(m)) \ \ \ {\rm to \ the \ integral} \ \ \int_{m_{*}}^{m_{**}} \ex(f(s)) ds,
$$
where $f(m) = P_{\xi-a/q}(m_1, qm+r_2))$.
Suppose that $a/q\in\Sigma_{\le l}$ approximates $\xi$ in the following sense
\begin{align}\label{xi-bound}
\Big|\xi-\frac{a}{q}\Big|\le \frac{(\log_{\tau} M_2)^{\beta}}{M_1^{v_{j, 1}}M_2^{v_{j, 2}}}.
\end{align}
 From the definition of $\Sigma_{\le l}$, we see that $q  \le 2^{C_{\rho}2^{\rho l}}$.
Therefore by \eqref{xi-bound},  the derivative $f'$ satisfies
$$
|f'(m)| \lesssim  q |\xi - a/q| (M_1')^{v_{j, 1}} (M_2')^{v_{j,2} - 1} \lesssim
q  (\log_{\tau} M_2)^{\beta}M_2^{-1} < \ 1/2
$$
since $v_{j,2}\ge 1$, $\log_2 q \lesssim 2^{\rho l^{\beta}(M_2)} \le (\log_{\tau}M_2)^{\rho \beta}$
and $\rho \beta \le 1/10$ by \eqref{eq:301}.
By Lemma \ref{sum-integral}, we have
$$
\Big| q \sum_{m_{*}\le m\le m_{**}}\ex(P_{\xi-a/q}(m_1, qm+r_2))
- \int_{\tau^{-1}M_2'}^{M_2'} \ex(P_{\xi-a/q}(m_1, t)) dt\Big| \lesssim q
$$
and hence by \eqref{eq:321}, we obtain
$$
\Big|\sum_{m_2\in\ZZ}\ex(P_{\xi}(m_1, m_2))\chi_{M_2'}(m_2) -  G_{m_1}^1(a/q) \mathfrak m_{m_1, M_2'}^1(\xi-a/q)\Big|\lesssim q (M_2')^{-1},
$$
which by $q  \le 2^{C_{\rho}2^{\rho l}}$ proves  \eqref{eq:129} as desired. 
\end{proof}

For $i\in[2]$ and $j\in[r]$ let $M_1^c = M_2, M_2^c = M_1$,
\begin{align*}
\SS_{\tau}^i(j):=\{M_i\in\DD_{\tau}: (M_1, M_2)\in \SS_{\tau}(j) \, {\rm for \ some} \, M_i^c\}  
\end{align*}
and for $M_1, M_2\in\DD_{\tau}$ we also let
\begin{align*}
\SS_{\tau}^1(j; M_2):=&\{M_1\in\DD_{\tau}: (M_1, M_2)\in \SS_{\tau}(j)\},\\
\SS_{\tau}^2(j; M_1):=&\{M_2\in\DD_{\tau}: (M_1, M_2)\in \SS_{\tau}(j)\}.
\end{align*}

\vskip 5pt
\subsection{Changing  scale estimates}
In our next step we will have to change the scale (or more precisely, we will truncate the size of denominators of fractions in $\Sigma_{\le l^{\beta}(M_2)}$) to make the approximation estimates with respect to the first variable possible.

We formulate the change of scale argument
as follows.

\begin{claim}
\label{claim:4}
For every $1\le j<r$  and for every $M_1\in\SS_{\tau}^1(j)$ one has
\begin{align}
\label{eq:315}
\|\sup_{M_2\in\SS_{\tau}^2(j; M_1)}|T_{\ZZ}[g_{M_1, M_2}^{M_2}-h_{M_1, M_2}^{M_1}]f|\|_{\ell^2(\ZZ)}
\lesssim_{\tau} (\log M_1)^{-\alpha}\|f\|_{\ell^2(\ZZ)}, \qquad f\in\ell^2(\ZZ),
\end{align}
with $\alpha$ as in \eqref{eq:299}, where  
\begin{align}
\label{eq:342}
\begin{gathered}
g_{M_1, M_2}^N:=m_{M_1, M_2}\Delta_{\le l^{\beta}(N), \le -n_{M_1, M_2}^{v_j,\beta}(N)}, \qquad N\ge1,\\
h_{M_1, M_2}^{N}(\xi):=\sum_{m_1\in\ZZ}\Phi_{\le -n_{M_1, M_2}^{v_j, \beta}(N)}^{\Sigma_{\le l^{\beta}(N)}}[G_{m_1}^1, \mathfrak m_{m_1, M_2}^1](\xi)\chi_{M_1}(m_1) , \qquad N\ge1.
\end{gathered}
\end{align}
The same estimate holds when $j=r=1$, as long as $\log M_1\le \log M_2$.
\end{claim}
The case $j=r\ge2$ requires a minor modification. Keeping in mind that $\log M_2\lesssim \log M_1$,  it suffices to establish
an analogue of \eqref{eq:315}. Namely, for every $M_2\in\SS_{\tau}^2(j)$ one has
\begin{align}
\label{eq:316}
\|\sup_{M_1\in\SS_{\tau}^1(j; M_2)}|T_{\ZZ}[g_{M_1, M_2}^{M_1}-h_{M_1, M_2}^{M_2}]f|\|_{\ell^2(\ZZ)}
\lesssim_{\tau} (\log M_2)^{-\alpha}\|f\|_{\ell^2(\ZZ)}.
\end{align}
We only
present the proof of \eqref{eq:315}, inequality \eqref{eq:316} can be
proved in a similar way.

\begin{proof}[Proof of Claim \ref{claim:4}]
The proof will proceed in several steps.
\vskip 5pt
\paragraph{\bf Step 1}
Using \eqref{eq:76} from Lemma \ref{lemma:3} we have  
\begin{align}
\label{eq:333}
\|T_{\ZZ}[g_{M_1, M_2}^{M_2}-h_{M_1, M_2}^{M_2}]f\|_{\ell^2(\ZZ)}
\lesssim_{\tau} M_2^{-1/2}\|f\|_{\ell^2(\ZZ)}.
\end{align}
Hence, by \eqref{eq:333} it suffices to prove (with $\alpha$ as in \eqref{eq:299})  that for every $M_1\in\SS_{\tau}^1(j)$ one has
\begin{align}
\label{eq:317}
\|\sup_{M_2\in\SS_{\tau}^2(j; M_1)}|T_{\ZZ}[h_{M_1, M_2}^{M_2}-h_{M_1, M_2}^{M_1}]f|\|_{\ell^2(\ZZ)}
\lesssim_{\tau} (\log M_1)^{-\alpha}\|f\|_{\ell^2(\ZZ)}, \qquad f\in\ell^2(\ZZ).
\end{align}
 
\paragraph{\bf Step 2}
To prove \eqref{eq:317} we define for any $s\in\NN$ a new multiplier by
\begin{align}
\label{eq:68}
h_{M_1, M_2, s}^{N}(\xi):=\sum_{m_1\in\ZZ}\Phi_{\le -n_{M_1, M_2}^{v_j, \beta}(N)}^{\Sigma_{s}}[G_{m_1}^1, \mathfrak m_{m_1, M_2}^1](\xi)\chi_{M_1}(m_1).
\end{align}
In view of \eqref{eq:29} we may assume that $l^{\beta}(M_1)<  l^{\beta}(M_2)$, then
one can  write
\begin{align*}
h_{M_1, M_2}^{M_2}(\xi)-h_{M_1, M_2}^{M_1}(\xi)=&\sum_{0\le s\le l^{\beta}(M_1)}\big(h_{M_1, M_2, s}^{M_2}(\xi)-h_{M_1, M_2, s}^{M_1}(\xi)\big)\\
&+\sum_{l^{\beta}(M_1)< s\le l^{\beta}(M_2)}h_{M_1, M_2, s}^{M_2}(\xi).
\end{align*}
For sufficiently large $s\in\NN$ if $l^{\beta}(M_1)\ge s$, then by \eqref{eq:29} we have $\log_{\tau}M_2\ge K_j^{-1}2^{s/\beta}\ge 2^{s/(2\beta)}$. Similarly,  if $l^{\beta}(M_2)\ge s$, then $\log_{\tau}M_2\ge  2^{s/(2\beta)}$.
Thus, we set $N_s:=\tau^{2^{s/(2\beta)}}$ for any $s\in\NN$, and 
let 
\begin{align*}
\tilde{\SS}_{\tau, M_1}^2(j; s):=\{M_2\in \SS_{\tau}^2(j; M_1): M_2\ge N_s\}.
\end{align*}
The proof will be finished if we can show (with $\alpha$ and $\delta$ as in \eqref{eq:299}) that for every $f\in\ell^2(\ZZ)$ and for every $M_1\in\SS_{\tau}^1(j)$, and   $0\le s\le l^{\beta}(M_1)$ one has
\begin{align}
\label{eq:328}
\|\sup_{M_2\in\tilde{\SS}_{\tau, M_1}^2(j; s)}|T_{\ZZ}[h_{M_1, M_2, s}^{M_2}-h_{M_1, M_2, s}^{M_1}]f|\|_{\ell^2(\ZZ)}
\lesssim_{\tau} 2^{-\delta s}(\log M_1)^{-\alpha}\|f\|_{\ell^2(\ZZ)},
\end{align}
and moreover, for every $s\in\NN$,  one also has
\begin{align}
\label{eq:334}
\|\sup_{M_2\in\tilde{\SS}_{\tau, M_1}^2(j; s)}|T_{\ZZ}[h_{M_1, M_2, s}^{M_2}]f|\|_{\ell^2(\ZZ)}
\lesssim_{\tau} s2^{-\delta s}\|f\|_{\ell^2(\ZZ)}.
\end{align}
Then summing \eqref{eq:334} over $s\ge l^{\beta}(M_1)$ we obtain the desired claim by \eqref{eq:299}.

\paragraph{\bf Step 3}
We now establish \eqref{eq:328}.   If $N\in\{M_1, M_2\}$ and $(M_1, M_2)\in\SS_{\tau}(j)$, then for $M_2\ge N_s$ we note that
\begin{align}
\label{eq:56}
\begin{split}
\eta_{\le -n_{M_1, M_2}^{v_j, \beta}(N)}(\xi)&=\eta_{\le -n_{M_1, M_2}^{v_j, \beta}(N)}(\xi)\eta_{\le -n_{M_1, M_2}^{v_j, \beta}(N)+1}(\xi)\\
&=\eta_{\le -n_{M_1, M_2}^{v_j, \beta}(N)}(\xi)\eta_{\le -n_{M_1, N_s}^{v_j, \beta}(N_s)+1}(\xi),
\end{split}
\end{align}
since $1\le j<r$ and  $v_{j, 2}\neq0$. Using \eqref{eq:56} we may write
\begin{align}
\label{eq:335}
\begin{split}
&\Phi_{\le -n_{M_1, M_2}^{v_j, \beta}(N)}^{\Sigma_{s}}[G_{m_1}^1, \mathfrak m_{m_1, M_2}^1](\xi)\\
&\hspace{2cm}=\Phi_{\le -n_{M_1, N_s}^{v_j, \beta}(N_s)+2}^{\Sigma_{\le s}}[1, \mathfrak m_{m_1, M_2}^1\eta_{\le -n_{M_1, M_2}^{v_j, \beta}(N)}](\xi)
\times\Phi_{\le -n_{M_1, N_s}^{v_j, \beta}(N_s)+1}^{\Sigma_{s}}[G_{m_1}^1, 1](\xi),
\end{split}
\end{align}
for sufficiently large $s$ such that $0\le s\le l^{\beta}(M_1)$, which in turn guarantees that  $M_2> N_{s}$ as we have seen in the previous step. Denote
\begin{gather*}
I(m_1, M_2):=
T_{\ZZ}\Big[\Phi_{\le -n_{M_1, N_s}^{v_j, \beta}(N_s)+2}^{\Sigma_{\le s}}\big[1, \sum_{N=l^{\beta}(M_1)}^{l^{\beta}(M_2)-1}\mathfrak m_{m_1, M_2}^1{\rm D}_N\big(\eta_{\le -n_{M_1, M_2}^{v_j}(N)}\big)\big]\Big],
\end{gather*}
where, see definitions \eqref{eq:228} and \eqref{eq:228'},
\begin{align*}
{\rm D}_N\big(\eta_{\le -n_{M_1, M_2}^{v_j}(N)}\big):=\eta_{\le -n_{M_1, M_2}^{v_j}(N+1)}-\eta_{\le -n_{M_1, M_2}^{v_j}(N)}.
\end{align*}
Using the factorization from \eqref{eq:335} one sees
\begin{align*}
&\|\sup_{M_2\in\tilde{\SS}_{\tau, M_1}^2(j; s)}|T_{\ZZ}[h_{M_1, M_2, s}^{M_2}-h_{M_1, M_2, s}^{M_1}]f|\|_{\ell^2(\ZZ)}\\
&\hspace{2cm}\le\sum_{m_1\in\ZZ}\|I(m_1, M_2)\|_{\ell^2(\ZZ)\to\ell^2(\ZZ; \ell^{\infty}_{M_2}(\tilde{\SS}_{\tau, M_1}^2(j; s)))}\chi_{M_1}(m_1)\\
&\hspace{4cm}\times
\Big\|
T_{\ZZ}\big[\Phi_{ \le -n_{M_1, N_s}^{v_j, \beta}(N_s)+1}^{\Sigma_{s}}[G_{m_1}^1, 1]\big]f\Big\|_{\ell^2(\ZZ)}.
\end{align*}
Using the Ionescu--Wainger multiplier theory (see Theorem \ref{thm:IW})  we conclude that
\begin{align*}
\sup_{m_1\in(\tau^{-1}M_1, M_1]\cap\ZZ}\|I(m_1, M_2)\|_{\ell^2(\ZZ)\to\ell^2(\ZZ; \ell^{\infty}_{M_2}(\tilde{\SS}_{\tau, M_1}^2(j; s)))}
\lesssim_{\tau} (\log M_1)^{-\alpha}
\end{align*}
with $\alpha$ as in \eqref{eq:299},
since using standard square function continuous arguments we have
\begin{align*}
\begin{gathered}
\bigg\|\bigg(\sum_{M_2\in \tilde{\SS}_{\tau, M_1}^2(j; s)}
\Big|T_{\RR}\Big[\sum_{N=l^{\beta}(M_1)}^{l^{\beta}(M_2)-1}\mathfrak m_{m_1, M_2}^1{\rm D}_N\big(\eta_{\le -n_{M_1, M_2}^{v_j}(N)}\big)\Big]f\Big|^2\bigg)^{1/2}
\bigg\|_{L^2(\RR)}
\lesssim (\log M_1)^{-\alpha}\|f\|_{L^2(\RR)}.
\end{gathered}
\end{align*}
Thus by the Cauchy--Schwarz inequality, Plancherel's theorem and inequality \eqref{eq:105}  we obtain
\begin{align*}
&\|\sup_{M_2\in\tilde{\SS}_{\tau, M_1}^2(j; s)}|T_{\ZZ}[h_{M_1, M_2, s}^{M_2}-h_{M_1, M_2, s}^{M_1}]f|\|_{\ell^2(\ZZ)}\\
&\hspace{1cm}\lesssim (\log M_1)^{-\alpha}
\Big\|\Big(\sum_{m_1\in\ZZ}\chi_{M_1}(m_1)
\big|T_{\ZZ}\big[\Phi_{\le -n_{M_1, N_s}^{v_j, \beta}(N_s)+1}^{\Sigma_{s}}[G_{m_1}^1, 1]\big]f\big|^2\Big)^{1/2}\Big\|_{\ell^2(\ZZ)}\\
&\hspace{3cm}\lesssim_{\tau} 2^{-\delta s}(\log M_1)^{-\alpha}\|f\|_{\ell^2(\ZZ)}
\end{align*}
with $\alpha$ and $\delta$ as in \eqref{eq:299}, which  yields \eqref{eq:328}. 
\paragraph{\bf Step 4} We now establish \eqref{eq:334}. Using notation from the previous step 
and denoting
\begin{gather*}
J(m_1, M_2):=
T_{\ZZ}\Big[\Phi_{\le -n_{M_1, N_s}^{v_j, \beta}(N_s)+2}^{\Sigma_{\le s}}\big[1, \mathfrak m_{m_1, M_2}^1\eta_{\le n_{M_1, M_2}^{v_j, \beta}(M_2)}\big)\big]\Big],
\end{gather*}
and again using the factorization from \eqref{eq:335} one sees
\begin{align*}
&\|\sup_{M_2\in\tilde{\SS}_{\tau, M_1}^2(j; s)}|T_{\ZZ}[h_{M_1, M_2, s}^{ M_2}]f|\|_{\ell^2(\ZZ)}\\
&\hspace{2cm}\le\sum_{m_1\in\ZZ}\|J(m_1, M_2)\|_{\ell^2(\ZZ)\to\ell^2(\ZZ; \ell^{\infty}_{M_2}(\tilde{\SS}_{\tau, M_1}^2(j; s)))}\chi_{M_1}(m_1)\\
&\hspace{4cm}\times
\Big\|
T_{\ZZ}\big[\Phi_{\le -n_{M_1, N_s}^{v_j, \beta}(N_s)+1}^{\Sigma_{s}}[G_{m_1}^1, 1]\big]f\Big\|_{\ell^2(\ZZ)}.
\end{align*}
Using the Ionescu--Wainger multiplier theory (see Theorem \ref{thm:IW1})  we conclude that
\begin{align*}
\sup_{m_1\in(\tau^{-1}M_1, M_1]\cap\ZZ}\|J(m_1, M_2)\|_{\ell^2(\ZZ)\to\ell^2(\ZZ; \ell^{\infty}_{M_2}(\tilde{\SS}_{\tau, M_1}^2(j; s)))}\lesssim_{\tau} s.
\end{align*}
Then proceeding as in the previous step we obtain  \eqref{eq:334}. This completes the proof of Claim \ref{claim:4}.
\end{proof}

\subsection{Transition estimates}
Our aim will be to understand the final approximation, which will allow us to apply the oscillation Ionescu--Wainger theory (see Theorem \ref{thm:IW2}) from Section \ref{section:6}.
\begin{claim}
\label{claim:7}
For every $1\le j<r$  and for every $M_1\in\SS_{\tau}^1(j)$ one has
\begin{align}
\label{eq:360}
\big\|\sup_{M_2\in\SS_{\tau}^2(j;M_1)}\big|T_{\ZZ}\big[h_{M_1, M_2}^{M_1}-\tilde{h}_{M_1, M_2}^{M_1}\big]f\big|\big\|_{\ell^2(\ZZ)}
\lesssim_{\tau} (\log M_1)^{-\alpha}\|f\|_{\ell^2(\ZZ)}, \qquad f\in\ell^2(\ZZ),
\end{align}
with $\alpha$ as in \eqref{eq:299}, where  $h_{M_1, M_2}^{N}$ was defined in \eqref{eq:342} and
\begin{align}
\label{eq:73}
\tilde{h}_{M_1, M_2}^{N}:=\Phi_{\le -n_{M_1, M_2}^{v_j, \beta}(N)}^{\Sigma_{\le l^{\beta}(N)}}[G, \mathfrak m_{M_1, M_2}], \qquad N\ge1.
\end{align}
The same estimate holds when $j=r=1$, as long as $\log M_1\le \log M_2$.
\end{claim}
The case $j=r\ge2$ requires a minor modification. Keeping in mind that $\log M_2\lesssim \log M_1$,   it suffices to establish
an analogue of \eqref{eq:360}. Namely, one has 
\begin{align}
\label{eq:359}
\big\|\sup_{M_1\in\SS_{\tau}^1(j;M_2)}\big|T_{\ZZ}\big[h_{M_1, M_2}^{M_2}-\tilde{h}_{M_1, M_2}^{M_2}\big]f\big|\big\|_{\ell^2(\ZZ)}
\lesssim_{\tau} (\log M_2)^{-\alpha}\|f\|_{\ell^2(\ZZ)}, \qquad f\in\ell^2(\ZZ),
\end{align}  We only
present the proof of \eqref{eq:360}, inequality \eqref{eq:359} can be
proved in a similar way.

\begin{proof}[Proof of Claim \ref{claim:7}]
The proof will proceed in several steps as before. 
Write
\begin{align*}
h_{M_1, M_2}^{M_1}-\tilde{h}_{M_1, M_2}^{M_1}=\sum_{0\le s\le l^{\beta}(M_1)}h_{M_1, M_2, s}^{M_1}-\tilde{h}_{M_1, M_2, s}^{M_1},
\end{align*}
where $h_{M_1, M_2, s}^{M_1}$ was defined in \eqref{eq:68} and 
\begin{align}
\label{eq:74}
\tilde{h}_{M_1, M_2, s}^{M_1}:=\Phi_{\le -n_{M_1, M_2}^{v_j, \beta}(M_1)}^{\Sigma_{s}}[G, \mathfrak m_{M_1, M_2}].
\end{align}
Then it suffices to show that for sufficiently large $s$ such that $0\le s\le l^{\beta}(M_1)$ we have
\begin{align}
\label{eq:28}
\big\|\sup_{M_2\in\SS_{\tau}^2(j;M_1)}\big|T_{\ZZ}\big[h_{M_1, M_2,s}^{M_1}-\tilde{h}_{M_1, M_2,s}^{M_1}\big]f\big|\big\|_{\ell^2(\ZZ)}
\lesssim_{\tau} s2^{-\delta s}(\log M_1)^{-\alpha}\|f\|_{\ell^2(\ZZ)}, \qquad f\in \ell^2(\ZZ),
\end{align}
with $\alpha$ and $\delta$ as in \eqref{eq:299}, which will clearly imply \eqref{eq:360}.

\paragraph{\bf Step 1}   Using \eqref{eq:56}, in a similar way as in \eqref{eq:335}, we may write   
\begin{align}
\label{eq:71}
\tilde{h}_{M_1, M_2,s}^{M_1}(\xi)=\Phi_{\le -n_{M_1, N_s}^{v_j, \beta}(N_s)+2}^{\Sigma_{\le s}}[1, \widetilde{\mathfrak m}_{M_1, M_2}](\xi)\times
\Phi_{ \le -n_{M_1, N_s}^{v_j, \beta}(N_s)+1}^{\Sigma_{s}}[G, 1](\xi),
\end{align}
where
\begin{align}
\label{eq:14}
\widetilde{\mathfrak m}_{M_1, M_2}(\xi):=\mathfrak m_{M_1, M_2}(\xi)\eta_{\le -n_{M_1, M_2}^{v_j, \beta}(M_1)}(\xi).
\end{align}
By Theorem \ref{thm:IW1}, we may conclude
\begin{align}
\label{eq:139}
\begin{gathered}
\Big\|T_{\ZZ}\Big[\Phi_{\le -n_{M_1, N_s}^{v_j, \beta}(N_s)+2}^{\Sigma_{\le s}}[1, \widetilde{\mathfrak m}_{M_1, M_2}]\Big]\Big\|_{\ell^2(\ZZ)\to \ell^2(\ZZ); \ell^{\infty}_{M_2}(\SS_{\tau}^2(j;M_1))}
\lesssim_{\tau} s,
\end{gathered}
\end{align}
By Plancherel's theorem and inequality \eqref{eq:104} we obtain
\begin{align}
\label{eq:140}
\Big\|T_{\ZZ}\Big[\Phi_{ \le -n_{M_1, N_s}^{v_j, \beta}(N_s)+1}^{\Sigma_{s}}[G, 1]\Big]f\Big\|_{\ell^2(\ZZ)}
\lesssim_{\tau} 2^{-\delta s}\|f\|_{\ell^2(\ZZ)}.
\end{align}
Inequalities \eqref{eq:139} and \eqref{eq:140} and \eqref{eq:71} imply
\begin{align}
\label{eq:28'}
\big\|\sup_{M_2\in\SS_{\tau}^2(j;M_1)}\big|T_{\ZZ}\big[\tilde{h}_{M_1, M_2,s}^{M_1}\big]f\big|\big\|_{\ell^2(\ZZ)}
\lesssim_{\tau} s2^{-\delta s}\|f\|_{\ell^2(\ZZ)}, \qquad f\in \ell^2(\ZZ).
\end{align}

\paragraph{\bf Step 2}
We now establish \eqref{eq:28}. 
For $0\le s\le l^{\beta}(M_1)$, we note that 
\begin{align*}
&h_{M_1, M_2, s}^{M_1}(\xi)=\sum_{a/q\in \Sigma_{s}}\eta_{\le -n_{M_1, M_2}^{v_j, \beta}(M_1)}(\xi-a/q)\sum_{r_1=1}^qG_{r_1}^1(a/q)\\
&\hspace{4cm} \times \sum_{m_1\in\ZZ}
\mathfrak m_{qm_1+r_1, M_2}^1(\xi-a/q)
\chi_{M_1}(qm_1+r_1).
\end{align*}
Introducing $\theta:=\xi-a/q$, 
$U_1:=\frac{\tau^{-1}M_1-r_1}{q}$ and $V_1:=\frac{M_1-r_1}{q}$ one can expand
\begin{align*}
\mathfrak m_{qm_1+r_1, M_2}^1(\theta)=\frac{1}{1-\tau^{-1}}
\int_{\tau^{-1}}^1\ex(P_{\theta}(qm_1+r_1, M_2y_2))dy_2,
\end{align*}
and by the fundamental theorem of calculus, one can write
\begin{align*}
&\sum_{\lfloor U_1\rfloor<m_1\le \lfloor V_1\rfloor}\int_{\tau^{-1}}^1\ex(P_{\theta}(qm_1+r_1, M_2y_2))dy_2-\int_{U_1}^{V_1}\int_{\tau^{-1}}^1\ex(P_{\theta}(qy_1+r_1, M_2y_2))dy_2dy_1\\
&\qquad=\sum_{\lfloor U_1\rfloor<m_1\le \lfloor V_1\rfloor}\int_{m_1-1}^{m_1}\int_{y_1}^{m_1}\int_{\tau^{-1}}^12\pi i q\theta(\partial_1P)(qt+r_1, M_2 y_2)\ex(P_{\theta}(qt+r_1, M_2y_2))dy_2dtdy_1\\
&\qquad\qquad+\Big(\int_{\lfloor U_1\rfloor}^{U_1}-\int_{\lfloor V_1\rfloor}^{V_1}\Big)\int_{\tau^{-1}}^1\ex(P_{\theta}(qy_1+r_1, M_2y_2))dy_2dy_1.
\end{align*}
By the change of variable we have
\begin{align*}
\int_{U_1}^{V_1}\int_{\tau^{-1}}^1\ex(P_{\theta}(qy_1+r_1, M_2y_2))dy_2dy_1=\frac{M_1(1-\tau^{-1})^2}{q}\mathfrak m_{M_1, M_2}(\theta).
\end{align*}
We now define new multipliers 
\begin{gather*}
\mathfrak g_{M_1, M_2}^{r_1, 1}(\theta):=\\
\sum_{\lfloor U_1\rfloor<m_1\le \lfloor V_1\rfloor}\int_{m_1-1}^{m_1}\int_{y_1}^{m_1}\int_{\tau^{-1}}^1\frac{2\pi i q(\log M_1)^{\beta}(\partial_1P)(qt+r_1, M_2y_2)\ex(P_{\theta}(qt+r_1, M_2y_2))}{(1-\tau^{-1})M_1^{v_{j, 1}}M_2^{v_{j,2}}|(\tau^{-1}M_1, M_1]\cap\ZZ|}dy_2dtdy_1,
\end{gather*}
and finally
\begin{align*}
\mathfrak g_{M_1, M_2}^{r_1,2}(\theta):=
\Big(\int_{\lfloor U_1\rfloor}^{U_1}-\int_{\lfloor V_1\rfloor}^{V_1}\Big)\int_{\tau^{-1}}^1\frac{\ex(P_{\theta}(qy_1+r_1, M_2y_2))}{(1-\tau^{-1})|(\tau^{-1}M_1, M_1]\cap\ZZ|}dy_2dy_1.
\end{align*}
Then with these definitions we can write
\begin{align*}
h_{M_1, M_2, s}^{M_1}(\xi)-\tilde{h}_{M_1, M_2,s}^{M_1}(\xi)=
\gamma_{\tau, M_1}\tilde{h}_{M_1, M_2,s}^{M_1}(\xi)
+\sum_{\ell\in[2]}\sum_{a/q\in  \Sigma_{s}}\sum_{r_1=1}^q
G_{r_1}^1(a/q) \mathfrak h_{M_1, M_2}^{r_1, \ell}(\xi-a/q),
\end{align*}
where $\gamma_{\tau, M_1}:=\frac{\{M_1\}-\{\tau^{-1}M_1\}}{|(\tau^{-1}M_1, M_1]\cap\ZZ|}$ and
\begin{align*}
\mathfrak h_{M_1, M_2}^{r_1, 1}(\theta):=
\mathfrak g_{M_1, M_2}^{r_1, 1}(\theta)\varrho_{\le -n_{M_1, M_2}^{v_{j}, \beta}(M_1)}(\theta)
\quad\text{ and } \quad
\mathfrak h_{M_1, M_2}^{r_1, 2}(\theta):=
\mathfrak g_{M_1, M_2}^{r_1, 2}(\theta)
\eta_{\le -n_{M_1, M_2}^{v_{j}, \beta}(M_1)}(\theta)
\end{align*}
and $\varrho_{\le n}(\theta):=(2^{-n}\theta)\eta_{\le n}(\theta)$.
For $\ell\in[2]$, we have 
\begin{align*}
|\mathfrak h_{M_1, M_2}^{r_1, \ell}(\theta)|\lesssim q(\log M_1)^{\beta}M_1^{-1}
\qquad \text{ and } \qquad |\gamma_{\tau, M_1}|\lesssim_{\tau} M_1^{-1}.
\end{align*}
Finally, using Theorem \ref{thm:IW1} for each $\ell\in[2]$ we conclude
\begin{align*}
\Big\|\sup_{M_2\in\SS_{\tau}^2(j;M_1)}\Big|T_{\ZZ}\Big[\sum_{\ell\in[2]}\sum_{a/q\in  \Sigma_{s}}\sum_{r_1=1}^q
G_{r_1}^1(a/q) \mathfrak h_{M_1, M_2}^{r_1, \ell}(\cdot-a/q)\Big]f\Big|\Big\|_{\ell^2(\ZZ)}\lesssim 2^{-\delta s} M_1^{-3/4}\|f\|_{\ell^2(\ZZ)}.
\end{align*}
This in turn, combined with \eqref{eq:28'}, implies \eqref{eq:28} and the proof of Claim \ref{claim:7} is established.
\end{proof}

\subsection{All together: proof of Theorem \ref{thm:maint}} We begin with a useful auxiliary lemma.

\begin{lemma}
\label{lemma:4}
For every $p\in(1, \infty)$ and every $j\in[r]$ there exists a  constant  $\delta_p\in(0, 1)$ such that for every $f\in\ell^p(\ZZ)$ and  $s\in\NN$  one has
\begin{align}
\label{eq:118}
\big\|T_{\ZZ}\big[\Phi_{ \le -n_{N_s, N_s}^{v_j, \beta}(N_s)+1}^{\Sigma_{s}}[G, \Pi_s^{\beta}]\big]f\big\|_{\ell^p(\ZZ)}
\lesssim_{p, \tau} 2^{-\delta_p s}\|f\|_{\ell^p(\ZZ)},
\end{align}
where $N_s:=\tau^{2^{s/(2\beta)}}$ for any $s\in\NN$, and $\Pi_s^{\beta}(\xi):=\prod_{u\in S_P}\eta_{\le -n_{N_s, N_s}^{u, \beta}(N_s)+1}(\xi)$ with $\beta>0$  from \eqref{eq:299}.
\end{lemma}

\begin{proof}
We may assume that $s\ge0$ is large, otherwise there is nothing to prove.
Inequality \eqref{eq:118} for $p=2$ with $\delta_2=\delta$ as in Proposition \ref{prop:32} follows by Plancherel's theorem from inequality \eqref{eq:104}  and the disjointness of supports of $\Pi_s^{\beta}(\xi-a/q)$ whenever $a/q\in \Sigma_{s}$.

We now prove
\eqref{eq:118} for $p\neq2$. We shall proceed in four steps.

\paragraph{\bf Step 1}
Let $M\simeq 2^{10C_{\rho}2^{10\rho s}}$ define
\begin{align*}
{\mathfrak h}_{M}^{s}:=m_{M,M}\Phi_{ \le -n_{N_s, N_s}^{v_j, \beta}(N_s)+1}^{\Sigma_{s}}[1, \Pi_s^{\beta}].
\end{align*}
By the Ionescu--Wainger multiplier theorem (see Theorem \ref{thm:IW}) one has
\begin{align}
\label{eq:123}
\|T_{\ZZ}[{\mathfrak h}_{M}^{s}]f\|_{\ell^u(\ZZ)}
\lesssim_{u, \tau} \|f\|_{\ell^u(\ZZ)},
\end{align}
whenever $u\in\{p_0, p_0'\}$. We will prove 
\begin{align}
\label{eq:126}
\big\|
T_{\ZZ}\big[{\mathfrak h}_{M}^{s}-\Phi_{ \le -n_{N_s, N_s}^{v_j, \beta}(N_s)+1}^{\Sigma_{s}}[G, \mathfrak m_{M, M}\Pi_s^{\beta}]\big]f\big\|_{\ell^p(\ZZ)}
\lesssim_{p, \tau}\|f\|_{\ell^p(\ZZ)},
\end{align}
and
\begin{align}
\label{eq:127}
\big\|
T_{\ZZ}\big[\Phi_{ \le -n_{N_s, N_s}^{v_j, \beta}(N_s)+1}^{\Sigma_{s}}[G, (1-\mathfrak m_{M, M})\Pi_s^{\beta}]\big]f\big\|_{\ell^p(\ZZ)}
\lesssim_{p, \tau}\|f\|_{\ell^p(\ZZ)}.
\end{align}
Assuming momentarily that \eqref{eq:126} and \eqref{eq:127} hold, then \eqref{eq:123} and the triangle inequality yield
\begin{align}
\label{eq:133}
\big\|T_{\ZZ}\big[\Phi_{ \le -n_{N_s, N_s}^{v_j, \beta}(N_s)+1}^{\Sigma_{s}}[G, \Pi_s^{\beta}]\big]f\big\|_{\ell^u(\ZZ)}
\lesssim_{u, \tau}\|f\|_{\ell^u(\ZZ)},
\end{align}
whenever $u\in\{p_0, p_0'\}$. Then interpolation between \eqref{eq:118} for $p=2$ (that we have shown with $\delta_2=\delta$) and \eqref{eq:133} gives \eqref{eq:118} for all $p\in(1, \infty)$.

\paragraph{\bf Step 2}
We now establish \eqref{eq:126}. For $p=2$ it will suffice to show that
\begin{align}
\label{eq:62}
|m_{M,M}(\xi)\Phi_{ \le -n_{N_s, N_s}^{v_j, \beta}(N_s)+1}^{\Sigma_{s}}[1, \Pi_s^{\beta}](\xi)-\Phi_{ \le -n_{N_s, N_s}^{v_j, \beta}(N_s)+1}^{\Sigma_{s}}[G, \mathfrak m_{M, M}\Pi_s^{\beta}](\xi)|\lesssim 2^{-5C_{\rho}2^{5\rho s}}.
\end{align}
Then by \eqref{eq:62} and Plancherel's theorem we obtain for sufficiently large $s\in\NN$ that
\begin{align}
\label{eq:136}
\big\|
T_{\ZZ}\big[{\mathfrak h}_{M}^{s}-\Phi_{ \le -n_{N_s, N_s}^{v_j, \beta}(N_s)+1}^{\Sigma_{s}}[G, \mathfrak m_{M, M}\Pi_s^{\beta}]\big]f\big\|_{\ell^2(\ZZ)}
\lesssim_{\tau} 2^{-5C_{\rho}2^{5\rho s}}\|f\|_{\ell^2(\ZZ)}.
\end{align}
 Moreover, for $u\in\{p_0, p_0'\}$ we have the trivial estimate
\begin{align}
\label{eq:119}
\big\|T_{\ZZ}\big[{\mathfrak h}_{M}^{s}-\Phi_{ \le -n_{N_s, N_s}^{v_j, \beta}(N_s)+1}^{\Sigma_{s}}[G, \mathfrak m_{M, M}\Pi_s^{\beta}]\big]f\big\|_{\ell^u(\ZZ)}
\lesssim_{u, \tau}2^{2C_{\rho}2^{\rho s}}\|f\|_{\ell^u(\ZZ)},
\end{align}
due to \eqref{eq:373}. Interpolating \eqref{eq:136} and \eqref{eq:119} gives \eqref{eq:126}.

\paragraph{\bf Step 3} To prove \eqref{eq:62} we proceed  as in the proof of Lemma \ref{lemma:3} and  show that 
\begin{align}
\label{eq:63}
|m_{M, M}(\xi) -  G(a/q) \mathfrak m_{M, M}(\xi-a/q)|\lesssim q M^{-1},
\end{align}
whenever $a/q\in\Sigma_s$ and $|\xi-{a}/{q}|\le \min_{u\in S_P}\{(\log_{\tau} N_s)^{\beta}N_s^{-u_1}N_s^{-u_2}\}.$ Then \eqref{eq:63} immediately gives \eqref{eq:62}, since $q  \le 2^{C_{\rho}2^{\rho s}}$ if $a/q\in \Sigma_s$. To verify \eqref{eq:63} we use 
Lemma \ref{sum-integral} twice, which can be applied, since the derivatives $\partial_{m_1}f$ and $\partial_{m_2}f$ of $f(m_1, m_2)=P_{\xi-a/q}(qm_1+r_1, qm_2+r_2)$ satisfy
\[
|\partial_{m_{\ell}}f(m_1, m_2)| \lesssim  q |\xi - a/q|\sum_{u\in S_P}M^{u_1+u_2-1} \lesssim
q  (\log_{\tau} N_s)^{\beta}N_s^{-1} < \ 1/2,\qquad \ell\in[2]
\]
for sufficiently large $s\in\NN$, since $M\le N_s^{1/5}$, $q  \le 2^{C_{\rho}2^{\rho s}}$
and $\rho \beta \le 1/10$ by \eqref{eq:301}, and we are done. 

\paragraph{\bf Step 4}
We now establish \eqref{eq:127}. Assume that  $p=2$ and observe that
\begin{align*}
|(1-\mathfrak m_{M, M}(\xi-a/q))\Pi_s^{\beta}(\xi-a/q)|
\lesssim |\xi - a/q|\sum_{u\in S_P}M^{u_1+u_2}
\lesssim N_{s}^{-3/4} \lesssim  2^{-10C_{\rho}2^{5\rho s}}
\end{align*}
for sufficiently large $s\in\NN$, since
$M\simeq 2^{10C_{\rho}2^{10\rho s}}$, and
 $\rho \beta < 1/1000$.  Using
this bound and Plancherel's theorem we see that
\begin{align}
\label{eq:125}
\big\|
T_{\ZZ}\big[\Phi_{ \le -n_{N_s, N_s}^{v_j, \beta}(N_s)+1}^{\Sigma_{s}}[G, (1-\mathfrak m_{M, M})\Pi_s^{\beta}]\big]f\big\|_{\ell^2(\ZZ)}
\lesssim_{\tau}2^{-5C_{\rho}2^{5\rho s}}\|f\|_{\ell^2(\ZZ)}. 
\end{align}
Moreover by \eqref{eq:373}, for $u\in\{p_0, p_0'\}$ we have the trivial estimate
\begin{align}
\label{eq:131}
\big\|
T_{\ZZ}\big[\Phi_{ \le -n_{N_s, N_s}^{v_j, \beta}(N_s)+1}^{\Sigma_{s}}[G, (1-\mathfrak m_{M, M})\Pi_s^{\beta}]\big]f\big\|_{\ell^u(\ZZ)}
\lesssim_{u, \tau}2^{2C_{\rho}2^{\rho s}}\|f\|_{\ell^u(\ZZ)}. 
\end{align}
Interpolation between \eqref{eq:125} and \eqref{eq:131} yields \eqref{eq:127} and the proof of Lemma \ref{lemma:4} is complete. 
\end{proof}

Recalling the definition of $\tilde{h}_{M_1, M_2}^{M_1}$ from \eqref{eq:73}  we now prove the following claim:
\begin{claim}
\label{claim:8}
For every $p\in(1, \infty)$ and every $1\le j<r$ and for every $f\in\ell^p(\ZZ)$ one has
\begin{align}
\label{eq:369}
\sup_{J\in\ZZ_+}\sup_{I\in\mathfrak S_J(\SS_{\tau}(j))}\|O_{I, J}(T_{\ZZ}[\tilde{h}_{M_1, M_2}^{M_1}]f: (M_1, M_2)\in\SS_{\tau}(j))\|_{\ell^p(\ZZ)}\lesssim_{p, \tau}\|f\|_{\ell^p(\ZZ)}.
\end{align}
The same estimate holds when $j=r=1$, as long as $\log M_1\le \log M_2$.
\end{claim}
When $j=r\ge2$, in view of \eqref{eq:359}, we will be able to reduce the problem to the following
\begin{align}
\label{eq:370}
\sup_{J\in\ZZ_+}\sup_{I\in\mathfrak S_J(\SS_{\tau}(j))}\|O_{I, J}(T_{\ZZ}[\tilde{h}_{M_1, M_2}^{M_2}]f: (M_1, M_2)\in\SS_{\tau}(j))\|_{\ell^p(\ZZ)}
\lesssim_{p, \tau} \|f\|_{\ell^p(\ZZ)}.
\end{align}
We will only prove \eqref{eq:369} the proof of \eqref{eq:370} will follow in a similar way. We omit details.
\begin{proof}[Proof of Claim \ref{claim:8}]
The proof will consist of two steps to make the argument clear.
\paragraph{\bf Step 1}
Similarly as in Claim \ref{claim:4}  we define  $N_s:=\tau^{2^{s/(2\beta)}}$ for any $s\in\NN$ and introduce
\begin{align*} 
\tilde{\SS}_{\tau}(j, s):=\{(M_1, M_2)\in \SS_{\tau}(j): M_1\ge N_s\}.
\end{align*}
For each  $(M_1, M_2)\in \SS_{\tau}(j)$ we have $M_1^{v_{j, 1}}M_2^{v_{j, 2}}\ge M_1^{u_{1}}M_2^{u_{2}}$ for every $u=(u_1, u_2)\in S_P$. Hence
\begin{align}
\label{eq:39}
\begin{split}
\eta_{\le -n_{M_1, M_2}^{v_{j}, \beta}(M_1)}(\xi)&=\eta_{\le -n_{M_1, M_2}^{v_{j}, \beta}(M_1)}(\xi)\prod_{u\in S_P}\eta_{\le -n_{M_1, M_2}^{u, \beta}(M_1)+1}(\xi)\\
&=\eta_{\le -n_{M_1, M_2}^{v_{j}, \beta}(M_1)}(\xi)\Pi_s^{\beta}(\xi)
\end{split}
\end{align}
holds for sufficiently large $s\in\NN$ so that $0\le s\le l^{\beta}(M_1)$, where $\Pi_s^{\beta}$ was defined in Lemma \ref{lemma:4}.

The proof of \eqref{eq:369} will be completed if we show (with $\tilde{h}_{M_1, M_2, s}^{M_1}$ defined in \eqref{eq:74}) that for every $p\in(1, \infty)$ there is $\delta_p\in(0, 1)$ such that for all $f\in\ell^p(\ZZ)$ we have 
\begin{align}
\label{eq:13}
\sup_{J\in\ZZ_+}\sup_{I\in\mathfrak S_J(\tilde{\SS}_{\tau}(j, s))}\|O_{I, J}(T_{\ZZ}[\tilde{h}_{M_1, M_2, s}^{M_1}]f: (M_1, M_2)\in\tilde{\SS}_{\tau}(j, s))\|_{\ell^p(\ZZ)}
\lesssim_{p, \tau} s2^{-\delta_p s}\|f\|_{\ell^p(\ZZ)}.
\end{align}
Using $\widetilde{\mathfrak m}_{M_1, M_2}$ from \eqref{eq:14} and \eqref{eq:39}  we may write
\begin{align}
\label{eq:40}
\tilde{h}_{M_1, M_2,s}^{M_1}(\xi)=\Phi_{\le -n_{N_s, N_s}^{v_j, \beta}(N_s)+2}^{\Sigma_{\le s}}[1, \widetilde{\mathfrak m}_{M_1, M_2}](\xi)\times
\Phi_{ \le -n_{N_s, N_s}^{v_j, \beta}(N_s)+1}^{\Sigma_{s}}[G, \Pi_s^{\beta}](\xi).
\end{align}
By Lemma \ref{lemma:4}, for sufficiently large $s\in\NN$, we have
\begin{align}
\label{eq:18}
\big\|T_{\ZZ}\big[\Phi_{ \le -n_{N_s, N_s}^{v_j, \beta}(N_s)+1}^{\Sigma_{s}}[G, \Pi_s^{\beta}]\big]f\big\|_{\ell^p(\ZZ)}
\lesssim_{p, \tau} 2^{-\delta_p s}\|f\|_{\ell^p(\ZZ)}.
\end{align}
Using factorization \eqref{eq:40} and \eqref{eq:18} it suffices to prove that
\begin{align*}
\sup_{J\in\ZZ_+}\sup_{I\in\mathfrak S_J(\tilde{\SS}_{\tau}(j,s))}
\|O_{I, J}(T_{\ZZ}\big[\Phi_{\le -n_{N_s, N_s}^{v_j, \beta}(N_s)+2}^{\Sigma_{\le s}}[1, \widetilde{\mathfrak m}_{M_1, M_2}]\big]f: (M_1, M_2)\in\tilde{\SS}_{\tau}(j,s))\|_{\ell^p(\ZZ)}\lesssim_{p, \tau} s\|f\|_{\ell^p(\ZZ)}
\end{align*}
which will readily imply \eqref{eq:13}.
\paragraph{\bf Step 2}
Appealing to the Ionescu--Wainger multiplier theory (see Theorem \ref{thm:IW2}) for oscillation semi-norms developed in the previous section  we see that
\begin{align*}
\sup_{J\in\ZZ_+}\sup_{I\in\mathfrak S_J(\tilde{\SS}_{\tau}(j,s))}\|O_{I, J}(T_{\ZZ}[\Phi_{\le -n_{N_s, N_s}^{v_j, \beta}(N_s)+2}^{\Sigma_{\le s}}[1, \mathfrak m_{M_1, M_2}]]f: (M_1, M_2)\in\tilde{\SS}_{\tau}(j,s))\|_{\ell^p(\ZZ)}\lesssim_{p, \tau} s\|f\|_{\ell^p(\ZZ)}.
\end{align*}
Hence the last inequality from the previous step will be proved if we establish
\begin{align}
\label{eq:20}
\sup_{J\in\ZZ_+}\sup_{I\in\mathfrak S_J(\tilde{\SS}_{\tau}(j,s))}\|O_{I, J}(T_{\ZZ}[\Phi_{\le -n_{N_s, N_s}^{v_j, \beta}(N_s)+2}^{\Sigma_{\le s}}[1, \mathfrak g_{M_1, M_2}]]f: (M_1, M_2)\in\tilde{\SS}_{\tau}(j,s))\|_{\ell^p(\ZZ)}\lesssim_{p, \tau} \|f\|_{\ell^p(\ZZ)},
\end{align}
with $\mathfrak g_{M_1, M_2}=\widetilde{\mathfrak m}_{M_1, M_2}-\mathfrak m_{M_1, M_2}$.
By the van der Corput estimate (Proposition \ref{thm:CCW}) for $\mathfrak m_{M_1, M_2}$ there exists $\delta_0>0$ (in fact $\delta_0\simeq (\deg P)^{-1}$) such that
\begin{align*}
|\mathfrak g_{M_1, M_2}(\xi)|=|\mathfrak m_{M_1, M_2}(\xi)(1-\eta_{\le -n_{M_1, M_2}^{v_j, \beta}(M_1)}(\xi))|
\lesssim \min\{(\log M_1)^{-\delta_0\beta}, (M_1^{v_{j, 1}}M_2^{v_{j, 2}}|\xi|)^{\pm\delta_0}  \}
\end{align*}
for $(M_1, M_2)\in\tilde{\SS}_{\tau}(j)$, since
\begin{align*}
|1-\eta_{\le -n_{M_1, M_2}^{v_j, \beta}(M_1)}(\xi)|\lesssim \min\{1, M_1^{v_{j, 1}}M_2^{v_{j, 2}}|\xi|\}.
\end{align*}
Then by Plancherel's theorem combined with  a simple interpolation and  Theorem \ref{thm:IW} we conclude that for every $p\in(1, \infty)$ there is $\alpha_p>10$ such that for every $f\in \ell^p(\ZZ)$ one has
\begin{align*}
\bigg\|\Big(\sum_{M_2\in \tilde{\SS}_{\tau}^2(j;M_1)}\big|T_{\ZZ}\big[\Phi_{\le -n_{N_s, N_s}^{v_j, \beta}(N_s)+2}^{\Sigma_{\le s}}[1, \mathfrak g_{M_1, M_2}]\big]f\big|^2\Big)^{1/2}\bigg\|_{\ell^p(\ZZ)}
\lesssim_{p, \tau} (\log M_1)^{-\alpha_p}\|f\|_{\ell^p(\ZZ)},
\end{align*}
completing  the proof of \eqref{eq:20}.
\end{proof}

\begin{proof}[Proof of Theorem \ref{thm:maint}]
We fix $1\le j< r$ as before. To prove \eqref{eq:289}, in view of \eqref{eq:369} and \eqref{eq:38}, it suffices to show that
\begin{align}
\label{eq:41}
\sum_{M_1\in \SS_{\tau}^1(j)}\big\|\sup_{M_2\in\SS_{\tau}^2(j;M_1)}|T_{\ZZ}[m_{M_1, M_2}-\tilde{h}_{M_1, M_2}^{M_1}]f|\big\|_{\ell^p(\ZZ)}\lesssim_{p, \tau}\|f\|_{\ell^p(\ZZ)}.
\end{align}
For $u\in\{p_0, p_0'\}$ by the one-parameter theory, which produces bounds independent of the coefficients of the underlying polynomials (see for instance \cite{MSZ3, MSS}), we may conclude
\begin{align}
\label{eq:51}
\sup_{M_1\in\ZZ_+}\big\|\sup_{M_2\in\ZZ_+}|T_{\ZZ}[m_{M_1, M_2}]f|\big\|_{\ell^u(\ZZ)}\lesssim_{u, \tau}\|f\|_{\ell^u(\ZZ)},
\end{align}
and by \eqref{eq:135} combined with \eqref{eq:369} we also have
\begin{align}
\label{eq:55}
\sup_{M_1\in \SS_{\tau}^1(j)}\big\|\sup_{M_2\in\SS_{\tau}^2(j;M_1)}|T_{\ZZ}[\tilde{h}_{M_1, M_2}^{M_1}]f|\big\|_{\ell^u(\ZZ)}\lesssim_{u, \tau}\|f\|_{\ell^u(\ZZ)}.
\end{align}
On the one hand, combining \eqref{eq:51} and \eqref{eq:55} we deduce that
\begin{align}
\label{eq:79}
\big\|\sup_{M_2\in\SS_{\tau}^2(j;M_1)}|T_{\ZZ}[m_{M_1, M_2}-\tilde{h}_{M_1, M_2}^{M_1}]f|\big\|_{\ell^u(\ZZ)}\lesssim_{u, \tau}\|f\|_{\ell^u(\ZZ)}.
\end{align}
On the other hand, inequalities \eqref{eq:54}, \eqref{eq:315} and \eqref{eq:360} imply for every $M_1\in \SS_{\tau}^1(j)$ that
\begin{align}
\label{eq:65}
\big\|\sup_{M_2\in\SS_{\tau}^2(j;M_1)}|T_{\ZZ}[m_{M_1, M_2}-\tilde{h}_{M_1, M_2}^{M_1}]f|\big\|_{\ell^2(\ZZ)}\lesssim_{\tau}(\log M_1)^{-\alpha}\|f\|_{\ell^2(\ZZ)}
\end{align}
with the parameter $\alpha>0$ as in \eqref{eq:299}. Simple interpolation between \eqref{eq:79} and \eqref{eq:65} yields \eqref{eq:41} and this completes the proof of Theorem \ref{thm:maint}.
\end{proof}

\end{document}